\documentclass[11pt]{amsart}
\usepackage[pdftex]{graphicx}
\usepackage{amssymb}
\usepackage{amsmath}
\usepackage{amsfonts}
\usepackage{amsthm}
\usepackage{amssymb}
\usepackage{mathrsfs}
\usepackage{enumitem}
\usepackage{xcolor}
\newtheorem{theorem}{Theorem}[section]
\newtheorem{proposition}[theorem]{Proposition}
\newtheorem{corollary}[theorem]{Corollary}
\newtheorem{oq}{Open Question}
\newtheorem{lemma}[theorem]{Lemma}

\theoremstyle{definition}
\newtheorem{definition}[theorem]{Definition}
\theoremstyle{remark}
\newtheorem{remark}[theorem]{Remark}

\newtheorem{example}[theorem]{Example}

\renewcommand{\leq}{\leqslant}
\renewcommand{\geq}{\geqslant}

\newcommand{\Sp}{\mathbb{S}^2}

\newcommand{\mH}{\mathcal H}

\newcommand{\mC}{\mathcal C}
\newcommand{\mL}{\mathcal L}

\newcommand{\la}{\langle}
\newcommand{\ra}{\rangle}
\newcommand{\diam}{\mathrm{diam}}
\newcommand{\dist}{\mathrm{dist}}

\newcommand{\supp}{\mathrm{supp}}
\newcommand{\eps}{\varepsilon}

\newcommand{\Met}{\mathrm{Met}}
\newcommand{\Diff}{\mathrm{Diff}}
\newcommand{\can}{\mathrm{can}}
\newcommand{\inj}{\mathrm{inj}}

\newcommand{\Conf}{\mathrm{Conf}}
\newcommand{\length}{\mathrm{Length}}
\newcommand{\Wdot}{\dot{W}^{-1,2}}
\newcommand{\Lip}{\mathrm{Lip}}
\newcommand{\GrB}{\mathcal{G}_2(n+1)}

\DeclareMathOperator{\area}{\mathrm{Area}}
\DeclareMathOperator{\ind}{\mathrm{ind}}

\numberwithin{equation}{section}

\begin{document}

\title[Title]{From Steklov to Laplace: free boundary minimal surfaces with many boundary components}

\author[M. Karpukhin]{Mikhail Karpukhin}
\address[Mikhail Karpukhin]{Mathematics 253-37, California Institute of Technology, 
Pasadena, CA 91125, USA}
\email{mikhailk@caltech.edu}
\author[D. Stern]{Daniel Stern}
\address[Daniel Stern]{Department of Mathematics, University of Chicago,
5734 S University Ave
Chicago IL, 60637, USA}
\email{dstern@uchicago.edu}
\date{}
\maketitle
\begin{abstract}
In the present paper, we study sharp isoperimetric inequalities for the first Steklov eigenvalue $\sigma_1$ on surfaces with fixed genus and large number $k$ of boundary components. We show that as $k\to \infty$ the free boundary minimal surfaces in the unit ball arising from the maximization of $\sigma_1$ converge to a closed minimal surface in the boundary sphere arising from the maximization of the first Laplace eigenvalue on the corresponding closed surface. For some genera, we prove that the corresponding areas converge at the optimal rate $\frac{\log k}{k}$. This result appears to provide the first examples of free boundary minimal surfaces in a compact domain converging to closed minimal surfaces in the boundary, suggesting new directions in the study of free boundary minimal surfaces, with many open questions proposed in the present paper. A similar phenomenon is observed for free boundary harmonic maps associated to conformally-constrained shape optimization problems.

\end{abstract}


\section{Introduction}

\subsection{Background} Since the 18th century, minimal surfaces have played a central role in geometry and analysis, among other areas of mathematics and physics. While early investigations focused on minimal surfaces in Euclidean space, the twentieth century saw an increased interest in the study of minimal surfaces and higher-dimensional minimal submanifolds in compact Riemannian manifolds, with a fundamental special case being the study of minimal submanifolds in the sphere $\mathbb{S}^n$. Indeed, in addition to their intrinsic geometric interest, minimal submanifolds in $\mathbb{S}^n$ are an unavoidable object of study for those investigating analytic aspects of minimal submanifolds, since cones over minimal varieties in the sphere generate the blow-up models for singularities of minimal submanifolds in any ambient space.

In recent decades, the study of minimal submanifolds in spheres has been greatly enriched by the discovery of an intimate link between minimal surfaces in spheres and certain natural shape optimization problems for Laplacian eigenvalues. On a closed Riemannian surface $(M,g)$, denote by 
$$
0=\lambda_0(M,g)<\lambda_1(M,g)\leqslant \lambda_2(M,g)\leqslant\ldots \nearrow+\infty
$$
the spectrum of the \emph{positive} Laplacian $\Delta_g=\delta_gd$. Normalizing by the area, one obtains a sequence of scale-invariant quantities, of which the most fundamental is the first non-trivial normalized eigenvalue
$$
\bar{\lambda}_1(M,g):=\lambda_1(M,g)\area_g(M).
$$
About fifty years ago, Hersch observed in the influential paper \cite{Hersch} that $\bar{\lambda}_1(\mathbb{S}^2,g)\leq 8\pi$ for any metric $g$ on $\mathbb{S}^2$, with equality only for round metrics. This paved the way for the study of the maximization problem for $\bar{\lambda}_1(M,g)$ over metrics on a surface of fixed topological type, and the associated maxima
$$\Lambda_1(M):=\sup_{g\in \Met(M)}\bar{\lambda}_1(M,g).$$
Early key contributions were made by Yang-Yau \cite{YY}, who showed that $\Lambda_1(M)<\infty$ for orientable $M$  (see \cite{KarNonOrientable} for the non-orientable case), and Li-Yau \cite{LY}, whose introduction of the conformal volume led to the characterization of the round metric as the unique $\bar{\lambda}_1$-maximizing metric on $\mathbb{RP}^2$, among other important consequences.

In the '90s, a significant breakthrough was made by Nadirashvili \cite{Nad}, who realized that metrics maximizing the normalized Laplacian eigenvalues $\bar{\lambda}_i(M,g)$ are induced by minimal immersions to the unit sphere--an observation which he used to confirm Berger's conjecture that the flat equilateral metric on $\mathbb{T}^2$ maximizes $\bar{\lambda}_1$. By the recent work of \cite{MS, PetL1}, it is now known that for any closed surface $M$, the supremum $\Lambda_1(M)$ is achieved by a $\bar{\lambda}_1$-maximizing metric--possibly with conical singularities--induced by a branched minimal immersion $M\to \mathbb{S}^n$ of area $\frac{1}{2}\Lambda_1(M)$, though at present the explicit maximizer has only been determined for $M=\mathbb{S}^2,\mathbb{RP}^2,\mathbb{T}^2$, the Klein bottle $\mathbb{K}$ \cite{EGJ, JNP, CKM}, and the orientable surface of genus two \cite{JLNNP, NaySh, Ros}. Let us also remark that the corresponding theory for higher eigenvalues has seen a lot of recent progress; see \cite{PetLk, KNPP, KNPP2, KarRP2}. Moreover, it has been observed \cite{ESI} that the induced metric on any minimally immersed closed submanifold in $\mathbb{S}^n$ is a (typically non-maximizing) critical point for one of the functionals $\bar{\lambda}_i$, so in principle all immersed minimal submanifolds in the sphere may arise from variational methods for $\bar{\lambda}_i$.

In the setting of surfaces with boundary $(N,g)$, one finds a strong analogy between closed minimal surfaces in the sphere and \emph{free boundary minimal surfaces}--critical points for the area functional among relative $2$-cycles--in Euclidean balls. In recent years, much activity in this direction has been stimulated by the work of Fraser and Schoen \cite{FS}, who demonstrated that metrics maximizing normalized \emph{Steklov} eigenvalues 
$$
\bar{\sigma}_i(N,g)=\sigma_i(N,g)\length(N,g),
$$
where $\sigma_i(N,g)$ are eigenvalues of the Dirichlet-to-Neumann map $C^{\infty}(\partial N)\to C^{\infty}(\partial N)$, are induced by free boundary minimal immersions to Euclidean balls. In addition to solving several important problems related to optimal bounds for Steklov eigenvalues, their work reinvigorated the study of free boundary minimal submanifolds in general; see the survey \cite{Li_survey} for a discussion of many results in this direction obtained over the last decade. For the purposes of this paper, let us note that the supremum
$$\Sigma_1(N):=\sup_g\bar{\sigma}_1(N,g)$$
of the first nontrivial normalized Steklov eigenvalue over metrics on a surface with boundary $N$ is finite \cite{GP:HPS, FS:conf_volume, Medvedev}, and it was recently proved in \cite{MP} that the supremum $\Sigma_1(N)$ is always achieved by a $\bar{\sigma}_1$-maximizing metric (possibly with conical singularities) on $N$, induced by a branched free boundary minimal immersion of area $\frac{1}{2}\Sigma_1(N)$ in a Euclidean unit ball of suitable dimension.

In the present paper, we establish an explicit link--beyond the well-known analogy--between the free boundary minimal surfaces in Euclidean balls arising from maximization of $\bar{\sigma}_1$ and the closed minimal surfaces in $\mathbb{S}^n$ arising from maximization of $\bar{\lambda}_1$. Namely, building on the recent results of \cite{GL,GKL,KS,KNPS}, we show that the free boundary minimal surfaces of genus $\gamma$ and $k$ boundary components in $\mathbb{B}^{n+1}$ arising from maximization of $\bar{\sigma}_1$ converge as $k\to\infty$, in the varifold sense, to the closed minimal surface of genus $\gamma$ in $\mathbb{S}^n$ arising from maximization of $\bar{\lambda}_1$ (and similarly in the nonorientable setting), and obtain sharp estimates for the rate at which their areas $\frac{1}{2}\Sigma_1(N_k)$ converge to $\frac{1}{2}\Lambda_1(M)$. 

In particular, while the explicit maximizers for $\bar{\sigma}_1$ are known only for the disk, the annulus, and the M\"obius band \cite{FS}, our results provide an asymptotic description of the $\bar{\sigma}_1$-maximizing metrics on surfaces with many boundary components in every case where the maximizing metric for $\bar{\lambda}_1$ is known on the corresponding closed surface (namely, at the moment, for $M=\mathbb{S}^2,\mathbb{RP}^2,\mathbb{T}^2,\mathbb{K}$, or the orientable surface of genus two). Moreover, to our knowledge, these results provide the first examples of families of compact free boundary minimal surfaces in a manifold with boundary limiting to closed minimal surfaces in the boundary, suggesting a number of new questions and lines of investigation in the study of free boundary minimal surfaces (see Section \ref{sec:disc_intro} below).

\subsection{Convergence of $\bar\sigma_1$-maximizing maps} Given a closed surface $M$, let $N_k$ be the compact surface with boundary obtained by removing $k$ disjoint disks from $M$. To be precise, $N_k$ is orientable if and only if $M$ is orientable, $N_k$ has $k$ boundary components and the same genus as $M$. The following surprising identities were recently established in~\cite{GL,KS},
\begin{equation}
\label{ineq:KS_intro}
\Sigma_1(N_k)<\Lambda_1(M);
\end{equation}
\begin{equation}
\label{lim:GL_intro}
\lim_{k\to\infty}\Sigma_1(N_k) = \Lambda_1(M);
\end{equation}
see also~\cite{GKL} for a more streamlined proof of~\eqref{lim:GL_intro}.
In particular, since $\bar\sigma_1$-maximal metrics correspond to free boundary minimal surfaces of bounded area $\frac{1}{2}\Sigma_1(N_k)<\frac{1}{2}\Lambda_1(M)<\infty$ in a ball $\mathbb{B}^{n+1}$ of fixed dimension $n$, it follows that these minimal surfaces must converge subsequentially in the varifold sense (see Section \ref{sec:convergence} below for relevant definitions) to some limit varifold in $\mathbb{B}^{n+1}$ satisfying a weak version of the free boundary stationary condition. By \eqref{lim:GL_intro}, we see that this limit varifold must have area $\frac{1}{2}\Lambda_1(M)$, and since closed minimal surfaces in $\mathbb{S}^n$ satisfy the weak definition of free boundary stationary varifolds in $\mathbb{B}^{n+1}$, it is natural to expect that these free boundary minimal surfaces converge as $k\to\infty$ to the closed minimal surface in $\mathbb{S}^n$ realizing $\Lambda_1(M)$. 
For $M=\Sp$ this was posed as a conjecture in~\cite{GL}. In the first result of the present paper we resolve this conjecture for an arbitrary closed surface $M$ and prove the following theorem. (See Section~\ref{sec:convergence} and Theorem~\ref{thm:convergence} for a more detailed statement.)

\begin{theorem}
\label{thm:convergence_intro}
For an appropriate $n=n(M)\in \mathbb{N}$, up to a choice of a subsequence, the (branched) free boundary minimal surfaces in $\mathbb{B}^{n+1}$ inducing the $\bar\sigma_1(N_k)$-maximizing metrics on $N_k$ converge in the varifold sense to a closed (branched) minimal surface in $\mathbb{S}^n$ inducing the $\bar\lambda_1(M)$-maximizing metric on $M$. Moreover, as a consequence, their supports converge in the Hausdorff distance, and the boundary measures converge to twice the area measure of the limit surface. 
\end{theorem}

\begin{remark}
As is customary in the theory of varifolds, the surfaces in Theorem~\ref{thm:convergence_intro} should be understood with the appropriate multiplicity; see Section~\ref{sec:convergence}. For example, if $M=M_2$ is the orientable surface of genus $2$, then the limiting surface is $\Sp$ with multiplicity $2$; see Open Question~\ref{oq:Bolza} below. 
\end{remark}

\begin{remark} Note that one should not expect to improve the convergence statement far beyond varifold convergence: in particular, note that a free boundary surface in $\mathbb{B}^{n+1}$ cannot be $C^1$ close to any surface in $\mathbb{S}^n$ near its boundary, and it is clear from direct examination of the minimal surface equations in $\mathbb{R}^{n+1}$ and $\mathbb{S}^n$ that a minimal surface in $\mathbb{B}^{n+1}$ is nowhere close to a minimal surface in $\mathbb{S}^n$ in a $C^2$ sense. Nonetheless, one can of course ask for a more refined picture of the convergence given in Theorem \ref{thm:convergence_intro}; see Section \ref{sec:disc_intro} below for some open questions in this direction.
\end{remark}

If $M$ is a sphere $\Sp$, a projective plane $\mathbb{RP}^2$, a torus $\mathbb{T}^2$ or a Klein bottle $\mathbb{K}$,
then the branched minimal surface corresponding to $\bar\lambda_1(M)$-maximal metric is unique up to an isometry of $\mathbb{R}^{n+1}$, see e.g.~\cite{CKM}. Since $O(n+1)$ is compact, the convergence of Theorem~\ref{thm:convergence_intro} holds along the full sequence after applying a suitable element of $O(n+1)$ to each member of the sequence. Moreover, in all these examples, the limit surface and the value of $n$ are known explicitly.

\begin{itemize}
\item If $M=\Sp$, then $n=2$ and the limit surface is the whole sphere. Furthermore, Proposition 8.1 of~\cite{FS} implies that all of these
free boundary minimal surfaces are embedded. In particular, we have the following corollary.
\end{itemize}
\begin{corollary}
\label{cor:sphere_intro}
For each $k>0$ there exists an embedded free boundary minimal surface in $\mathbb{B}^3$ of genus $0$ with $k$ boundary components, such that as $k\to\infty$ these surfaces converge in the varifold sense to the boundary sphere $\mathbb{S}^2=\partial \mathbb{B}^3$.
\end{corollary}

\begin{remark}
Note that Corollary~\ref{cor:sphere_intro} and the relation~\eqref{lim:GL_intro} are in contradiction with~\cite[Theorem 1.6]{FS}. We refer to the appendix in~\cite{GL} for the explanation.
\end{remark}

\begin{remark}
Approximate pictures of these free boundary minimal surfaces are obtained in~\cite{GL, KOO} using numerical computations.
\end{remark} 
\begin{itemize}
\item If $M=\mathbb{RP}^2$, then $n=4$ and the limit surface is the {\em Veronese surface}. 

\item If $M=\mathbb{K}$, then $n=4$ and the limit surface is one of the minimal Klein bottles constructed by Lawson in~\cite{Lawson}, in his notation $\tilde\tau_{3,1}$. Note that the same surface is conjectured to be the Klein bottle with the smallest Willmore energy. 

\item If $M=\mathbb{T}^2$, then $n=5$ and the limit surface is the so-called ``Bryant-Itoh-Montiel-Ros" torus. It is characterized by the fact that the induced metric is the flat metric corresponding to the equilateral lattice. 
\end{itemize}

We remark that, using arguments similar to those in~\cite[Proposition 8.1]{FS}, it is possible to show that the free boundary immersions corresponding to $M=\mathbb{RP}^2,\mathbb{K},\mathbb{T}^2$ are unbranched as soon as they are linearly full, i.e. not contained in a proper linear subspace of $\mathbb{R}^{n+1}$. However, embeddedness seems to be a more subtle issue; see Open Question~\ref{oq:embedded}.  

The phenomenon described in Theorem~\ref{thm:convergence_intro} is not unique to minimal surfaces, but also occurs naturally in the setting of non-conformal harmonic maps. Indeed, we show that there are many examples exhibiting the same behaviour in the harmonic map setting (see Section~\ref{sec:convergence_conf} for a more precise statement).
\begin{proposition}\label{prop:conf_cvg_intro}
For every conformal class $[g]$ on $M$, there exist conformal classes $[g_k]$ on $N_k$ induced by an inclusion $N_k\rightarrow (M,[g])$ and nonconstant free boundary harmonic maps $u_k\colon (N_k,\partial N_k)\to (\mathbb{B}^{n+1},\mathbb{S}^n)$ whose harmonic extensions in $M$ converge strongly in $W^{1,2}(M,g)$ to a nonconstant harmonic map $u: (M,[g])\to \mathbb{S}^n=\partial \mathbb{B}^{n+1}$.
\end{proposition}

\subsection{Refined asymptotics for $\Sigma_1(N_k)$.}

In our second result we provide a sharp rate of convergence for the limit~\eqref{lim:GL_intro}, namely, we identify $\frac{\log k}{k}$ as the decay rate of the correction term.
\begin{theorem}
\label{thm:asymptotics_intro}
Let $N_k$ be a compact surface with boundary obtained by removing $k$ disjoint disks from a closed surface $M$. Then the following holds.
\begin{enumerate}
\item There exists a constant $C=C(M)>0$ such that for all $k>0$ one has
\begin{equation}
\label{ineq:lowerbound_intro}
\Sigma_1(N_k)\geqslant \Lambda_1(M) - C\frac{\log k}{k}.
\end{equation}
\item Let $M$ be a sphere $\Sp$, a projective plane $\mathbb{RP}^2$, a torus $\mathbb{T}^2$ or a Klein bottle $\mathbb{K}$. Then there exists a constant $c=c(M)>0$ such that for all $k>0$ one has
\begin{equation}
\label{ineq:upperbound_intro}
\Sigma_1(N_k)\leqslant \Lambda_1(M) - c\frac{\log k}{k}.
\end{equation}
\end{enumerate}
\end{theorem}  
\begin{remark} In terms of the free boundary minimal immersions $u_k: N_k\to \mathbb{B}^{n+1}$ realizing $\Sigma_1(N_k)$ and the minimal immersion $u:M\to \mathbb{S}^n$ realizing $\Lambda_1(M)$, this tells us that the areas satisfy
\begin{equation}
Area(u(M))-Area(u_k(N_k))\leq C(M)\frac{\log k}{k}
\end{equation}
for any $M$, and this convergence rate is sharp for $M=\mathbb{S}^2, \mathbb{RP}^2, \mathbb{T}^2$, or $\mathbb{K}$.
\end{remark}

\begin{remark}
\label{rmk:quant_stab_intro}
The proof of~\eqref{ineq:upperbound_intro} is based on a refinement of the quantitative stability of $\bar\lambda_1$-maximal metrics as defined in~\cite{KNPS}, and the surfaces listed in the assumptions are precisely those for which quantitative stability is verified in~\cite{KNPS}. However, it is interesting to note that the results of~\cite{KNPS} alone do not suffice to establish the sharp bound~\eqref{ineq:upperbound_intro}; see the discussion in Section~\ref{ref.stab}.
\end{remark}
\begin{remark}
Inequality~\eqref{ineq:upperbound_intro} is a quantitative improvement over~\eqref{ineq:KS_intro}. The only other known result of this type is~\cite[Theorem 1.8]{GKL}, where the correction term decays exponentially with $k$. We note that a variant of \eqref{ineq:lowerbound_intro} also holds for the conformally-constrained maximization problem (see Proposition \ref{construct}), and the corresponding variant of the upper bound \eqref{ineq:upperbound_intro} holds for many non-maximizing conformal classes--e.g., for any conformal class admitting a minimal immersion to $\mathbb{S}^n$ by first eigenfunctions (see Remark \ref{conf.ubds}). 
\end{remark}

\subsection{Ideas of the proofs}

To prove Theorem \ref{thm:convergence_intro}, we begin by applying uniformization results of~\cite{Maskit} and~\cite{Haas} to identify $\bar{\sigma}_1$-maximizing metrics $(N_k,\tilde{g}_k)$ on $N_k$ conformally with a domain $\Omega_k\subset (M,g_k)$ given by removing disjoint geodesic disks from a constant curvature metric $g_k$ on $M$. 
Combining \eqref{lim:GL_intro} with the stability results of \cite{KNPS}, we are able to deduce that the the conformal classes $[g_k]$ converge subsequentially to $[g]$, and the boundary length measures $ds_{\tilde g_k}$ of $\partial\Omega_k$ converge in $W^{-1,2}(M,g)$ to the area measure $dv_{g_{\max}}$ of a $\bar{\lambda}_1$-maximizing metric $g_{\max}\in [g]$ on $M$.

We then show that there exists a metric $\tilde g$ on $M$ with respect to which the harmonic extension $\hat{u}_k\colon M\to \mathbb{B}^{n+1}$ of the branched free boundary minimal immersions $u_k\colon (N_k,\tilde{g}_k)\to \mathbb{B}^{n+1}$ by $\sigma_1(N_k,\tilde{g}_k)$-eigenfunctions have vanishing energy in the complement $M\setminus\Omega_k$, and use the strong $W^{-1,2}$ convergence $ds_{\tilde g_k}\to dv_{g_{\max}}$ to deduce that the maps $\hat{u}_k$ converge strongly (subsequentially) in $W^{1,2}(M,\tilde g)$ to a minimal immersion $(M,g_{\max})\to \mathbb{S}^n$ by first eigenfunctions of $\Delta_{g_{\max}}$. The convergence of the associated varifolds then follows by standard arguments from the strong convergence $\hat{u}_k\to u$ and the vanishing of the energy $\int_{M\setminus\Omega_k}|d\hat{u}_k|^2_{g_k}\to 0.$ The proof of Proposition \ref{prop:conf_cvg_intro} follows similar lines. 

The proof of the lower bound \eqref{ineq:lowerbound_intro} in Theorem \ref{thm:asymptotics_intro} is constructive--namely, we produce a metric on $N_k$ satifying $\bar{\sigma}_1\geq \Lambda_1(M)-C\frac{\log k}{k}$. As in \cite{GL}, we begin by removing several small geodesic disks--of radius $k^{-\alpha}$ for $\alpha$ sufficiently large--with respect to a constant curvature metric conformal to a $\bar{\lambda}_1$-maximizing metric $g_{\max}$, to produce a domain $\Omega_k\subset M$ diffeomorphic to $N_k$. We then choose a conformal metric $\tilde{g}_k$ on this domain with the property that the pairing $\langle \mu_k,\varphi\rangle$ of the resulting length measure $\mu_k=ds_{\tilde{g}_k}$ of $\partial\Omega_k$ with a smooth function $\varphi\in C^{\infty}(\partial\Omega_k)$ is given by the integral $\int_{\Omega_k}\hat{\varphi}dv_{g_{\max}}$ of the harmonic extension $\hat{\varphi}$ over $\Omega_k$ with respect to $g_{\max}$. 

To show that the resulting metric $(\Omega_k,\tilde{g}_k)$ satisfies $\bar{\sigma}_1(\Omega_k,\tilde{g}_k)\geq \Lambda_1(M)-C\frac{\log k}{k}$, we first argue that the restriction to $\partial\Omega_k$ of the first eigenfunctions for $\Delta_{g_{\max}}$ are--in an appropriate sense--approximate eigenfunctions, i.e. {\em quasimodes}, of the Dirichlet-to-Neumann map for $(\Omega_k,\tilde{g}_k)$, with the normalized eigenvalue lying in $\left[\Lambda_1-C\frac{\log k}{k},\Lambda_1+C\frac{\log k}{k}\right]$. Denoting by $m$ the multiplicity of the first eigenvalue of $\Delta_{g_{\max}}$, we then deduce that there must exist at least $m$ Steklov eigenvalues in $\left[\Lambda_1-C\frac{\log k}{k},\Lambda_1+C\frac{\log k}{k}\right]$, and employ a contradiction argument to conclude that the \emph{first} normalized Steklov eigenvalue $\bar{\sigma}_1(\Omega,\tilde{g}_k)$ must lie in $\left[\Lambda_1-C\frac{\log k}{k},\Lambda_1+C\frac{\log k}{k}\right]$, as desired.

To prove the upper bound \eqref{ineq:upperbound_intro} in Theorem \ref{thm:asymptotics_intro}, we need to show that if $M=\mathbb{S}^2,\mathbb{RP}^2,\mathbb{T}^2$, or $\mathbb{K}$, then \emph{every} metric $g$ on $N_k$ must satisfy $\bar{\sigma}_1(N_k,g)\leq \Lambda_1(M)-c\frac{\log k}{k}$. To this end, we again begin by identifying a given metric $(N_k,g)$ conformally with the complement of geodesic disks for some constant curvature metric on $M$. Building on the techniques of \cite{KNPS}, we then show that for any such domain $\Omega_k$ with conformal metric $\tilde{g}_k$, there exists a $\bar{\lambda}_1$-maximizing metric $g_{\max}$ on $M$ such that the gap $\Lambda_1(M)-\bar{\sigma}_1(\Omega,\tilde{g}_k)$ is bounded below by $\area_{g_{\max}}(M\setminus\Omega)$ and the square of the $W^{-1,2}(M,g_{\max})$-distance between the length measure $ds_{\tilde g_k}$ of $\partial\Omega$ and an appropriate multiple of the area measure $dv_{g_{\max}}$. The area bound $\area_{g_{\max}}(M\setminus\Omega)\leq C \left(\Lambda_1(M)-\bar{\sigma}_1(\Omega,\tilde g_k)\right)$ is then used to show that a certain test function $\varphi_k$ (related to the logarithm of the distance to the centers of the disks comprising the complement $M\setminus\Omega$) satisfies
$$
c\frac{\log k}{k}\leq \frac{\langle \varphi_k,ds_{\tilde g_k}-dv_{g_{\max}}\rangle}{\|\varphi_k\|_{W^{1,2}(M,g_{\max})}}\leq \|ds_{\tilde g_k}-dv_{g_{\max}}\|_{W^{-1,2}(M,g_{\max})},
$$
from which the desired bound follows.

\subsection{Discussion and Open Questions}
\label{sec:disc_intro}

Item (2) of Theorem~\ref{thm:asymptotics_intro} immediately begs the following question
\begin{oq}
Does the inequality~\eqref{ineq:upperbound_intro} hold for all closed surfaces $M$?
\end{oq}
One of the ways to resolve this question would be to prove an appropriate quantitative stability result relating the difference $\Lambda_1(M)-\bar{\sigma}_1(\Omega,\tilde{g})$ for a domain $\Omega\subset M$ to the $W^{-1,2}(M,g_{\max})$ difference between the measures $dv_{g_{\max}}$ and $ds_{\tilde g}$ and the area of $M\setminus \Omega$ with respect to some $\bar{\lambda}_1$-maximizing metric $g_{\max}$ on $M$. It could be illuminating to investigate this problem first for surfaces of genus $2$, where the $\bar{\lambda}_1$-maximizing metrics are known, but do not meet the criteria needed to apply our methods of proof for ~\eqref{ineq:upperbound_intro}. 

It is also natural to ask to what extent the estimates of Theorem ~\ref{thm:asymptotics_intro} can be sharpened. In this direction, the following question is an obvious place to begin.
\begin{oq}
\label{oq:limit}
Does the limit 
$$
\lim_{k\to\infty}(\Lambda_1(M) - \Sigma_1(N_k))\frac{k}{\log k}
$$ 
exist? If so, then find its value.
\end{oq}
An explicit answer to this question will likely go hand-in-hand with a sharper geometric picture of the associated $\bar{\sigma}_1$-maximizing metrics, see Question \ref{oq:Omega_k} below. 

There are many natural questions concerning the limiting behavior of the free boundary minimal surfaces realizing $\Sigma_1(N_k)$. From the perspective of geometric measure theory, one of the first questions one might pose concerns the \emph{persistence of singularities} of these surfaces in the limit as $k\to\infty$.

\begin{oq}
\label{oq:embedded}
 If the limiting minimal surface in $\mathbb{S}^n$ realizing $\Lambda_1(M)$ is embedded, does it necessarily follow that the free boundary minimal surfaces in $\mathbb{B}^{n+1}$ realizing $\Sigma_1(N_k)$ are embedded for $k$ sufficiently large? 
\end{oq}

\begin{remark} Note that the standard persistence-of-singularities result for stationary varifolds in a fixed domain does not hold for families of free boundary stationary varifolds in $\mathbb{B}^{n+1}$ approaching a stationary varifold in $\mathbb{S}^n$: for an elementary counterexample, note that the boundary of an inscribed regular $k$-gon in the $2$-dimensional unit disk $\mathbb{B}^2$ gives a singular free boundary stationary geodesic network, which approaches the (smooth, multiplicity one) boundary circle as $k\to \infty$. However, it is straightforward to check that the embededdness of the limit surface in $\mathbb{S}^n$--by Allard regularity and standard monotonicity results--rules out the possibility of singularities with density larger than $2$ in nearby free boundary minimal surfaces in $\mathbb{B}^{n+1}$; moreover, these free boundary minimal surfaces must look roughly conical at all small scales (though perhaps with different cones at different scales) near a singularity of density equal to $2$, so the conditions under which singularities could disappear in the limit appear to be quite restrictive.
\end{remark}

The following question is inspired by Corollary~\ref{cor:sphere_intro} and concerns a finer structure of free boundary minimal surfaces corresponding to $\bar\sigma_1(N_k)$-maximal metrics. We formulate the question for $M=\Sp$, but, of course, similar problems can be posed for other closed surfaces. 

\begin{oq}
\label{oq:Omega_k}
Let $\Omega_k\subset \mathbb{B}^3$ be an embedded free boundary minimal surface of genus $0$ with $k$ boundary components corresponding to a $\bar\sigma_1$-maximizing metric. Prove or disprove the following.
\begin{enumerate}
\item  $\Omega_k$ is unique up to isometries of $\mathbb{B}^3$.
\item All boundary components are approximately of the same size. In particular, let $L_{j,k}$, $j=1,\ldots, k$ be the lengths of boundary components of $\Omega_k$, then there exist $c,C>0$ such that for all $j,k$
$$
\frac{c}{k}\leqslant  L_{j,k}\leqslant \frac{C}{k}.
$$
More precisely, show that 
$$
\lim_{k\to\infty}\sup_{1\leqslant j\leqslant k}\frac{L_{j,k}}{\sum_{j=1}^jL_{j,k}} = \frac{1}{k}.
$$
\item The boundary $\partial\Omega_k$ is dense on the scale $\frac{1}{\sqrt{k}}$ inside $\Sp$, i.e. there exists $C>0$ such that the $\frac{C}{\sqrt{k}}$-tubular neighborhood of $\partial\Omega_k$ contains $\Sp$. 
\item For large $k$, each boundary component is close to a half-catenoid, i.e. the blow-ups of boundary components on the scale $\frac{1}{k}$ converge to the the unique rotationally symmetric free boundary minimal surface in the half-space.
\end{enumerate}
\end{oq}

The numerical examples of~\cite{GL,KOO} point to the fact that the topology of $\Omega_k$ alone does not guarantee uniqueness for free boundary minimal surfaces in $\mathbb{B}^3$. For example, numerical computations of~\cite{GL} suggest that for $k=8,20$ there exist free boundary minimal surfaces of genus $0$ and $k$-boundary components with the symmetry group of cube and dodecahedron respectively. At the same time, the computations in~\cite{KOO} indicate that these surfaces are {\em not} Steklov maximizers and that, instead, the  boundary components of the maximizers are distributed more irregularly. This resulted in the observation in~\cite[Section 5]{KOO} that centers of mass of boundary components form a solution to a point distribution problem--in particular, the Thompson problem was suggested as a candidate. While the sample size in~\cite{KOO} is too small to formulate an exact open question, the possibility is too tantalizing to ignore. Thus, the following question is purposefully open-ended.
\begin{oq}
In the notation of Question~\ref{oq:Omega_k}, show that the centers of mass of the boundary components of $\Omega_k$ are located according to a solution of some $k$-point distribution problem of $\Sp$.
\end{oq}

Another special case which merits further study is the genus 2 setting. When $M$ is the orientable surface of genus $2$, then $\Lambda_1(M) = 16\pi$, there is a continuous family of $\bar\lambda_1(M)$-maximal metrics, and the corresponding branched minimal immersions are simply branched covers of $\Sp$ with the location of branch points varying within the family, see~\cite{JLNNP, NaySh, Ros}. The most symmetric member of the family is the co-called {\em Bolza surface} with branch points at the vertices of an octahedron. At the same time, the fact that the limiting map is a cover of $\Sp$ does not mean that $n=2$ in Theorem~\ref{thm:convergence_intro}, although it seems reasonable to suggest that the immersion corresponding to $N_{2k}$ could be a double branched cover of $\Omega_k$ defined in Question~\ref{oq:Omega_k}, at least for large $k$. If the answer to the latter question is positive, it would be interesting to understand the location of the branch points, even though it is likely such covers are not unique, similarly to the closed case. We collect these thoughts below.
 \begin{oq}
 \label{oq:Bolza}
 Let $M$ be an orientable surface of genus $2$ and let $N_k$ be a surface with boundary obtained by removing $k$ disjoint disks from $M$. Let $u_k$ be a branched free boundary immersion corresponding to a $\bar\sigma_1(N_k)$-maximizing metric. 
 \begin{enumerate}
 \item Is the map $u_k$ unique, up to isometries of $\mathbb{R}^{n+1}$? What are the possible limits of $u_k$, e.g. is it true that the Bolza cover is the only accumulation point of $\{u_k\}$?
 \item Is it true that for large enough $k$ the image of $u_k$ is contained in $\mathbb{B}^3$?
 \item More specifically, is it true that for large enough $k$ the maps $u_{2k}$ are branched covers  over surfaces $\Omega_k$ defined in Question~\ref{oq:Omega_k}? If so, then what are the locations of branch points?
 \item Similarly, is there any relation between $u_{2k+1}$ and the surfaces $\Omega_j$?
 \end{enumerate}
 \end{oq}

Finally, let us close by posing a question which should be of general interest to the minimal surface community, independent of any connections to spectral geometry. 
 
\begin{oq} Given a smooth, convex domain $P\subset \mathbb{R}^{n+1}$ and a minimal submanifold $M\subset \partial P$ in $\partial P$, does there exist a family of free boundary minimal surfaces in $P$ approaching $M$ in a varifold sense? As a special case, do there exist free boundary minimal hypersurfaces in $P$ approaching $\partial P$ in the varifold sense? (Note that Corollary \ref{cor:sphere_intro} gives a positive answer in the case $P=\mathbb{B}^3$.)
\end{oq}

\begin{remark} Na\"{i}vely, one might hope to approach this via novel gluing methods, or perhaps some variational scheme. One can also formulate an analogous question for harmonic maps, modeled on the conclusion of Proposition \ref{prop:conf_cvg_intro}.
\end{remark}

\subsection*{Acknowledgements} The research of M.K. is supported by the NSF grant DMS-2104254. The research of D.S. is supported by the NSF fellowship DMS-2002055.

\section{Preliminaries}

\subsection{Uniformization theorems for surfaces}

Recall the notation used in the introduction, where $N_k$ is a compact surface obtained by removing $k$ disjoint disks from a closed surface $M$.

Uniformization theorems are concerned with choosing a canonical metric in each conformal class of metrics. For example, the classical uniformization theorem states that given a conformal class $\mC$ on $M$ there exists a  unit area metric $g\in\mC$ of constant Gauss curvature. Furthermore, if $M\ne\Sp$, then such a metric is unique, whereas on $\Sp$ is is unique up to a conformal automorphism. We denote by $\Met_\can(M)$ the space of metrics of unit area and constant Gauss curvature on $M$.

The most commonly used uniformization theorem for manifolds with boundary states that for any conformal class $\mC$ on $N$ there is a unit area metric $g\in \mC$ with constant Gauss curvature and geodesic boundary (see, e.g., \cite{OPS}). In the present paper we use another, perhaps lesser-known, uniformization result.

\begin{theorem}[\cite{Maskit, Haas}]
\label{thm:uni}
Let $(N_k,\mC)$ be a compact surface with $k$ boundary components endowed with a conformal class $\mC$. Then there exists a closed Riemannian surface $(M,g)$ of unit area and constant curvature, a collection $B_i\subset M$, $i=1,\ldots, k$ of embedded open non-empty geodesic disks with disjoint closure, and a conformal diffeomorphism $F\colon (N,\mC)\to (\Omega_k,g)$, where $\Omega_k = M\setminus \bigcup_{i=1}^k B_i$. Moreover, for any two such conformal diffeomorphisms $F\colon(N,\mathcal{C})\to (\Omega_k,g)$, $F'\colon(N,\mathcal{C})\to (\Omega_k',g')$, there exists a conformal automorphism $G\colon (M,g)\to (M',g')$ such that $G\circ F=F'$.

In other words, $(N,\mC)$ can be (uniquely) conformally identified with a complement of $k$ geodesic disks in a closed surface endowed with a metric of constant curvature.
\end{theorem}

\begin{remark} Theorem A in \cite{Maskit} as originally stated provides a biholomorphism of the \emph{interior} of $N_k$ onto the complement of closed disks, but standard results on boundary regularity of biholomorphisms imply that this extends to a diffeomorphism up to the boundary. For example, one can refer to the introduction of~\cite{BK}, where it is explained how the classical boundary regularity for the Riemann mapping theorem implies the analogous result for multiply connected planar domains. In particular, any biholomorphism of open annuli extends to a diffeomorphism of their closures -- a result which we can apply to small annuli near the boundary circles of $N_k$.
\end{remark}

As written, the original statement of Theorem A in \cite{Maskit} applies only to Riemann surfaces, but it is straightforward to extend it to the nonorientable case, as follows. Given a nonorientable compact surface $N$ with boundary, let $\pi\colon\tilde{N}\to N$ be the oriented double cover, with free antiholomorphic involution $s\colon \tilde{N}\to \tilde{N}$ such that $s^2=\mathrm{id}$ and $\pi\circ s=\pi$. By Haas and Maskit's results for orientable surfaces, we know that there exists a closed surface $(\tilde{M},g)$ of unit area and constant curvature, and geodesic disks $B_i\subset \tilde{M}$ with disjoint closures, such that $\tilde{N}$ admits a conformal diffeomorphism
$$
\tilde{F}\colon (\tilde{N},\mathcal{C})\to (\tilde{\Omega},g)
$$
onto the complement $\tilde{\Omega}=\tilde{M}\setminus \bigcup B_i$. Moreover, observe that the composition $\tilde{F}\circ s\colon (N,\mathcal{C})\to (\tilde{\Omega},g)$ with the antiholomorphic involution $s$ gives another conformal diffeomorphism, so by the uniqueness part of Maskit's theorem, there must exist a conformal diffeomorphism $G: (\tilde{M},g)\to (\tilde{M},g)$ such that $\tilde{F}\circ s=G$. Without loss of generality, we may assume that $G$ is an isometry of $(\tilde{M},g)$; if $\tilde{M}\neq \Sp$, this is automatic, while if $\tilde{M}=\Sp$, this may be achieved by replacing $\tilde{F}$ with $\Phi\circ \tilde{F}$ for a suitable conformal automorphism $\Phi\colon \Sp\to \Sp$.

Evidently, this isometry $G\colon (\tilde{M},g)\to (\tilde{M},g)$ preserves the image $\tilde{\Omega}=\tilde{M}\setminus \bigcup B_i$, reverses orientation, and satisfies $G^2=\mathrm{id}$ on $\tilde{\Omega}$; hence $G^2=\mathrm{id}$ on $M$, by unique continuation. Moreover, since $s$ has no fixed points on $\tilde{N}$, $G$ cannot have fixed points in $\tilde{\Omega}$. We claim now that $G$ has no fixed points in $\tilde{M}$. Indeed, if $G$ fixes a point $x\in \tilde{M}\setminus \tilde{\Omega}$, then $x$ must lie in the interior of one of the disks $B_i\subset M\setminus \Omega$, and since $G$ fixes the disjoint union $\bigcup \overline{B_i}$, we see that the restriction $G|_{\overline{B_i}}$ must then act as an antiholomorphic diffeomorphism of the closed disk $\overline{B_i}$. However, it follows from the standard classification of holomorphic automorphisms of the closed disk that any antiholomorphic automorphism $G$ of $\overline{B_i}$ must have fixed points on the boundary $\partial B_i\subset \partial\tilde{\Omega}$, which cannot occur since $G$ acts freely on $\tilde{\Omega}$.

We therefore see that the isometry $G\colon(\tilde{M},g)\to (\tilde{M},g)$ is an orientation-reversing involution acting freely on $\tilde{M}$, from which we obtain a smooth quotient surface 
$$
p\colon (\tilde{M},g)\to (M,g):=(\tilde{M},g)/G
$$
of constant curvature (and area $\frac{1}{2}$), such that the conformal diffeomorphism $\tilde{F}\colon(\tilde{N},\mathcal{C})\to (\tilde{\Omega},g)$ descends to a conformal diffeomorphism
$$F:(N,\mathcal{C})\to (\Omega,g)\subset (M,g),$$
where $\Omega:=p(\tilde{\Omega})=M\setminus p\left(\bigcup B_i\right)$. Evidently, the image $p\left(\bigcup_{i=1}^k B_i\right)$ of the geodesic disks under the $2$-fold covering map $p$ is a union of $k/2$ geodesic disks in $(M,g)$, so that $F$ gives the desired uniformization of $(N,\mathcal{C})$. Uniqueness of $F$ up to conformal diffeomorphisms likewise follows from uniqueness in the orientable case and the observation that any such uniformization $(N,\mathcal{C})\to (\Omega,g)$ lifts to an orientable uniformization $(\tilde{N},\mathcal{C})\to (\tilde{\Omega},g)$.


\subsection{Eigenvalues of measures}
In recent years, it has been observed that the study of variational problems for Laplace and Steklov eigenvalues fit into a useful, more general framework, based on assigning certain natural spectra to Radon measures on Riemann surfaces. To be precise, let $N$ be a compact surface with boundary (possibly empty) and let $\mC$ be a conformal class on $N$. Given a Radon measure $\mu$ on $N$ one can define the variational eigenvalues
\begin{equation}
\label{eq:eigmes}
\lambda_k(N,\mC,\mu)=\inf_{E_{k+1}} \,\, \sup_{0 \ne f \in E_{k+1}}\frac{\int_N |\nabla f|_{g}^2 dv_{g}}{\int_N f^2 d\mu},
\end{equation}
where $g\in\mC$ is any representative of the conformal class and $E_{k+1}$ ranges over all $(k+1)$-dimensional of $C^\infty(N)$; one then defines the mass-normalized eigenvalues
$$
\bar\lambda_k(N,\mC,\mu) = \lambda_k(N,\mC,\mu)\mu(N).
$$
We say that the measure $\mu$ is {\em admissible} (\cite{Kokarev, KS, GKL}) if the identity map on $C^\infty(M)$ can be extended to a compact map $W^{1,2}(M,g)\to L^2(\mu)$, $g\in\mC$. This definition does not depend on the choice of $g\in\mC$ and essentially guarantees that the eigenvalues $\lambda_k(N,\mC,\mu)$ behave similarly to the classical eigenvalues of the Laplacian, see e.g.~\cite{GKL}. While many examples of admissible measures lead to interesting eigenvalue problems~\cite[Section 4]{GKL}, the following are the only examples used in the present paper.
\begin{example}
\label{ex:Laplace}
Let $\partial N = \varnothing$, $\mu$ be a volume measure of a smooth metric $g\in \mC$, $\mu = dv_g$. Then the Rellich-Kondrachov compactness theorem implies that $\mu$ is admissible. In fact, then $\lambda_k(N,\mC,\mu) = \lambda_k(N,g)$ corresponds to the $k$th nontrivial eigenvalue of the Laplacian $\Delta_g$.
\end{example} 
\begin{example}
\label{ex:con_sing}
Let $\partial N = \varnothing$, $\mu = fdv_g$, where $g\in\mC$ is a smooth metric, $f\geqslant 0$ with zeroes of finite order at isolated points of $N$. Then $\mu$ is a volume measure of the metric $fg$, which is a smooth metric outside of finitely many conical singularities. 
The variational eigenvalues $\lambda_k(N,\mC,\mu)$ coincide with the eigenvalues of the Friedrichs extension of $\Delta_{fg}$ and we continue to write $\lambda_k(N,fg) = \lambda_k(N,\mC,\mu)$.
\end{example}
\begin{example}
\label{ex:Steklov}
Let $\partial N\ne \varnothing$, $\mu = ds_g^{\partial N}$ be the boundary length measure of a metric $g\in \mC$. Then the Sobolev trace embedding implies that $\mu$ is admissible and $\lambda_k(N,\mC,\mu) = \sigma_k(N,g)$. In particular, the Steklov eigenvalues $\sigma_k(N,g)$ depend only on the conformal class $[g]$ and the restriction of $g$ to the boundary $\partial N$. 
\end{example}
\begin{example}
Let $(M,g)$ be a closed surface and let $\Omega\subset M$ be a smooth domain. Let $\mu$ be the boundary length measure of $\partial\Omega$, $\mu = ds_g^{\partial\Omega}$. Consider $\overline{\Omega}\subset M$ as a manifold with boundary, then one has $\bar\lambda_k(\overline{\Omega},[g],\mu) = \bar\sigma_k(\Omega,g)$. Furthermore, the definition~\eqref{eq:eigmes} easily implies that 
\begin{equation}
\label{ineq:Sdom}
\bar\sigma_k(\Omega,g) = \bar\lambda_k(\overline{\Omega},[g],\mu) \leqslant \bar\lambda_k(M,[g],\mu).
\end{equation} 
\end{example}

\begin{remark}[Invariance under diffeomorphisms]
\label{rmk:diffeo}
Let $N$ be a compact surface, $g\in\Met_{\can}(N)$ and $\mu$ be an admissible measure on $N$. If $\Phi\colon N\to N_1$ is a diffeomorphism, then it is easy to see that $\lambda_1(N,[g],\mu) = \lambda_1(N_1,[(\Phi^{-1})^*g],\Phi_*\mu)$. Furthermore, if $f$ is an eigenfunction on $N_1$, then $\Phi^*f$ is an eigenfunction on $N$ with the same eigenvalue. In particular, this induces the  action of $\Diff(N)$ on the set of pairs $(g,\mu)$ by
$$
\Phi\cdot(g,\mu) = ((\Phi^{-1})^*g,\Phi_*\mu),
$$
which preserves the variational eigenvalues.
%
\end{remark}

Finally, we endow the space of all admissible measures with the topology induced by the $W^{-1,2}_g(M)$-norm. Namely, for any Radon measure $\mu$ we define
$$
\|\mu\|_{W^{-1,2}(M,g)}=\sup_{u}\int_M u\,d\mu,
$$ 
where the supremum is over all smooth functions $u$ satisfying $\|u\|_{W^{1,2}(M,g)}=1$. It is easy to see that any admissible $\mu$ has finite $W^{-1,2}(M,g)$-norm. 

\subsection{Geometric characterization of maximal metrics: Laplacian}
\label{sec:Laplace_extremal}
In the next two sections we recall some key results on the connection between eigenvalue optimization problems and minimal surfaces, starting with the Laplace eigenvalues of closed surfaces.

Given a closed surface $M$ consider again the supremum
$$
\Lambda_1(M) = \sup_g\bar\lambda_1(M,g);
$$
of the normalized first eigenvalue over all metrics on $M$, as well as the conformally-constrained supremum
$$
\Lambda_1(M,[g]) = \sup_{h\in [g]}\bar\lambda_1(M,h),
$$
where in the second quantity one can always assume $g\in\Met_{\can}(M)$. As was mentioned in the introduction, $\Lambda_1(M,[g])\leqslant\Lambda_1(M)<\infty$ for any surface $M$ and conformal class $[g]$ and the supremum is achieved for a metric, smooth up to a finite number of conical singularities~\cite{KS,PetL1,MS,KNPP2}. Furthermore, these singularities have integer angles. In particular, if $h$ is such a metric, then $h=fg,$ where $g\in\Met_{\can}(M)$ and $f\in C^\infty(M)$, with $f>0$ outside of finitely many points corresponding to the singularities of $h$; see Example~\ref{ex:con_sing}. If the metric $h$ (possibly with isolated conical singularities) is such that $\bar\lambda_1(M,g) = \Lambda_1(M)$ (or $\bar\lambda_1(M,h) = \Lambda_1(M,[g])$), then we say that $h$ is a $\bar\lambda_1$-(conformally) maximal metric. Additionally, keeping in mind Example~\ref{ex:con_sing}, we also say that $dv_h$ is a $\bar\lambda_1$-(conformally) maximal measure. We denote by $\Met_0(M)\subset\Met_{\can}(M)$ the subset of unit-area, constant curvature metrics corresponding to $\bar\lambda_1$-maximal conformal classes; i.e. 
$$
\Met_0(M):=\{g\in \Met_{\can}(M)\mid \Lambda_1(M,[g])=\Lambda_1(M)\}.
$$

Recall that a map $u \colon (M,g)\to\mathbb{S}^n$ is called {\em harmonic} if 
$$
\Delta_g u = |du|_g^2u,
$$
which holds precisely when $u$ is a critical point for the \emph{energy}
$$
E_g(u) = \frac{1}{2}\int_M|du|^2_g\,dv_g
$$
among $\mathbb{S}^n$-valued maps. This equation is conformally invariant on surfaces--i.e. $u$ is harmonic with respect to any other metric in the conformal class $[g]$. In particular, setting $g_u = \frac{1}{2}|du|_g^2g$ one obtains $\Delta_{g_u}u = 2u$, so that the components are the eigenfunctions of $\Delta_{g_u}$ with eigenvalue $2$. If, furthermore, $\lambda_1(M,g_u) = 2$, then we say that $u$ is of {spectral index $1$} and write $\ind_S(u) = 1$. Note that $du = 0$ only at isolated points of $M$, which correspond to conical singularities of $g_u$.

\begin{theorem}[\cite{ESI, FS:extremal}]
Let $g$ be a $\bar\lambda_1$-conformally maximal metric. Then there exists $n>0$, a harmonic map $u\colon (M,[g])\to\mathbb{S}^n$ of spectral index $1$ and $\alpha>0$ such that $g = \alpha g_u$. In particular, $\Lambda_1(M,[g]) = 2E_g(u)$. 
\end{theorem} 
\begin{remark}
\label{rmk:Laplace_mult}
Note that $n$ is bounded by the multiplicity of the first eigenvalue of $\Delta_{g_u}$, but is not necessarily equal to it~\cite[Remark 1.4]{CKM}. The multiplicity bounds of~\cite{Cheng, Besson, Nad:mult} imply that $n$ is bounded from above only in terms of the topology of $M$.
\end{remark}

Conversely, for any harmonic map $u\colon(M,g)\to\mathbb{S}^n$ of spectral index $1$ satisfying $2E_g(u)= \Lambda_1(M,[g])$, the metric $g_u$ is $\bar\lambda_1$-conformally maximal. We say that such a map $u$ is a $\bar\lambda_1$-conformally maximal map. 

A map $u\colon (M,g)\to (P,h)$ is called {\em weakly conformal} if $u^*h = g_u$. On surfaces, any weakly conformal harmonic map is a branched minimal immersion and vice versa. The branch points of the immersion correspond to the singularities of $g_u$.
\begin{theorem} [\cite{Nad, ESI, FS:extremal}]
Let $g$ be a $\bar\lambda_1$-maximal metric. Then there exists $n>0$, a branched minimal immersion $u\colon M\to\mathbb{S}^n$ of spectral index $1$ and $\alpha>0$ such that $g = \alpha u^*g_{\mathbb{S}^n}$. In particular, $\Lambda_1(M) = 2\area(u(M))$.
\end{theorem}

Conversely, for any branched minimal immersion $u\colon M\to\mathbb{S}^n$ of spectral index $1$ satisfying $2\area(u(M))= \Lambda_1(M,[g])$, the metric $u^*g_{\mathbb{S}^n}$ is $\bar\lambda_1$-maximal. We say that such a map $u$ is a $\bar\lambda_1$-maximal map. 

\subsection{Geometric characterization of maximal metrics: Steklov} 
\label{sec:Steklov_extremal}
The variational theory for normalized Steklov eigenvalues is to a large extent parallel to that of the Laplacian.

Given a connected compact surface with boundary $N$, we consider again the supremum
$$
\Sigma_1(N) = \sup_g\bar\sigma_1(N,g);
$$
of the first nontrivial (length-normalized) Steklov eigenvalue over all metrics on $N$, as well as the conformally constrained supremum
$$
\Sigma_1(N,[g]) = \sup_{h\in [g]}\bar\sigma_1(N,h).
$$
As was mentioned in the introduction, $\Sigma_1(N)<\infty$ for any surface $N$, and the supremum is achieved by a smooth metric, possibly with some conical singularities. For the conformally-constrained supremum, the situation is a little different.
\begin{theorem}[~\cite{PetS}]
\label{thm:Steklov_2pi}
Assume that $\Sigma_1(N,[g])>2\pi$. Then the supremum is achieved for a smooth metric (possibly with a finite number of conical singularities).
\end{theorem}
\begin{remark}
It is expected that $\Sigma_1(N,[g])>2\pi$ for $N\ne\mathbb{D}$ and for any conformal class $[g]$, but as of this writing, this has only been verified for the annulus and the M\"obius band~\cite{MP2}.
\end{remark}

If the metric $g$ is such that $\bar\sigma_1(N,g) = \Sigma_1(N)$ (or $\bar\sigma_1(N,g) = \Sigma_1(N,[g])$), then we say that $g$ is a $\bar\sigma_1$-(conformally) maximal metric. Additionally, keeping in mind Example~\ref{ex:Steklov}, we also say that $ds^{\partial N}_g$ is a $\bar\sigma_1$-(conformally) maximal measure. Note that by Example~\ref{ex:Steklov}, if $g$ is a $\bar\sigma_1$-(conformally) maximal metric, then any $h\in [g]$ with $ds_h^{\partial N} =  ds_g^{\partial N}$ is also $\bar\sigma_1$-(conformally) maximal. For that reason, in the following we predominantly refer to $\bar\sigma_1$-(conformally) maximal measures as opposed to metrics.

Recall that a map $u\colon (N,g)\to\mathbb{B}^{n+1}$ is called {\em free boundary harmonic} if $u(\partial N)\subset \mathbb{S}^n = \partial\mathbb{B}^{n+1}$ and
\begin{equation*}
\begin{cases}
\Delta_g u = 0 &\text{ in $N$;}\\
\partial_{\nu_g} u\parallel u &\text{ on $\partial N$,}
\end{cases}
\end{equation*}
where $\nu_g$ is an outer unit normal. Its energy satisfies
$$
E_g(u) = \frac{1}{2}\int_N|du|_g^2\,dv_g = \frac{1}{2}\int_{\partial N}|\partial_{\nu_g} u|\,ds_g
$$
Similar to the harmonic maps, this definition only depends on the conformal class $[g]$ in our two-dimensional setting. In particular, setting $\mu_u = |\partial_{\nu_g} u|\,ds_g$ one obtains that the components of $u$ are Steklov eigenfunctions associated with the measure $\mu_u$, whose eigenvalue is equal to $1$. If, furthermore, $\lambda_1(N,[g],\mu_u) = 1$, then we say that $u$ is of spectral index $1$ and write $\ind_S(u) = 1$. 

\begin{theorem}[\cite{FS:extremal, KM}]
Let $\mu$ be a $\bar\sigma_1$-conformally maximal measure. Then there exists $n>0$, a free boundary harmonic map $u\colon (N,[g])\to\mathbb{B}^{n+1}$ of spectral index $1$ and $\alpha>0$ such that $\mu=\mu_u$. In particular, 
$\Sigma_1(N,[g])=2E_g(u)$.
\end{theorem}
\begin{remark}
\label{rmk:Steklov_mult}
Similarly to Remark~\ref{rmk:Laplace_mult}, the number $n$ is bounded only in terms of the topology of $N$,~\cite{FS, Jammes, KKP}. Furthermore--and crucially for the purposes of the present paper--the upper bound \emph{does not depend on the number of boundary components of $N$}. 
\end{remark}

Conversely, for any free boundary harmonic map $u\colon (N,g)\to\mathbb{B}^{n+1}$ of spectral index $1$ satisfying  $2E_g(u)=\Sigma_1(N,[g])$, the measure $\mu_u$ is $\bar\sigma_1$-conformally maximal. We refer to such maps $u$ as $\bar\sigma_1$-conformally maximal maps.

\begin{theorem}[\cite{FS:extremal, KM}]
Let $\mu$ be a $\bar\sigma_1$-maximal measure. There exists $n>0$, a free boundary branched minimal immersion $u\colon N\to\mathbb{B}^{n+1}$ of spectral index $1$ and $\alpha>0$ such that $\mu=\alpha ds_{u^*g_{\mathbb{S}^n}}$. In particular, $\Sigma_1(N) = 2\area (u(N))$.
\end{theorem}

For any free boundary branched minimal immersion $u\colon N\to\mathbb{B}^{n+1}$ of spectral index $1$ satisfying $2\area (u(N))=\Sigma_1(N)$, the measure $ds^{\partial N}_{u^*(g_{\mathbb{S}^n})}$ is $\bar\sigma_1$-maximal. We say that such a map $u$ is $\bar\sigma_1$-maximal. 

\subsection{Convergence of $\Sigma_1(N_k)$} 
\label{sec:GLKS}
In the present section we explain the ideas behind the identities~\eqref{ineq:KS_intro},~\eqref{lim:GL_intro}.

In \cite{KS}, the following regularity/rigidity result for conformally $\bar{\lambda}_1$-maximal measures is obtained as a byproduct of a new characterization of $\Lambda_1(M,[g])$ via the min-max theory of harmonic maps.

\begin{theorem}[Regularity of maximal measures,~\cite{KS}]
\label{thm:KS}
Let $M$ be a closed surface and $\mC$ be a conformal class on $M$. Then for any admissible measure $\mu$ on $M$ one has
\begin{equation}
\label{ineq:KS}
\bar\lambda_1(M,\mC,\mu)\leqslant \Lambda_1(M,\mC)
\end{equation}
with equality iff $\mu$ is a $\bar\lambda_1$-conformally maximal measure, i.e. $\mu=dv_g$, where $g$ is a $\bar\lambda_1$-conformally maximal metric.
\end{theorem}
The meaning of this theorem is as follows: even after relaxing the optimization problem for $\bar\lambda_1(M,g)$ to include (admissible) measures, the set of maximizers (and, as a result, the optimal value) does not change. Now, let $\Omega\subset M$ be a smooth domain. Combining~\eqref{ineq:Sdom} with~\eqref{ineq:KS} implies that
$$
\Sigma_1(\overline{\Omega},\mC)<\Lambda_1(M,\mC),
$$
where we abuse notation slightly by letting $\mC$ denote the conformal class on $\overline\Omega$ induced by the inclusion $\Omega\subset (M,\mC)$. Taking the supremum over all conformal classes $\mC$ yields~\eqref{ineq:KS_intro}.

The relation~\eqref{lim:GL_intro} follows from the following theorem.

\begin{theorem}[\cite{GL}]
\label{thm:GL}
For any closed surface $(M,g)$ there exists a sequence of domains $\Omega_k\subset M$, such that 
\begin{equation}
\label{lim:GL1}
\bar\sigma_1(\Omega_k,g)\to\bar\lambda_1(M,g)
\end{equation}
as $k\to\infty$. The domains $\Omega_k$ are obtained by removing many small disks from $M$.
\end{theorem}

As in the introduction, let $M$ be a closed surface, and denote by $N_k$ the compact surface with boundary obtained by removing $k$ disjoint disks from $M$. It is easy to see that the sequence $\Sigma_1(N_k)$ is non-decreasing (and much harder to see that it is strictly increasing, see~\cite{MP}). Thus, taking the supremum over all $g$ in~\eqref{lim:GL1} yields
$$
\lim_{k\to\infty}\Sigma_1(N_k)\geqslant \Lambda_1(M),
$$
which combined with~\eqref{ineq:KS} yields~\eqref{lim:GL_intro}.

For convenience, we formulate the following corollary of the proof of Theorem~\ref{thm:GL}.
\begin{proposition}
\label{prop:conf_GL}
For any closed surface $(M,g)$ there exists a sequence of domains $\Omega_k\subset M$, such that
$$
\lim_{k\to\infty}\Sigma_1(\Omega_k,[g]) = \Lambda_1(M,[g]).
$$
Furthermore, the domains $\Omega_k$ can be chosen so that the harmonic extension operator $\mH_k\colon W^{1,2}_g(\Omega_k)\to W^{1,2}_g(M\setminus\Omega_k)$ satisfies 
\begin{equation}
\label{cond:harm_extension}
\lim_{k\to\infty}\|\mH_k\|\to 0.
\end{equation}
\end{proposition}

\subsection{Varifold convergence of $\bar\sigma_1$-maximal maps} 
\label{sec:convergence}
Let us recall some basic notions from the theory of varifolds, following~\cite[Chapter 3]{CM}. Let $\pi\colon\GrB\to\mathbb{R}^{n+1}$ denote the bundle of (tangent) $2$-planes over $\mathbb{R}^{n+1}$. A {\em 2-varifold}  $T$ is a Radon measure on $\GrB$. The {\em weight measure} of $T$ is the pushforward $\nu_T:=\pi_*(T)$. 
Given a Sobolev map $v\in W^{1,2}(N,\mathbb{R}^{n+1})$ from a surface $(N,g)$ (possibly with boundary) to $\mathbb{R}^{n+1}$, one defines the associated $2$-varifold $T_v\in C_0^0(\mathcal{G}_2(n+1))^*$ by 
$$
\int_{\mathcal{G}_2(n+1)}fdT_v:=\int_{N\cap \{J_v(x)>0\}} f(v(x), dv(T_xN))J_v(x) dv_g,
$$
where $J_v(x)$ denotes the Jacobian determinant
$$
J_v(x):=\sqrt{\mathrm{det}_g(dv_x^tdv_x)}=\sqrt{\mathrm{det}_g(v^*g_{\mathbb{R}^{n+1}})(x)}.
$$
Note that, while $J_v(x)$ and $dv_g$ depend on the metric $g$, their product does not, and in the case where $v\colon N\hookrightarrow\mathbb{R}^{n+1}$ is a smooth embedding, the preceding definition is equivalent to setting
$$
T_v(U)=\area(U\cap TN),
$$
by the area formula. Similarly, if $v$ is a branched $d$-sheeted covering over the image, then
$$
T_v(U)=d\area(U\cap Tv(N)).
$$ 
A sequence of varifolds $T_k$ is said to converge to $T$ if they $*$-weak converge as measures. A sequence of surfaces $N_i\subset \mathbb{R}^{n+1}$ (possibly with multiplicity) arising as images of branched conformal immersions is likewise said to converge to $M\subset\mathbb{R}^{n+1}$ in the varifold sense if the corresponding varifolds converge.

Recall now the setup from Section~\ref{sec:GLKS}: $M$ is a closed surface, $N_k$ is the compact surface with boundary obtained by removing $k$ disjoint disks from $M$. 
For each $k$, choose a $\bar\sigma_1$-maximal map $u_k\colon N_k\to\mathbb{B}^{n_k+1}$. While in principle the dimension of the ball $n_k+1$ does depend on $k$, Remark~\ref{rmk:Steklov_mult} guarantees that 
$n_k$ are bounded independent of $k$. Thus, without loss of generality, one can assume $n_k\equiv n$. This allows us to define the $2$-varifolds $T_k$ in $\mathbb{B}^{n+1}\subset\mathbb{R}^{n+1}$ associated to $u_k$.
 Thus, the exact statement of Theorem~\ref{thm:convergence_intro} is as follows
\begin{theorem}
\label{thm:convergence}
There exists a $\bar\lambda_1$-maximal map $u\colon M\to \mathbb{S}^n$, such that up to a choice of a subsequence
the varifolds $T_k$ associated to $\bar\sigma_1(N_k)$-maximal maps $u_k$ converge to the varifold $T$ associated with $u$.
\end{theorem}

We record several consequences of the varifold convergence, which may paint a clearer picture for readers unfamiliar with varifolds.
\begin{corollary}
Along the converging subsequence $T_k\rightharpoonup^*T$ one has
\begin{enumerate}
\item the free boundary branched minimal surfaces $u_k(N_k)\subset\mathbb{B}^{n+1}$ converge to the branched minimal surface $u(M)\subset\mathbb{S}^n$ in the Hausdorff distance;
\item the boundary length measures of $u_k(N_k)$ converge to the twice the area measure of $u(M)$. 
\end{enumerate}
\end{corollary}
\begin{proof}
To prove (1), assume the contrary, i.e. there exists a further subsequence, $\delta>0$ and a point $x_k\in u_k(N_k)$ at the distance $\delta$ from $\mathbb{S}^{n+1}$. Applying a rotation and extracting a further subseqeunce, we can assume that $x_k\equiv y$ does not depend on $k$. Let $0<f\in C^0_0(\mathbb{B}^{n+1})$ be equal to $1$ on the ball of radius $\delta/2$ around $y$ and $0$ outside the ball of radius $\delta$. Let $\nu_k$, $\nu$ be the weight measures of $T_k$ and $T$ respectively.  Then the monotonicity formula for minimal surfaces (see e.g.~\cite[Proposition 1.12]{CM}, implies that $\nu_k(f)\geqslant \frac{\pi\delta^2}{4}$. At the same time, the varifold convergence yields
$$
\frac{\pi\delta^2}{4}\leqslant\nu_k(f)\to \nu(f) = 0,
$$ 
which is a contradiction.

To show (2) let $f\in C_0^0(\mathbb{R}^{n+1})$ and consider the vector field $X(x) = fx$ on $\mathbb{R}^{n+1}$. Then the first variation formula implies
$$
\int_{\partial u_k(N_k)} f = \int_{u_k(N_k)}\left(2f + \la x,\nabla^{u_k(N_k)} f\ra\right)
$$
Define $F\in C_0^0(\GrB)$ by $F(x,\Pi) = \la x,\nabla^{\Pi} f(x)\ra$, where $\nabla^\Pi f(x)$ is the projection of $\nabla f(x)$ onto $\Pi$. Then the varifold convergence implies
$$
 \int_{u_k(N_k)}\la x,\nabla^{u_k(N_k)} f\ra = \int F\,dT_k\to \int F\,dT = \int_{u(M)}\la x,\nabla^{u(M)} f\ra = 0, 
$$
since $x\perp T\mathbb{S}^n\supset Tu(M)$. At the same time, since $\nu_k\rightharpoonup^*\nu$ one has
$$
\int_{u_k(N_k)}2f\to \int_{u(M)} 2f,
$$
which completes the proof.
\end{proof}

\subsection{Convergence of $\bar\sigma_1$-conformally maximal maps} 
\label{sec:convergence_conf}
Let $(M,\mC)$ be a closed surface with a fixed conformal class $\mathcal{C}$. Consider domains $\Omega\subset M$ with the restricted conformal class, which we denote by the same letter $\mC$. By Proposition~\eqref{prop:conf_GL} there exist sequences $\Omega_k\subset M$ satisfying
$$
\Sigma_1(\Omega_k,\mC)\to\Lambda_1(M,\mC).
$$
In particular, since $\Lambda_1(M,\mC)\geqslant 8\pi$, Theorem~\ref{thm:Steklov_2pi} implies that for large enough $k$ there exist a $\bar\sigma_1(\Omega_k,\mC)$-conformally maximal map $u_k\colon (\Omega_k,\mC)\to\mathbb{B}^{n_k+1}$. By Remark~\ref{rmk:Steklov_mult}, we can assume $n_k\equiv n$ is independent of $k$.
The following theorem describes convergence properties of the sequence $\{u_k\}$. 

\begin{theorem}
\label{thm:convergence_conf}
Let $(M,\mC)$ be a closed surface with a fixed conformal class, $g\in\mC$. Let $\Omega_k\subset M$ be a sequence of domains such that 
\begin{equation}
\label{cond:convergence}
\Sigma_1(\Omega_k,\mC)\to\Lambda_1(M,\mC).
\end{equation}
Assume further that the $\bar\sigma_1(\Omega_k,\mC)$-conformally maximal maps $u_k\colon (\Omega_k,\mC)\to\mathbb{B}^{n+1}$ admit an extension $\hat u_k\in W_g^{1,2}(M,\mathbb{B}^{n+1})$ such that
\begin{equation}
\label{cond:extension}
\lim_{k\to\infty}E_g(\hat u_k;M\setminus\Omega_k)=0.
\end{equation}
Then there exists a $\bar\lambda_1(M,\mC)$-conformally maximal map $u\colon (M,\mC)\to\mathbb{S}^n$, such that, up to a choice of a subsequence, $\hat u_k\to u$ in $W^{1,2}_g(M,\mathbb{B}^{n+1})$.
\end{theorem}
\begin{remark}
It is plausible that the condition~\eqref{cond:extension} is superfluous, i.e. it could be a consequence of~\eqref{cond:convergence}. 
\end{remark}

Note that the relation~\eqref{cond:harm_extension} implies the condition~\eqref{cond:extension}. In particular, Proposition~\ref{prop:conf_GL} ensures that for each $(M,\mC)$ there is at least one sequence $\Omega_k$ satisfying the conditions of Theorem~\ref{thm:convergence_conf}. Furthermore, by examining the arguments in~\cite{GL} it is easy to see that there is, in fact, a lot of freedom in the construction of such $\Omega_k$. Thus, the type of convergence described in Theorem~\ref{thm:convergence_conf} is not uncommon, which is indicative of a more general theory for such sequences of free boundary harmonic maps.



\section{Convergence of $\bar\sigma_1$-maximal maps}


\subsection{Qualitative stability of $\bar\lambda_1$-maximal metrics}

A key ingredient in the proof of Theorem \ref{thm:convergence} is the following qualitative stability result for globally $\bar{\lambda}_1$-maximizing measures, see~\cite[Theorem 1.2, Theorem 1.14]{KNPS}.

\begin{theorem}[\cite{KNPS}] 
\label{thm:glob_qual_stab}

Let $\mu_k$ be a sequence of admissible probability measures on $M$ and $g_k\in \Met_{\can}(M)$ a sequence of constant curvature metrics such that 
$$
\lambda_1(M,[g_k],\mu_k)\to \Lambda_1(M)
$$
as $k\to\infty$. Then there exist $\Phi_k\in\Diff(M)$, $g\in \Met_0(M)$ and a $\bar{\lambda}_1$-maximal probability measure $\mu_{\max}$ such that, up to a choice of a subsequence, the pairs
$$
(\tilde g_k,\tilde \mu_k):= \Phi_k\cdot(g_k,\mu_k)
$$
satisfy
\begin{equation}
\label{lim:glob_qual_stab}
\|\tilde g_k - g\|_{C^1(g)} + \|\tilde\mu_k - \mu_{\max}\|_{W^{-1,2}(g)}\to 0.
\end{equation}

If $M=\Sp$, then one can additionally choose $\mu_{\max} = dv_g$.
\end{theorem}
\begin{remark}
For a closed surface $M\ne\Sp$ one has $\mu_{\max} = fdv_g$, $f\in C^\infty(M)$ and the set of $\bar\lambda_1$-maximal measures is compact (up to the action by diffeomorphisms). In particular, $\|f\|_\infty\leqslant C$, where $C$ only depends on $M$. The last statement of the theorem implies that the same inequality can be used on $\Sp$.
\end{remark}

For technical reasons, we find it convenient to replace the $W^{-1,2}(g)$ distance in the conclusion~\eqref{lim:glob_qual_stab} with a slightly different (but equivalent) one, which has the advantage of being conformally invariant, in addition to simplifying some computations.

\begin{definition}
Let $\nu$, $\mu$ be two probability measures on $M$ and let $g$ be a metric on $M$. Then we set
$$
\|\nu - \mu\|_{\Wdot(g)}:=\sup\left\{\int_Mf\,d(\mu-\nu)\,|\, f\in C^{\infty}(M),\,\|df\|_{L^2(g)} = 1\right\}.
$$
\end{definition}

Extended to measures of arbitrary mass, this definition would yield a pseudometric; for probability measures, however, we have the following.

\begin{lemma}
\label{lem:Wdot}
For any probability measures $\mu,\nu$ on $M$, one has 
$$
\|\nu - \mu\|_{W^{-1,2}(g)}\leqslant\|\nu - \mu\|_{\Wdot(g)}\leqslant \sqrt{1+\frac{1}{\lambda_1(M,g)}}\|\nu - \mu\|_{W^{-1,2}(g)}.
$$
\end{lemma}
\begin{proof}
Recall that the $W^{-1,2}(g)$ norm is given by
$$\|\nu-\mu\|_{W^{-1,2}(g)}:=\sup\left\{\int_Mf\,d(\mu-\nu)\,|\, f\in C^{\infty}(M), \, \|f\|_{W^{1,2}(g)}=1\right\},$$
where the $W^{1,2}(g)$ norm of a function $f$ is given as usual by
$$\|f\|_{W^{1,2}(g)}^2:=\|f\|_{L^2(g)}^2+\|df\|_{L^2(g)}^2.$$
In particular, comparing with the definition of $\Wdot(g)$ and noting that $\|df\|_{L^2(g)}\leq \|f\|_{W^{1,2}(g)}$ holds trivially, the first inequality
$$\|\nu-\mu\|_{W^{-1,2}(g)}\leqslant \|\nu-\mu\|_{\Wdot(g)}$$
is immediate.

For the latter inequality, note that since $\nu$ and $\mu$ are probability measures, one has
$$
\int_Mf\,d(\mu-\nu) = \int_M(f + c)\,d(\mu-\nu)
$$
for any constant $c$; as a consequence, one can equivalently characterize the $\Wdot(g)$ metric via
$$\|\nu-\mu\|_{\Wdot(g)}=\sup\left\{\int_Mf\,d(\mu-\nu)\, |\, f\in C^{\infty}(M), \,\|df\|_{L^2(g)}=1,\, \int_M fdv_g=0\right\}.$$
But for $f\in C^{\infty}(M)$ satisfying $\int_Mf dv_g=0$, we of course have
$$\lambda_1(M,g)\|f\|^2_{L^2(g)}\leqslant \|df\|^2_{L^2(g)},$$
so that $\|f\|_{W^{1,2}(g)}\leqslant \sqrt{1+\frac{1}{\lambda_1(M,g)}}$, and the desired bound follows easily from definitions.
\end{proof}

Lemma~\ref{lem:Wdot} implies that in the conclusion of Theorem~\ref{thm:glob_qual_stab} one can replace $W^{-1,2}$ by $\Wdot$-distance, i.e. 
\begin{equation}
\label{lim:glob_qual_stab2}
\|\tilde\mu_k - \mu_{\max}\|_{\Wdot(g)}\to 0.
\end{equation}

 
 A result similar to Theorem \ref{thm:glob_qual_stab} holds in a fixed conformal class, see \cite{KNPS}[Theorem 1.9]. We assume that $M\ne \Sp$, since in that case there is only one conformal class, and the result is already covered by Theorem~\ref{thm:glob_qual_stab}.
 
 \begin{theorem}
 \label{thm:conf_qual_stab}
  Assume $M\ne\Sp$. Let $g\in\Met_{\can}(M)$ and $\mu_k$ be a sequence of admissible probability measures such that
 $$
 \lambda_1(M,[g],\mu_k)\to\Lambda_1(M,[g]).
 $$
 Then there exist a conformally $\bar\lambda_1$-maximal probability measure in $[g]$ such that, up to a choice of a subsequence, $\mu_k\to \mu_{\max}$ in $W^{-1,2}(M,g)$.
 \end{theorem}
Note that by the discussion above, the $W^{-1,2}$-distance can be replaced by $\Wdot$-distance.
%


\subsection{Proof of Theorem~\ref{thm:convergence}}
As before, let $M$ be a closed surface, and let $N_k$ be a surface with boundary diffeomorphic to $M$ with $k$ disjoint disks removed.
Let $h_k$ be a $\bar\sigma_1$-maximal metric on $N_k$, so that $\Sigma_1(N_k) = \bar\sigma_1(N,h_k),$ normalized to have $\length(\partial N_k,h_k) = 1$. 
By Theorem~\ref{thm:uni} $(N_k,h_k)$ can be conformally identified with a domain $\Omega_k\subset (M,g_k)$, where $g_k\in\Met_\can(M)$. 

Denote by $\mu_k$ the push-forward $F_*(ds_{h_k}^{\partial N_k})$ of the boundary length measure $ds_{h_k}^{\partial N_k}$ by the conformal embedding $F:(N,h_k)\to (M,g_k)$. By Remark~\ref{rmk:diffeo} and relation~\eqref{lim:GL_intro}, we obtain admissible probability measures $\mu_k$ supported on $\partial \Omega_k$ satisfying 
$$\Sigma_1(N_k) = \lambda_1(\Omega_k,[g_k],\mu_k)\to\Lambda_1(M)$$
as $k\to\infty$. By inequality~\eqref{ineq:Sdom}, one further has that $\lambda_1(M,[g_k],\mu_k)\to \Lambda_1(M)$ and, therefore, the measures $\mu_k$ satisfy the conditions of Theorem~\ref{thm:glob_qual_stab}. It turns that the biggest challenge in proving Theorem~\ref{thm:convergence} is showing that the $\bar\sigma_1$-maximal maps corresponding to $h_k$ have a small energy extension to $M$. This is item (2) of the following proposition, whose proof we postpone to the next subsection.

\begin{proposition}
\label{prop:extension} 
For every $k$ there exists $g_k\in \Met_{\can}(M)$ and a smooth domain $\Omega_k\subset (M,g_k)$ such that
\begin{enumerate}
\item The $\bar\sigma_1$-maximal metrics $h_k$ on $N_k$ are conformally equivalent to the domain 
$(\overline\Omega_k,g_k)\subset (M,g_k)$;
\item The corresponding (branched) free boundary minimal immersions $u_k\colon\Omega_k\to \mathbb{B}^{n+1}$ admit an extension $\hat u_k\in W^{1,2}(M,\mathbb{B}^{n+1};g_k)$ such that 
$$
E_{g_k} (\hat u_k;M\setminus\Omega_k)\to 0
$$
as $k\to\infty$.
\item There exists $g\in\Met_0(M)$ and a $\bar{\lambda}_1$-maximizing probability measure $\mu_{max}$--so that $\bar{\lambda}_1(M,[g],\mu_{max})=\Lambda_1(M)$--for which $g_k\to g$ in $C^1(g)$ and $\|\mu_k-\mu_{\max}\|_{\Wdot(g)}\to 0$.
\end{enumerate}
 In particular, items (2) and (3) imply that $E_g(\hat u_k;M)\to \frac{1}{2}\Lambda_1(M)$.
\end{proposition}
Item (3) is a direct consequence of item (1) and Theorem~\ref{thm:glob_qual_stab}, since by Remark~\ref{rmk:diffeo} we can assume without loss of generality that $\Phi_k=\mathrm{id}$. Note that we do not explicitly require here that the domains $\Omega_k$ are complements of geodesic discs, although the full power of Theorem~\ref{thm:uni} is used in the proof of item (2) in Section~\ref{sec:extension} below (and again in Section \ref{sec:u.bds}). For now, let us show how Proposition~\ref{prop:extension} implies Theorem~\ref{thm:convergence}.

For the remainder of the section we work with the metric $g$, in particular, $W^{1,2}(M)$ refers to the Sobolev space with respect to $g$. Items (2) and (3) imply that the sequence $\hat u_k$ is uniformly bounded in $W^{1,2}(M,\mathbb{B}^{n+1})$. Therefore, up to a choice of a subsequence, $\hat u_k$ converge weakly to a map $u\in W^{1,2}(M,\mathbb{B}^{n+1})$. 

\begin{lemma}
\label{lemma:strong_H1}
The limit map $u$ is a weakly conformal harmonic map (i.e., a branched minimal immersion) to the sphere $M\to \mathbb{S}^n=\partial\mathbb{B}^{n+1}$, whose components are $\lambda_1(M,[g],\mu_{\max})$-eigenfunctions. Furthermore, $\hat u_k\to u$ strongly in $W^{1,2}(M,\mathbb{B}^{n+1})$.
\end{lemma}
\begin{proof}

Since $|\hat u_k|^2\leqslant 1$ are uniformly bounded in $W^{1,2}(M)$, up to a choice of a subsequence, we may assume that $|\hat u_k|^2$ converge to $|u|^2$ weakly in $W^{1,2}(M)$, and since $\mu_{\max}\in W^{-1,2}(M)$, it follows that
$$
\int_M(1-|u|^2)\,d\mu_{\max} = \lim_{k\to\infty}\int_M(1-|\hat u_k|^2)\,d\mu_{\max}.
$$
Moreover, since $\mu_k\to\mu_{\max}$ in $W^{-1,2}(M)$ and $|\hat u_k|^2$ are uniformly bounded in $W^{1,2}(M)$, one further has
$$
\lim_{k\to\infty}\int_M(1-|\hat u_k|^2)\,d\mu_{\max} = \lim_{k\to\infty}\int_M(1-|\hat u_k|^2)\,d\mu_k = 0,
$$
where in the last step we used that $|\hat u_k|^2\equiv 1$ on $\supp(\mu_k)\subset\partial\Omega_k$. Recalling that $(1-|u|^2)\geqslant 0$, we obtain $u\in W^{1,2}(M,\mathbb{S}^n)$.

Next, for any $v\in C^\infty(M,\mathbb{R}^{n+1})$ one has 
\begin{equation*}
\begin{split}
\int_M\la du,dv\ra\,dv_g &= \lim_{k\to\infty}\int_M\la d\hat u_k,dv\ra\,dv_g \\
\text{(since $E_g(\hat u_k; M\setminus\Omega_k)\to 0$)} &= \lim_{k\to\infty}\int_{\Omega_k}\la d u_k,dv\ra\,dv_g \\
\text{(since $g_k\to g$ in $C^1(g)$)} &=\lim_{k\to\infty}\int_{\Omega_k}\la d u_k,dv\ra\,dv_{g_k}.
\end{split}
\end{equation*}
In particular, since the components of the maps $u_k:\Omega_k\to \mathbb{B}^{n+1}$ are Steklov eigenfunctions corresponding to the eigenvalue $\lambda_1(\Omega_k,[g_k],\mu_k)=\Sigma_1(N_k)$, and $\lim_{k\to\infty}\Sigma_1(N_k)=\Lambda_1(M)$ by \eqref{lim:GL_intro}, this gives

\begin{equation*}
\begin{split}
\int_M\la du,dv\ra dv_g &=\lim_{k\to\infty}\Sigma_1(N_k)\int_{\Omega_k}\la u_k,v\ra\,d\mu_k\\
\text{(since $\supp(\mu_k)\subset\partial\Omega_k$)}&=\Lambda_1(M)\lim_{k\to\infty}\int_{M}\la \hat u_k,v\ra\,d\mu_k.
\end{split}
\end{equation*}

In particular, since $d\mu_k\to d\mu_{max}$ in $W^{-1,2}$ and $\|\langle \hat{u}_k,v\rangle\|_{W^{1,2}}\leq C$, and using the fact that $\langle \hat{u}_k,v\rangle\rightharpoonup \la u,v\ra$ weakly in $W^{1,2}$, we deduce that

\begin{equation*}
\begin{split}
\int_M\la du,dv\ra dv_g &=\Lambda_1(M)\lim_{k\to\infty}\int_{M}\la \hat u_k,v\ra\,d\mu_{\max}=\\
 &=\Lambda_1(M)\int_{M}\la u,v\ra\,d\mu_{\max}.
\end{split}
\end{equation*}
As a result, since $\lambda_1(M,[g],\mu_{\max})=\Lambda_1(M)$, the components of $u$ are $\lambda_1(M,[g],\mu_{\max})$-eigenfunctions. In particular, since $\mu_{\max} = dv_{g_{\max}}$ for the maximizing metric $g_{\max}\in [g]$, this implies
$$
0 = \Delta_{g_{\max}}(|u|^2) = 2\Lambda_1(M) - 2|du|^2_{g_{\max}},
$$
i.e. $\Delta_{g_{\max}}u = |du|^2_{g_{\max}} u$, which implies that $u\in W^{1,2}(M,\mathbb{S}^n)$ is harmonic.

Now, since 
$$
\lim_{k\to\infty}E_g(\hat u_k;M) = \frac{1}{2}\Lambda_1(M)=\frac{\Lambda_1(M)}{2}\int |u|^2\,d\mu_{\max} =\frac{1}{2} \int_M|du|^2\,dv_g = E_g(u;M),
$$
we see that there is no energy drop in the limit, and $\hat u_k\to u$ strongly in $W^{1,2}(M)$.

Finally, using the facts that $\hat{u}_k\to u$ strongly in $W^{1,2}(M)$, $g_k\to g$ in $C^1(g)$, $E_{g_k}(\hat{u}_k;M\setminus\Omega_k)\to 0$, and the branched free boundary minimal immersions $u_k$ are conformal on $(\Omega_k,g_k)$, we have the $L^1$ convergence of the stress-energy tensors
\begin{equation*}
\begin{split}
du^tdu-\frac{1}{2}|du|_g^2g &=\lim_{k\to\infty}d\hat{u}_k^td\hat{u}_k-\frac{1}{2}|d\hat{u}_k|_{g_k}^2g_k\\
&=\lim_{k\to\infty}\left(du_k^tdu_k-\frac{1}{2}|du_k|_{g_k}^2g_k\right)\cdot {\bf 1}_{\Omega_k}\\
&= 0,
\end{split}
\end{equation*}
confirming that $u$ is weakly conformal--hence a branched, minimal immersion $u\colon M\to \mathbb{S}^n$ by $\lambda_1(M,[g],\mu_{\max})$-eigenfunctions, as desired.

\end{proof}

The varifold convergence statement of Theorem~\ref{thm:convergence} now follows by fairly standard arguments.


\begin{proposition}
Up to a choice of a subsequence, the $2$-varifolds $T_k$ associated to the branched free boundary minimal immersions $u_k:\Omega_k\to \mathbb{B}^{n+1}$ converge as varifolds
$$
T_k\rightharpoonup^*T
$$
to the varifold $T$ associated to the $\bar{\lambda}_1$-maximal map $u: M\to \mathbb{S}^n$.
\end{proposition}
\begin{proof}
It is well-known that strong $W^{1,2}$-convergence of maps $M\to \mathbb{R}^{n+1}$ from a closed surface $M$ implies convergence of the associated varifolds in $\mathbb{R}^{n+1}$; see e.g.~\cite[Section 3.6]{CM}, where a much stronger result is proved for maps from the sphere $M=\mathbb{S}^2$ (easily adapted to maps from any closed surface). Thus, as a consequence of Lemma \ref{lemma:strong_H1}, we see that the varifolds $\hat T_k$ associated to the maps $\hat{u}_k:M\to \mathbb{B}^{n+1}$ converge as varifolds $\hat T_k\rightharpoonup^* T$ to the varifold $T$ associated with the limiting $\bar{\lambda}_1$-maximal map $u: M\to \mathbb{S}^n$.

Moreover, since $E_g(\hat u_k; M\setminus\Omega_k)\to 0$ as $k\to\infty$, letting $T_k$ denote the varifolds associated to the free boundary (branched) minimal immersions $u_k:\Omega_k\to \mathbb{B}^{n+1}$, we see that, for any $f\in C^0(\mathcal{G}_2(n+1))$,
\begin{equation*}
\begin{split}
|\langle \hat{T}_k-T_k,f\rangle| &=\left|\int_{M\setminus \Omega_k}f(\hat{u}_k(x),d\hat{u}_k(T_xM)) J_{\hat u_k}(x) dv_g\right|\\
&\leq \|f\|_{C^0}\int_{M\setminus \Omega_k}\frac{1}{2}|d\hat{u}_k|_g^2dv_g\\
&\to 0\text{ as }k\to\infty.
\end{split}
\end{equation*}
Thus, the varifold limit of the sequence $T_k$ coincides with that of $\hat T_k$, giving us the desired convergence $T_k \rightharpoonup^* T$.

\end{proof}

\subsection{Small energy extension}
\label{sec:extension}
In this section we prove item (2) of Proposition~\ref{prop:extension} using items (1), (3) and Theorem~\ref{thm:uni}.
We first observe that by Theorem~\ref{thm:uni} we can assume that the complement of $\Omega_k$ is a collection of geodesic disks in the metric $g_k$, which we denote by 
$$
M\setminus\Omega_k = \bigcup_{j=1}^k B_{j,k},
$$
where $B_{j,k} = B_{r_{j,k}}(p_{j,k};g_k)$ is the disk of radius $r_{j,k}>0$ and center $p_{j,k}\in M$ in the metric $g_k$.
Recall moreover that item (3) of Proposition ~\ref{prop:extension} gives
\begin{equation}
\label{lim:glob_qual_stab3}
\|\mu_k - \mu_{\max}\|_{\Wdot(g)}\to 0.
\end{equation}
We argue now that the radii $r_{j,k}$ of the disks $B_{j,k}$ must vanish as $k\to\infty$, as does the contribution of each individual boundary component $\partial B_{j,k}$ to the total length $\mu_k(\partial\Omega_k)$.

\begin{lemma}
\label{lem:rad_length}
One has
\begin{equation}
\label{lim:radii}
\lim_{k\to\infty}\max_{1\leqslant j\leqslant k} r_{j,k} = 0
\end{equation}
and
\begin{equation}
\label{lim:length_bc}
\lim_{k\to\infty} \max_{1\leqslant j\leqslant k} \mu_k(\partial B_{j,k})= 0.
\end{equation}
\end{lemma}
\begin{proof}
We first observe that, given a sequence of radii $r_n$ and points $p_n\in M$ with $\mu_{\max}(B_{r_n}(p_n;g))\to 0$, one has $r_n\to 0$. Indeed, otherwise there exists a sequence $r_n\geqslant \rho>0$ and $p_n\to p\in M$ such that $\mu_{\max}(B_{r_n}(p_n;g))\to 0$. Then for large enough $n$ one has $B_{r_n}(p_n;g)\supset B_{\rho/2}(p;g)$. As a result, $\mu_{\max}(B_{\rho/2}(p;g))=0$, which contradicts the fact that $\mu_{\max} = dv_{g_{\max}}= f\,dv_g$, where $f\geqslant 0$ has only finitely many zeroes.

Let us now prove~\eqref{lim:radii}. Recall that $g_k\to g$ in $C^1(g)$, and therefore 
$$
B_{r_{j,k}/2}(p_{j,k};g)\subset B_{j,k}
$$
 for $k$ sufficiently large. Consider the Lipschitz functions
\begin{equation*}
f_{j,k}(x)=
\begin{cases}
1,\, &\text{ if  $\dist_{g}(x,p_{j,k})<\frac{r_{j,k}}{4}$};\\
0,\, &\text{ if $\dist_{g}(x,p_{j,k})>\frac{r_{j,k}}{2}$};\\
2-\dfrac{4\dist_{g}(x,p_{j,k})}{r_{j,k}},\,&\text{ otherwise}.
\end{cases}
\end{equation*}
Since the metric $g$ lies in the compact subset $Met_0(M)\subset Met_{can}(M)$, it is easy to see that $\|df_{j,k}\|_{L^2(g)}\leqslant C$, where $C$ depends only on the topology of $M$. At the same time, since $f_{j,k}$ vanishes on the support of $\mu_k$, one has
$$
\int f_{j,k}\,d(\mu_{\max}-\mu_k) \geqslant  \mu_{\max}(B_{r_{j,k}/4}(p_{j,k};g)).
$$
Combining these estimates one obtains 
\begin{equation*}
\begin{split}
0<\mu_{\max}(B_{r_{j,k}/4}(p_{j,k};g)) &\leqslant \int f_{j,k}\,d(\mu_{\max}-\mu_k)\\
&\leqslant \|\mu_k - \mu_{\max}\|_{\Wdot(g_k)}\|df_{j,k}\|_{L^2(g_k)}\\
&\leqslant C\|\mu_k-\mu_{\max}\|_{\Wdot(g_k)}\to 0
\end{split}
\end{equation*}
as $k\to\infty$, which implies~\eqref{lim:radii} by the observation at the beginning of the proof.

To prove~\eqref{lim:length_bc} note that 
$$
B_{j,k}\subset B_{2r_{j,k}}(p_{j,k};g)
$$
for large enough $k$. Consider the function
\begin{equation*}
\tilde f_{j,k}(x)=
\begin{cases}
1,\, &\text{ if  $\dist_{g}(x,p_{j,k})<2r_{j,k}$};\\
0,\, &\text{ if $\dist_{g}(x,p_{j,k})>4r_{j,k}$};\\
2-\dfrac{\dist_{g}(x,p_{j,k})}{2r_{j,k}},\,&\text{ otherwise}.
\end{cases}
\end{equation*}
Once again, it is easy to see that  $\|d\tilde f_{j,k}\|_{L^2(g_k)}\leqslant C$. Furthermore, $\mu_{\max} = f\,dv_g\leqslant C\,dv_{g}$ for large enough $k$, therefore,
$$
\left|\int \tilde f_{j,k}\,d(\mu_{\max}-\mu_k)\right| \geqslant \mu_k(\partial B_{j,k}) - C\area_{g}(B_{4r_{j,k}}(p_{j,k};g)).
$$
Thus,
$$
\mu_k(\partial B_{j,k}) - C\area_{g}(B_{4r_{j,k}}(p_{j,k};g))\leqslant C\|\mu_k - \mu_{\max}\|_{\Wdot(g_k)}\to 0,
$$
as $k\to \infty$, and since, by \eqref{lim:radii},
$$\max_j Area_g(B_{4r_{j,k}}(p_{j,k};g))\leqslant C' \max_j r_{j,k}^2\to 0$$
as $k\to\infty$, it follows that
$$
\limsup_{k\to\infty}\max_j \mu_k(\partial B_{j,k})\leqslant C\lim_{k\to\infty}\max_j\area_{g}(B_{4r_{j,k}}(p_{j,k};g))\to 0,
$$
as desired.

\end{proof}

Recall that if $u_k\colon \Omega_k\to \mathbb{B}^{n+1}$ is the $\bar\sigma_1$-maximal map corresponding to $\mu_k$, then the induced boundary length measure $ds^{\partial\Omega_k}_{u_k^*(g_{\mathbb{R}^{n+1}})}$ coincides with $\Sigma_1(N_k)\mu_k$. Thus, identity~\eqref{lim:length_bc} can be equivalently stated as
$$
\lim_{k\to\infty} \max_{1\leqslant j\leqslant k} \length(u_k(\partial B_{j,k}))= 0.
$$
\begin{lemma}
There exists an extension $\hat{u}_k\in W^{1,2}(M,\mathbb{B}^{n+1})$ of $\hat{u}_k\in W^{1,2}(\Omega_k,\mathbb{B}^{n+1})$ such that 
$$
\lim_{k\to\infty}\area(\hat{u}_k(M\setminus\Omega_k)) = 0.
$$
\end{lemma}
\begin{proof} 
The argument is standard, but we recall it here for completeness. Consider the curve 
$$
\gamma_{j,k}:= u_k|_{\partial B_{j,k}}: \partial B_{j,k}\to \mathbb{B}^{n+1},
$$
of length 
$$
l_{j,k} = \int_{\partial B_{j,k}}|\gamma_{j,k}'(s)|ds=\Sigma_1(N_k)\mu_k(\partial B_{j,k}).
$$
The diameter $\diam( C_{j,k})$ of the image $C_{j,k}=\gamma_{j,k}(\partial B_{j,k})$ clearly satisfies
$$
\diam(C_{j,k})\leq \frac{1}{2} l_{j,k},
$$
so there exists a point $z_{j,k}\in \mathbb{B}^{n+1}$ such that $C_{j,k}\subset B_{l_{j,k}/2}(z_{j,k})$ (indeed, one can take any $z_{j,k}\in C_{j,k}$). 

We then define the extension $\hat{u}_k$ on $B_{j,k}$ to be the cone over $C_{j,k}$ centered at $z_{j,k}$, i.e. in geodesic polar coordinates centered at $p_{j,k}$ we set
$$
\hat{u}_k(r,\theta) = z_{j,k} + \frac{r}{r_{j,k}}(u_k(r_{j,k},\theta)-z_{j,k}).
$$
Since $|u(r_{j,k},\theta) - z_{j,k}|\leqslant \dfrac{l_{j,k}}{2}$, one has
\begin{equation*}
\begin{split}
\area(\hat{u}_k(B_{j,k}))&\leqslant C\int_{0}^{2\pi}\int_0^{r_{j,k}} \frac{r}{r^2_{j,k}}|u(r_{j,k},\theta) - z_{j,k}| |u_\theta(r_{j,k},\theta)|\,dr\,d\theta \leqslant\\
&\leqslant \frac{Cl_{j,k}}{4}\int_0^{2\pi}|u_\theta(r_{j,k},\theta)|\,d\theta = \frac{Cl_{j,k}^2}{4}.
\end{split}
\end{equation*}
As a result,
$$
\area(\hat{u}_k(M\setminus\Omega_k))\leqslant \frac{C}{4}\sum_{j=1}^k l_{j,k}^2\leqslant \frac{C\Sigma_1(N_k)}{4}\mu_k(\partial\Omega_k)\max_{1\leqslant j\leqslant k}l_{j,k}\to 0,
$$
since $\mu_k$ is a probability measure and $\Sigma_1(N_k)\to\Lambda_1(M)$.
\end{proof}

The final obstacle is that the extensions constructed in the previous lemma could be far from being conformal and, thus, the area bound does not imply the energy bound. However, this can be easily remedied by changing the metric $g_k$ in the interior of the holes. 

\begin{lemma}
There exists a metric $\hat g_k$ on $M$ such that $\hat g_k = g_k$ on $\Omega_k$ and 
$$
\lim_{k\to\infty} E_{\hat g_k}(\hat u_k;M\setminus\Omega_k)=0
$$
\end{lemma}
\begin{proof}
Fix $j,k$. It is sufficient to construct a metric $h$ on $B_{j,k}$ such that $h = g_k$ near $\partial B_{j,k}$ and 
\begin{equation}
\label{ineq}
E_{h}(\hat u_k; B_{j,k})\leqslant C\left(\area(\hat u_k(B_{j,k})) + \frac{1}{k^2}\right).
\end{equation}
We construct a metric satisfying~\eqref{ineq} and then arrange it to agree with $g_k$ near the boundary. First, one can approximate $w_k\in \Lip(B_{j,k},\mathbb{R}^{n+1})$ by a smooth map $v$ (see e.g.~\cite[p. 251]{EG}) arbitrarily close in Lipschitz norm, which in turn can be approximated by a smooth \emph{immersion} to $\mathbb{R}^{n+1}\times\mathbb{R}^2$. Indeed, if $v\in C^\infty(B_{j,k},\mathbb{R}^{n+1})$, then $v_\eps(x):=(v(x),\eps x)$ is obviously an immersion for any $\eps>0$ (where $B_{j,k}$ is identified with the unit disk so that $B_{j,k}\subset \mathbb{R}^2$). Setting $h_0 = v_\eps^*g_{euc}$ for small enough $\eps>0$, we obtain a metric satisfying~\eqref{ineq}.

Let $\delta>0$ be such that  
$$
E_{g_k}(\hat u_k;B_{j,k}\setminus B_{r_{j,k}-\delta}(p_{j,k}))\leqslant \frac{1}{k^2}.
$$
Define a discontinuous metric $h_1$ to be $g_k$ on $B_{j,k}\setminus B_{r_{j,k}-\delta}(p_{j,k})$ and $h_0$ otherwise, then $h_1$ satisfies the requirements of the claim. Then a suitable mollification of $h_1$ yields the desired smooth metric $h$.

\end{proof}

As a final step of the proof we apply the uniformization theorem for closed surfaces to the pair $(M,\hat g_k)$ to replace a smooth metric $\hat g_k$ by $\tilde g_k\in\Met_{\can}(M)$. This completes the proof, up to a slight abuse of notation $\tilde g_k\mapsto g_k$.

\subsection{Proof of Theorem~\ref{thm:convergence_conf}} The proof of Theorem~\ref{thm:convergence_conf} follows the same ideas, but is substantially simpler. We outline the main steps. One applies the conformal qualitative stability of Theorem~\ref{thm:conf_qual_stab} to obtain an analogue of Proposition~\ref{prop:extension}. However, since the small energy extension is an assumption~\eqref{cond:extension} of Theorem~\ref{thm:convergence_conf}, we do not need to prove item (2). After that, using the same arguments as in Lemma~\ref{lemma:strong_H1} with only minor modifications completes the proof.

\section{Lower bounds for $\Sigma_1(N_k)$}
\label{l.bds}

In this section, we prove the lower bound \eqref{ineq:lowerbound_intro} for $\Sigma_1(N_k)$ given in Theorem \ref{thm:asymptotics_intro}, by producing a metric on $N_k$ whose first normalized Steklov eigenvalue is within $C\frac{\log k}{k}$ of the maximal normalized Laplacian eigenvalue $\Lambda_1(M)$. Indeed, we prove the following slightly stronger, conformally constrained version of this result.

\begin{proposition}\label{construct} Let $(M,g_0)$ be a closed surface of unit area and constant curvature $K_{g_0}\equiv 2\pi\chi(M)$. Then there exists $C=C(M,g_0)$ such that for any $k\in \mathbb{N}$, we can find a collection of disjoint geodesic disks $B_{r_1}(p_1),\ldots B_{r_k}(p_k)$ in the metric $g_0$ and a conformal metric $\tilde{g}_k\in [g_0]$ such that on the domain
$$
\Omega_k:=M\setminus \bigcup_{j=1}^kB_{r_j}(p_j),
$$
we have
$$
\bar{\sigma}_1(\Omega_k,\tilde{g}_k)\geq \Lambda_1(M,[g_0])-C\frac{\log k}{k}.
$$
\end{proposition}

Note that the lower bound \eqref{ineq:lowerbound_intro} of Theorem \ref{thm:asymptotics_intro} is an immediate corollary of Proposition \ref{construct}, simply by taking $[g_0]$ to be a maximizing conformal class, so that $\Lambda_1(M,[g_0])=\Lambda_1(M)$. The remainder of this section is therefore devoted to the proof of this proposition.

To begin, fix a conformally $\bar{\lambda}_1$-maximizing metric of unit area in $[g_0]$ 
$$
g_{\max}=fg_0.
$$
Recall from Section \ref{sec:Laplace_extremal} that $f$ has the form $f=\alpha|du|_{g_0}^2$, where $u\colon (M,g_0)\to \mathbb{S}^n$ is a nonconstant harmonic map by first eigenfunctions for $\Delta_{g_{\max}}$. In particular, while the metric $g_{\max}$ may have conical singularities at the zeroes of $du$, the function $f$ is a smooth, nonnegative function, with a finite zero set. Denote by $V\subset C^{\infty}(M)$ the space of first eigenfunctions for $\Delta_{g_{\max}}$-i.e., 
$$
V:=\{\phi\in C^{\infty}(M)\mid \Delta_{g_0}\phi=\Lambda_1f\phi\};
$$
since $f$ is smooth, standard elliptic estimates then give
\begin{equation}\label{vests}
\|\phi\|_{C^2}\leq C(M,g_0)\|d\phi\|_{L^2}\text{ for all }\phi\in V.
\end{equation}

\begin{remark} Throughout this section, all function spaces and associated norms ($W^{1,2}$, $L^p$, $C^k$, etc) will be defined with respect to the constant curvature metric $g_0$ unless otherwise indicated.
\end{remark}

Before beginning the proof of Proposition \ref{construct} in earnest, we find it useful to record the following elementary estimates for the areas and boundary lengths of geodesic disks.

\begin{lemma}
\label{std.disk.lem} 
Let $(M,g_0)$ be a closed surface of unit area and constant curvature $K_{g_0}\equiv 2\pi \chi(M)$, and injectivity radius 
$\inj(M)$. Then there is a constant $r_0(\chi(M))>0$ such that for any geodesic disk $B_r(x)\subset M$ with $r<\min\{r_0,\inj(M)\}$, we have 
\begin{equation}\label{std.length.est}
\frac{3}{4}\cdot 2\pi r\leq \length(\partial B_r(x))\leq \frac{5}{4}\cdot 2\pi r
\end{equation}
and
\begin{equation}\label{std.area.est}
\frac{3}{4}\pi r^2\leq \area(B_r(x))\leq \frac{5}{4}\pi r^2.
\end{equation}
\end{lemma}
\begin{proof} Since $(M,g_0)$ has curvature $K_{g_0}\equiv 2\pi \chi(M)$, standard computations (e.g., applying the Gauss-Bonnet formula to geodesic disks) show that the length function
$$
L(t):=\length(\partial B_t(x))
$$
satisfies the equation
$$
L''(t)+2\pi \chi(M)L(t)=0
$$
when $t<\inj(M)$, with $L(0)=0$ and $L'(0)=2\pi$. In particular, for $t<\inj(M)$, it follows that
$$
L(t)=\sqrt{\frac{2\pi}{|\chi(M)|}}\sinh\left(\sqrt{2\pi|\chi(M)|}t\right)\text{ when }\chi(M)<0,
$$
$$
L(t)=2\pi t\text{ if }\chi(M)=0,
$$
and
$$
L(t)=\sqrt{\frac{2\pi}{\chi(M)}}\sin\left(\sqrt{2\pi\chi(M)}t\right)\text{ when }\chi(M)>0.
$$
The estimate \eqref{std.length.est} for $r<\min\{r_0(\chi(M)),\inj(M)\}$ follows by direct inspection of these functions. Likewise, since $\frac{d}{dr}\area(B_r(x))=L(r)$ for $r<\inj(M)$, the estimate \eqref{std.area.est} follows by integation of \eqref{std.length.est}.
\end{proof}

\subsection{Choosing the domain $\Omega_k$}

To prove Proposition \eqref{construct}, we first select the desired domain $\Omega_k$ in a manner similar to the construction in \cite{GL}, by removing several small disks centered at a collection of $k$ maximally separated points. 

\begin{lemma}\label{pointslem} For $k\geq k_0(M,g_0)$ sufficiently large, there exist points $p_1,\ldots, p_k\in M$ and universal constants $0<c_0\leq C_0<\infty$ such that
$$
\dist(p_i,p_j)\geq \frac{c_0}{\sqrt{k}}\text{ when }i\neq j,
$$
and
$$
M\subset \bigcup_{j=1}^kB_{C_0/\sqrt{k}}(p_j).
$$
\end{lemma}
\begin{proof} The proof follows elementary covering arguments, but we give it here for completeness. Given 
$$0<R<\frac{1}{2}\min\{r_0(\chi(M)),\inj(M)\},$$
let $B_R(x_1),\ldots, B_R(x_{\ell(R)})$ be a maximal disjoint collection of disks of radius $R$. By maximality, we see that
$$
M\subset \bigcup_{j=1}^{\ell(R)}B_{2R}(x_j),
$$
and since--by Lemma \ref{std.disk.lem}--the area of a geodesic disk $B_t(x)$ with $t<\min\{r_0,\inj(M)\}$ satisfies
\begin{equation}
\frac{3\pi}{4}t^2\leq \area(B_t(x))\leq \frac{5\pi}{4}t^2,
\end{equation}
it follows that
\begin{eqnarray*}
1=\area(M,g_0)&\leq &\sum_{j=1}^{\ell(R)}\area(B_{2R}(x_j))\\
&\leq &\ell(R)\frac{5\pi}{4} 4R^2=\ell(R) 5\pi R^2.
\end{eqnarray*}
In particular, taking $R_k=\frac{1}{\sqrt{5\pi k}}$ for $k\geq k_0(M,g_0)$ sufficiently large, we see that
$$
\ell_k:=\ell(R_k)\geq k.
$$

For each $T\in (R_k,\min\{r_0,\inj(M)\})$, let $S_T\subset \{x_1,\ldots,x_{\ell_k}\}$ be a maximal subcollection such that $\{B_T(x_j)\mid x_j\in S_T\}$ is disjoint. It follows from disjointness and Lemma \ref{std.disk.lem} that
\begin{eqnarray*}
1=\area(M,g_0)&\geq &\sum_{x\in S_T}\area(B_T(x))\\
&\geq &|S_T|\cdot \frac{3\pi}{4}T^2,
\end{eqnarray*}
so that the number of points $|S_T|$ in $S_T$ is bounded above by $\frac{4}{3\pi T^2}$. In particular, taking $T_k=\frac{2}{\sqrt{3\pi k}}$, we have
$$
m_k:=|S_{T_k}|\leq k.
$$
Writing $S_{T_k}=\{x_1,\ldots,x_{m_k}\}\subset\{x_1,\ldots,x_{\ell_k}\}$, note that the maximality in the definition of $S_{T_k}$ implies that $\dist(x_j,S_{T_k})\leq 2T_k$ for all $1\leq j\leq \ell_k$, and consequently
$$
\bigcup_{j=1}^{m_k}B_{2T_k+2R_k}(x_j)\supset \bigcup_{j=1}^{\ell_k}B_{2R_k}(x_j)\supset M.
$$
Thus, since $m_k\leq k\leq \ell_k$, we can arbitrarily extend $S_{T_k}\subset S_{R_k}$ to a set of $k$ points
$$
S_{T_k}\subset \{x_1,\ldots,x_k\}\subset \{x_1,\ldots,x_{\ell_k}\},
$$
which necessarily satisfy
$$
\dist(x_i,x_j)\geq 2R_k=\frac{2}{\sqrt{5\pi k}}\quad\text{ for }\quad i\neq j
$$
and
$$
M\subset \bigcup_{j=1}^kB_{2(R_k+T_k)}(x_j)\subset \bigcup_{j=1}^kB_{8/\sqrt{3\pi k}}(x_j),
$$
so that the conclusion of the lemma is satisfied by $\{p_1,\ldots,p_k\}=\{x_1,\ldots,x_k\}$ with $c_0=\frac{1}{\sqrt{5\pi}}$ and $C_0=\frac{8}{\sqrt{3\pi}}$.

\end{proof}

Now, fix a collection of points $\{p_1,\ldots, p_k\}$ satisfying the conclusions of Lemma~\ref{pointslem}. Since $\dist(p_i,p_j)\geq \frac{c_0}{\sqrt{k}}>4k^{-3/2}$ for $k\geq k_0$, the disks $B_{2k^{-3/2}}(p_1),\ldots, B_{2k^{-3/2}}(p_k)$ are disjoint, and we can consider the domain
\begin{equation}\label{lbsec.omegadef}
\Omega_k:=M\setminus \bigcup_{j=1}^k B_{k^{-3/2}}(p_j).
\end{equation}
In what follows, we will make use of the following simple lemma, stating that the norm of the harmonic extension operator $W^{1,2}(\Omega_k)\to W^{1,2}(M\setminus \Omega_k)$ is bounded independent of $k$. 

\begin{lemma}\label{ext.lem} There exists a constant $C_1<\infty$ such that for $k\geq k_0$ ,any $\chi\in W^{1,2}(\Omega_k)$, the harmonic extension $\hat{\chi}\in W^{1,2}(M\setminus \Omega_k)$ to $M\setminus \Omega_k$ satisfies
\begin{equation}
\|d\hat{\chi}\|_{L^2(M\setminus\Omega_k)}\leq C_1\|d\chi\|_{L^2(\Omega_k)}.
\end{equation} 
\end{lemma}

\begin{proof} Denote by $C_0$ the constant such that the harmonic extension operator 
$$
W^{1,2}(D_2(0)\setminus D_1(0))\ni \chi\mapsto \hat{\chi}\in W^{1,2}(D_1(0))
$$
from the Euclidean annulus $D_2(0)\setminus D_1(0)$ in $\mathbb{R}^2$ to $D_1(0)$ satisfies
$$
\|d\hat{\chi}\|_{L^2(D_1)}\leq C_0\|d\chi\|_{L^2(D_2\setminus D_1)}.
$$
By the conformal invariance of the Dirichlet energy in dimension two, it follows that
\begin{equation}\label{euc.bds}
\|d\hat{\chi}\|_{L^2(D_r(0))}\leq C_0\|d\chi\|_{L^2(D_{2r}(0)\setminus D_r(0))}
\end{equation}
for any $r>0$. 

Using the exponential map, it is straightforward to extend this inequality to $(M,g_0)$ for $r>0$ small, see e.g.~\cite[Lemma 3.4]{GL}. The desired statement then follows from the fact that the balls $B_{2k^{-3/2}}(p_j)$ are disjoint. 
\end{proof}

\begin{remark} The choice of $k^{-3/2}$ as the radius of the holes in the definition \eqref{lbsec.omegadef} is somewhat arbitrary. Any sufficiently large (fixed) power of $\frac{1}{k}$ would suffice, and determining the optimal such power (perhaps $k^{-1}$, as in \cite{GL}) could be an important step toward answering Open Question \ref{oq:limit} in the introduction.
\end{remark}

\subsection{Choosing the metric $\tilde{g}_k$}

Having chosen our domain $\Omega_k\subset M$, we next need to produce a conformal metric $\tilde{g}_k$ on $\Omega_k$ whose first Steklov eigenvalue is close to $\Lambda_1(M,[g_0])$. To this end, define $\psi_k\in C^{\infty}(\Omega_k)$ to be the unique solution of
\begin{equation}\label{psidef}
\Delta_{g_0}\psi_k=d^*d\psi_k=-f\text{ in }\Omega_k\text{\hspace{2mm} and\hspace{2mm} }\psi_k|_{\partial\Omega_k}=0,
\end{equation}
where $f$ again denotes the unit-area $\hat{\lambda}_1$-maximizing conformal factor $g_{\max}=fg_0$. Setting
$$
\beta_k:=\frac{\partial \psi_k}{\partial\nu} \in C^{\infty}(\partial\Omega_k),
$$
we then see that
\begin{equation}\label{betachar}
\int_{\partial\Omega_k}\varphi \beta_kds_{g_0}=\int_{\Omega_k}\hat{\varphi}fdv_{g_0}g,
\end{equation}
for all $\varphi\in C^{\infty}(\partial\Omega_k)$ with harmonic extension $\hat{\varphi}\in C^{\infty}(\Omega_k)$. 

Applying \eqref{betachar} to the collection of all nonnegative, nonvanishing functions $\varphi\in C^{\infty}(\partial\Omega_k)$, one deduces that
$$
\beta_k>0\text{ on }\partial\Omega_k,
$$
and applying \eqref{betachar} with $\varphi=1$ gives
\begin{eqnarray*}
\int_{\partial\Omega_k}\beta_kds_{g_0}&=&\int_{\Omega_k}fdv_{g_0}\\
&=&\area(M,g_{\max})-\sum_{j=1}^k\int_{B_{k^{-3/2}}(p_j)}fdv_{g_0}\\
&\geq & 1-C(M,g_0)k\cdot k^{-\frac{2\cdot3}{2}},
\end{eqnarray*}
so that
\begin{equation}\label{beta.mass}
1=\area(M,g_{\max})\geq \int_{\partial\Omega_k}\beta_kds_g\geq 1-\frac{C}{k^2}.
\end{equation}
In what follows, we argue that the first Steklov eigenvalue of $(\Omega_k, \hat{\beta^2_k}g_0)$ must lie within $O\left(\frac{\log k}{k}\right)$ of $\Lambda_1$, where $\hat{\beta_k}$ is an arbitrary extension of $\beta_k\in C^{\infty}(\partial\Omega_k)$ to $\Omega_k$. As a first step, we record the following $L^2$ estimate for the function $\psi_k$ solving \eqref{psidef}.

\begin{lemma}\label{psi.est}
Let $\psi_k\in C^{\infty}(\Omega_k)$ be the unique solution of \eqref{psidef}. Then 
\begin{equation}
\|\psi_k\|_{L^2(\Omega_k)}\leq C(M,g_0)\frac{\log k}{k}.
\end{equation}
\end{lemma}
\begin{proof}

For each $i=1,\ldots,k$, let 
$$U_i:=B_{Ck^{-1/2}}(p_i)\setminus B_{k^{-3/2}}(p_i),$$
so that
$$\Omega_k\subset \bigcup_{i=1}^kU_i.$$
For any $x\in M$, we note that the quantity
$$N(x):=\#\{i\in \{1,\ldots,k\}\mid x\in U_i\}$$
satisfies a uniform bound
\begin{equation}\label{loc.fin.est}
N(x)\leq N_0
\end{equation}
independent of $k$. Indeed, if 
$$x\in U_1\cap \cdots\cap U_N\subset B_{Ck^{-1/2}}(p_1)\cap \cdots\cap B_{Ck^{-1/2}}(p_k),$$
then since Lemma \ref{pointslem} guarantees $\dist(p_i,p_j)\geq \frac{c_0}{\sqrt{k}}$ for $i\neq j$, it follows that the disk $B_{(C+c_0)k^{-1/2}}(x)$ contains at least $N=N(x)$ disjoint disks of radius $\frac{c_0}{2\sqrt{k}}$, so that
$$N(x)\frac{c_0^2}{k}\leq C'\frac{(C+c_0)^2}{k},$$
from which \eqref{loc.fin.est} follows.

Now, since $\psi_k$ vanishes on $\partial\Omega_k$, we may trivially extend $\psi_k$ to a function $\overline{\psi}_k$ in $\Lip(M)$ by setting 
$$\overline{\psi}_k|_{M\setminus\Omega_k}\equiv 0.$$
Now, for each $i=1,\ldots,k$, define a function $\xi_i\in \Lip(B_{Ck^{-1/2}}(0)\setminus B_{k^{-3/2}}(0))$ by setting
$$\xi_i:=\overline{\psi}_k\circ \exp_{p_i},$$
and note that
\begin{equation}\label{psi.xi.comp}
\|\overline{\psi}_k\|_{L^2(U_i)}\leq C\|\xi_i\|_{L^2}\text{, }\|d\xi_i\|_{L^2}\leq C\|d\overline{\psi}_k\|_{L^2(U_i)}
\end{equation}
for a constant $C=C(M,g_0)$ (which could be taken arbitrarily close to $1$ for $k$ sufficiently large).

Then direct computation (which the squeamish reader may prefer to apply to a smooth approximation of $\xi_i$) gives, for $t\in \left[k^{-3/2},Ck^{-1/2}\right]$,
\begin{eqnarray*}
\frac{d}{dt}\left(\frac{1}{t}\int_{\partial B_t(0)}\xi_i^2\right)&=&\frac{2}{t}\int_{\partial B_t(0)}\xi_i\frac{\partial \xi_i}{\partial\nu}\\
\text{(by Cauchy-Schwarz) }&\leq &2\left(\frac{1}{t}\int_{\partial B_t(0)}\xi_i^2\right)^{1/2}\left(\frac{1}{t}\int_{\partial B_t(0)}|d\xi_i|^2\right)^{1/2},
\end{eqnarray*}
or equivalently,
$$\frac{d}{dt}\left(\frac{1}{t}\int_{\partial B_t(0)}\xi_i^2\right)^{1/2}\leq \left(\frac{1}{t}\int_{\partial B_t(0)}|d\xi_i|^2\right)^{1/2}.$$
Since $\xi_i|_{\partial B_{k^{-3/2}}(0)}\equiv 0$, integrating the above over $t\in \left[k^{-3/2},s\right]$ gives
\begin{eqnarray*}
\left(\frac{1}{s}\int_{\partial B_s(0)}\xi_i^2\right)^{1/2}&\leq &\int_{k^{-3/2}}^s t^{-1/2}\left(\int_{\partial B_t(0)}|d\xi_i|^2\right)^{1/2}dt\\
\text{(by Cauchy-Schwarz) }&\leq &\left(\int_{k^{-3/2}}^s\frac{1}{t}dt\right)^{1/2}\left(\int_{k^{-3/2}}^s\int_{\partial B_t(0)}|d\xi_i|^2dt\right)^{1/2}\\
&=&\sqrt{\log(s/k^{-3/2})}\left(\int_{B_s(0)\setminus B_{k^{-3/2}}(0)}|d\xi_i|^2\right)^{1/2},
\end{eqnarray*}
which we can rearrange to see that
\begin{eqnarray*}
\int_{\partial B_s(0)}\xi_i^2&\leq & s \log(s k^{3/2})\int_{B_s(0)\setminus B_{k^{-3/2}}(0)}|d\xi_i|^2\\
&\leq &s\log(k)\int_{B_{Ck^{-1/2}}(0)\setminus B_{k^{-3/2}}(0)}|d\xi_i|^2.
\end{eqnarray*}
Integrating once more over $s\in \left[k^{-3/2},Ck^{-1/2}\right]$ gives
$$
\int_{B_{Ck^{-1/2}}(0)\setminus B_{k^{-3/2}}(0)}\xi_i^2\leq\frac{C^2\log(k)}{2k}\int_{B_{Ck^{-1/2}}(0)\setminus B_{k^{-3/2}}(0)}|d\xi_i|^2,
$$
and using \eqref{psi.xi.comp}, it follows that
$$
\int_{U_i}\overline{\psi}_k^2\,dv_{g_0}\leq C(M,g_0)\frac{\log k}{k}\int_{U_i}|d\overline{\psi}_k|_{g_0}^2\,dv_{g_0}.
$$
Summing over $i=1,\ldots,k$ and applying \eqref{loc.fin.est}, we then conclude that
\begin{equation}\label{l2ext.est}
\int_{\Omega_k}\psi_k^2dv_{g_0}\leq C\frac{\log k}{k}\int_{\Omega_k}|d\psi_k|^2_{g_0}dv_{g_0}.
\end{equation}

On the other hand, by \eqref{psidef} (and the fact that $\|f\|_{C^2}\leq C(M,g_0)$), we have
$$
\int_{\Omega_k}|d\psi_k|^2\,dv_{g_0}=-\int_{\Omega_k}f\psi_k\,dv_{g_0}\leq C(M,g_0)\|\psi_k\|_{L^2(\Omega_k)},
$$
so that
$$
\|\psi_k\|_{L^2(\Omega_k)}^2\leq C\frac{\log k}{k}\int_{\Omega_k}|d\psi_k|_{g_0}^2\,dv_{g_0}\leq C'\frac{\log k}{k}\|\psi_k\|_{L^2(\Omega_k)},
$$
from which the desired estimate follows.
\end{proof}

\begin{remark}\label{ext.rk} Note that the proof of the preliminary estimate \eqref{l2ext.est} does not require that $\psi_k$ is a solution of \eqref{psidef}, and only uses that $\psi_k\in W_0^{1,2}(\Omega_k)$. Thus, we see that
$$
\int_{\Omega_k}\varphi^2\,dv_{g_0}\leq C \frac{\log k}{k}\int_{\Omega_k}|d\varphi|^2_{g_0}\,dv_{g_0}
$$
holds for any $\varphi\in W_0^{1,2}(\Omega_k)$. 
\end{remark}

In the proof of Proposition \ref{construct}, the following (non-sharp) $L^{\infty}$ estimate for $\psi_k$ will also be useful.

\begin{lemma}\label{psi.linfty} The function $\psi_k\in C^{\infty}(\Omega_k)$ given by \eqref{psidef} satisfies a bound of the form
$$\|\psi_k\|_{L^{\infty}}\leq C\sqrt{\frac{\log k}{k}}$$
for some constant $C=C(M,g_0)$.
\end{lemma}

\begin{proof} We proceed by a Moser iteration-type argument. For any $\varphi\in W_0^{1,2}(\Omega_k)$, we may extend $\varphi$ to all of $M$ by setting $\varphi\equiv 0$ on $M\setminus\Omega_k$, and apply the Sobolev embedding theorem for $W^{1,1}(M,g_0)\to L^2(M,g_0)$ to the square $\varphi^2$ to see that
\begin{eqnarray*}
\|\varphi^2\|_{L^2}&\leq & C\|\varphi^2\|_{L^1}+C\|d(\varphi^2)\|_{L^1}\\
&=&C\int_M\varphi^2+2C\int|\varphi||d\varphi|\\
&\leq & C\left(\|\varphi\|_{L^2}^2+\|\varphi\|_{L^2}\|d\varphi\|_{L^2}\right).
\end{eqnarray*}
Moreover, using Remark \ref{ext.rk} to bound the $\|\varphi\|_{L^2}$ terms, it follows that
\begin{equation}\label{omega.sob}
\|\varphi\|_{L^4}^2\leq C\sqrt{\frac{\log k}{k}}\|d\varphi\|^2_{L^2}
\end{equation}
for all $\varphi\in W_0^{1,2}(\Omega_k)$.

Now, for each integer $p\geq 1$, recalling that $\Delta |\psi_k|=-\Delta\psi_k=f$, we compute
\begin{equation*}
\begin{split}
\int_{\Omega_k}|d(\psi_k^p)|^2 & =\int_{\Omega_k}p^2|\psi_k|^{2p-2}|d\psi_k|^2
=\frac{p^2}{2p-1}\int_{\Omega_k}\langle d\left(|\psi_k|^{2p-1}\right),d\psi_k\rangle=\\
&=\frac{p^2}{2p-1}\int_{\Omega_k}|\psi_k|^{2p-1}\Delta|\psi_k|
=\frac{p^2}{2p-1}\int_{\Omega_k}f|\psi_k|^{2p-1}
\leq  \\
&\leqslant C\frac{p^2}{2p-1}\|\psi_k\|_{L^{2p}}^{2p-1}\area(M)^{\frac{1}{2p}}
\end{split}
\end{equation*}
which together with \eqref{omega.sob} (taking $\varphi=\psi_k^p$) yields
\begin{equation}\label{iter.est}
\|\psi_k\|_{L^{4p}}^{2p}\leq C_0p\sqrt{\frac{\log k}{k}}\|\psi_k\|_{L^{2p}}^{2p-1}.
\end{equation}
Next, set
$$
q_0:=\sup\left\{q\in (1,\infty) \mid \|\psi_k\|_{L^q}\leq \sqrt{\frac{\log k}{k}}\right\};
$$
we know from Lemma \ref{psi.est} that $q_0>2$ (for $k\geq k_0(M,g_0)$ sufficiently large), and if $q_0=\infty$, then it follows that $\|\psi_k\|_{L^{\infty}}=\lim_{q\to \infty}\|\psi_k\|_{L^q}\leq \sqrt{\frac{\log k}{k}}$, giving the desired estimate. Thus, we can assume without loss of generality that there is a finite $q_0\in [2,\infty)$ such that
$$
\|\psi_k\|_{L^{q_0}}=\sqrt{\frac{\log k}{k}}\text{ and }\|\psi_k\|_{L^q}>\sqrt{\frac{\log k}{k}}\text{ for all }q>q_0.
$$
Now, taking $p=q/2$ for $q\geq q_0$ in \eqref{iter.est} gives
$$
\|\psi_k\|_{L^{2q}}^q\leq \frac{C_0q}{2}\sqrt{\frac{\log k}{k}}\|\psi_k\|_{L^q}^{q-1}\leq \frac{C_0q}{2}\|\psi_k\|_{L^q}^q.
$$
In particular, taking the $q$-th root of both sides gives
\begin{equation}
\|\psi_k\|_{L^{2q}}\leq (C_0q/2)^{1/q}\|\psi_k\|_{L^q}
\end{equation}
for all $q\geq q_0$, and the standard iteration argument starting at $q=q_0$ then gives, for all $j\in \mathbb{N}$
\begin{equation}\label{moser.est}
\|\psi_k\|_{L^{2^{\ell}q_0}}\leq \left(\prod_{j=0}^{\ell-1}\left(C_02^{j-1}q_0\right)^{\frac{1}{2^jq_0}}\right)\|\psi_k\|_{L^{q_0}}.
\end{equation}
Taking the logarithm of the product term on the right-hand side of \eqref{moser.est}, we see that
\begin{equation*}
\begin{split}
\log & \left(\Pi_{j=0}^{\ell-1}[C_02^{j-1}q_0]^{\frac{1}{2^jq_0}}\right)=\sum_{j=0}^{\ell-1}\frac{\log(C_0)+\log(q_0)+(j-1)\log(2)}{2^jq_0}\leqslant\\
&\leq \frac{1}{q_0}(\log(C_0)+\log(q_0))\sum_{j=0}^{\infty}\frac{1}{2^j}+\frac{\log(2)}{q_0}\sum_{j=0}^{\infty}\frac{j-1}{2^j}
=\\
&=\frac{2\left(\log(C_0)+\log(q_0)\right)}{q_0}\leq C_1,
\end{split}
\end{equation*}
where in the final inequality we used the fact that $q_0\geq 2$ and the boundedness of $\frac{\log x}{x}$ over $1\leq x<\infty$. Returning to \eqref{moser.est}, it follows that
$$
\|\psi_k\|_{L^{2^{\ell}q_0}}\leq e^{C_1}\|\psi_k\|_{L^{q_0}}=e^{C_1}\sqrt{\frac{\log k}{k}},
$$
and taking $\ell\to\infty$ yields $\|\psi_k\|_{L^{\infty}}\leq e^{C_1}\sqrt{\frac{\log k}{k}}$, as desired.

\end{proof}

\subsection{Proof of Proposition \ref{construct}}

With Lemma \ref{psi.est} in place, we next show that the restriction to $\partial\Omega_k$ of the first $\Delta_{g_{\max}}$-eigenfunctions
$$
V:=\left\{\phi\in C^{\infty}(M)\mid \Delta_{g_0}\phi=\Lambda_1f\phi\right\}
$$
are $O\left(\frac{\log k}{k}\right)$-quasimodes of the Dirichlet-to-Neumann operator on $(\Omega_k,\tilde{g}_k)$ with eigenvalue $\Lambda_1$, and use this to deduce the existence of at least $\dim(V)$ Steklov eigenvalues on $(\Omega_k,\tilde{g}_k)$ in the range $\left[\Lambda_1-C\frac{\log k}{k},\Lambda_1+C\frac{\log k}{k}\right]$.

For convenience we consider the norm adapted to the Steklov problem on $(\Omega_k,\tilde g_k)$: for any harmonic $\chi\in W^{1,2}(\Omega_k)$, we set
$$
\|\chi\|_{\mL_k}:=\|\chi\|_{L^2(\partial\Omega_k,\tilde g_k)} + \|d\chi\|_{L^2(\Omega_k,g_0)}.
$$

\begin{lemma}\label{quasimode.lem} For any $\phi\in V$ and any harmonic function $\chi\in W^{1,2}(\Omega_k)$, we have
\begin{equation}
\label{phiquasi}
\left|\int_{\Omega_k}\langle d\phi,d\chi\rangle-\Lambda_1\int_{\partial\Omega_k}\beta_k \chi\phi\right|\leq C\frac{\log k}{k}\|d\phi\|_{L^2(M)}\|\chi\|_{\mL_k}.
\end{equation}
Moreover, for any $\phi\in V$, we have
\begin{equation}\label{phi.quot}
\|d\phi\|_{L^2(M)}^2\leq \left(\Lambda_1+C\frac{\log k}{k}\right)\int_{\partial\Omega_k}\beta_k\phi^2,
\end{equation}

In particular \eqref{phiquasi} can be rewritten as
\begin{equation}
\label{phiquasi.2}
\left|\int_{\Omega_k}\langle d\phi,d\chi\rangle-\Lambda_1\int_{\partial\Omega_k}\beta_k \chi\phi\right|\leq C'\frac{\log k}{k}\|\phi\|_{L^2(\partial\Omega_k,\tilde{g}_k)}\|\chi\|_{\mL_k}.
\end{equation}
\end{lemma}
\begin{proof} Given $\phi\in V$, let $\chi\in W^{1,2}(\Omega_k)$ be harmonic, and let $\hat{\chi}\in W^{1,2}(M)$ be the harmonic extension to $M\setminus\Omega_k$; recall that
\begin{equation}
\label{chi.ext.bd}
\|\hat{\chi}\|_{W^{1,2}(M)}\leq C\|\chi\|_{W^{1,2}(\Omega_k)},
\end{equation}
by Lemma \ref{ext.lem}. Next, note that
\begin{equation*}
\begin{split}
&\left|\int_{\Omega_k}\left(\langle d\phi,d\chi\rangle -\Lambda_1f\phi\chi\right)\right|=\left|\int_{\Omega_k}\mathrm{div}(\chi d\phi)\right|
=\left|\int_{\partial\Omega_k}\chi\frac{\partial\phi}{\partial\nu}\right|=\\
&=\left|\int_{M\setminus \Omega_k}\left(\langle d\hat{\chi},d\phi\rangle-\Lambda_1f \phi\hat{\chi}\right)\right|
\leq \|\phi\|_{C^1}|M\setminus \Omega_k|^{1/2}\|\hat{\chi}\|_{W^{1,2}(M)}.
\end{split}
\end{equation*}
In particular, by \eqref{vests}, \eqref{chi.ext.bd}, and the fact that
$$
\area(M\setminus \Omega_k)\leq \sum_{j=1}^k\area(B_{k^{-3/2}}(p_j))\leq C k^{1-\frac{2\cdot 3}{2}}=\frac{C}{k^2},
$$
it follows that
\begin{equation}\label{comp.est}
\left|\int_{\Omega_k}\left(\langle d\phi,d\chi\rangle-\Lambda_1f\phi \chi\right)\right|\leq \frac{C}{k} \|d\phi\|_{L^2}\|\chi\|_{W^{1,2}(\Omega_k)}.
\end{equation}
Next, recalling the definition \eqref{psidef} of $\psi_k$ and $\beta_k=\frac{\partial\psi_k}{\partial\nu}$, and keeping in mind that $\chi$ is harmonic, we compute
\begin{eqnarray*}
\int_{\partial\Omega_k}\beta_k \phi\chi&=&\int_{\partial\Omega_k}\frac{\partial\psi_k}{\partial\nu}\phi\chi\\
&=&\int_{\Omega_k}(-\Delta\psi_k)\phi\chi+\int_{\Omega_k}\langle d\psi_k,d(\phi\chi)\ra\\
&=&\int_{\Omega_k}f\phi \chi+\int_{\Omega_k}\psi_k \Delta(\phi\chi)\\
&=&\int_{\Omega_k}f\phi\chi+\int_{\Omega_k}\psi_k\left(f\phi\chi-2\langle d\phi,d\chi\rangle\right),
\end{eqnarray*}
so that
\begin{equation*}
\begin{split}
&\left|\int_{\partial\Omega_k}\beta_k\phi\chi-\int_{\Omega_k}f\phi\chi\right|\leq \left|\int_{\Omega_k}\psi_k\left(f\phi\chi-2\langle d\phi,d\chi\rangle\right)\right|\leq\\
&\leq C\|\psi_k\|_{L^2(\Omega_k)}\|\phi\|_{C^1}\left(\|d\chi\|_{L^2(\Omega_k)} + \|\chi\|_{L^2(\Omega_k,g_{\max})}\right).
\end{split}
\end{equation*}
To relate the right hand side to the $\mL_k$-norm, we write
\begin{equation*}
\begin{split}
 \|\chi\|^2_{L^2(\Omega_k,g_{\max})} = \int _{\Omega_k}f\chi^2 = -\int_{\Omega_k}\chi^2\Delta\psi_k = \int_{\partial\Omega_k}\beta_k\chi^2 + 2\int_{\Omega_k}\psi_k|d\chi|^2.
\end{split}
\end{equation*}
As a result, Lemma~\ref{psi.linfty} implies that $ \|\chi\|^2_{L^2(\Omega_k,g_{\max})} \leqslant C \|\chi\|_{\mL_k}$.
Combining with \eqref{comp.est},~\eqref{vests} and Lemma~\ref{psi.est}, this implies that
\begin{equation}
\left|\int_{\Omega_k}\langle d\phi,d\chi\rangle-\Lambda_1\int_{\partial\Omega_k}\beta_k\phi\chi\right|\leq C\frac{\log k}{k}\|d\phi\|_{L^2(M)}\|\chi\|_{\mL_k}
\end{equation}
for any harmonic $\chi\in W^{1,2}(\Omega_k)$, as desired.

To prove \eqref{phi.quot}, first note that
\begin{eqnarray*}
\|d\phi\|_{L^2(M)}^2&=&\Lambda_1\int_Mf\phi^2
=\Lambda_1\int_{\Omega_k}f\phi^2+\Lambda_1\int_{M\setminus\Omega_k}f\phi^2\\
\text{(by \eqref{vests}) }&\leq &\Lambda_1\int_{\Omega_k}f\phi^2+C\|d\phi\|^2_{L^2(M)}\area(M\setminus\Omega_k)\\
&\leq & \Lambda_1\int_{\Omega_k}f\phi^2+C'\|d\phi\|_{L^2(M)}^2\cdot \frac{1}{k^2},
\end{eqnarray*}
so that
$$
\left(1-\frac{C'}{k^2}\right)\|d\phi\|_{L^2(M)}^2\leq \Lambda_1\int_{\Omega_k}f\phi^2.
$$
On the other hand, we see that
\begin{eqnarray*}
\int_{\Omega_k}f\phi^2&=&\int_{\Omega_k}(-\Delta \psi_k)\phi^2\\
&=&\int_{\Omega_k}\left(\mathrm{div}(\phi^2 d\psi_k)-\langle d\psi_k,d(\phi^2)\rangle\right)\\
&=&\int_{\partial\Omega_k}\frac{\partial\psi_k}{\partial\nu}\phi^2-\int_{\Omega_k}\psi_k \Delta(\phi^2),
\end{eqnarray*}
so that
\begin{eqnarray*}
\int_{\Omega_k}f\phi^2-\int_{\partial\Omega_k}\beta_k\phi^2&=&-\int_{\Omega_k}\psi_k\Delta(\phi^2)\\
&\leq & C\|\psi_k\|_{L^2(\Omega_k)}\|\phi\|_{C^2}^2,
\end{eqnarray*}
and by \eqref{vests} and Lemma \ref{psi.est}, it follows that
$$
\int_{\Omega_k}f\phi^2\leq \int_{\partial\Omega_k}\beta_k\phi^2+C'\frac{\log k}{k}\|d\phi\|_{L^2(M)}^2.
$$
Combining this with the preceding estimates, we deduce that
$$
\left(1-\frac{C'}{k^2}\right)\|d\phi\|_{L^2(M)}^2\leq \Lambda_1\int_{\partial\Omega_k}\beta_k\phi^2+C'\frac{\log k}{k}\|d\phi\|_{L^2(M)}^2,
$$
and consequently
$$
\left(1-\frac{C''\log k}{k}\right)\|d\phi\|_{L^2(M)}^2\leq \Lambda_1\int_{\partial\Omega_k}\beta_k\phi^2,
$$
from which \eqref{phi.quot} readily follows.

\end{proof}

Now, denote by
$$0=\sigma_0(\Omega_k,\tilde{g}_k)<\sigma_1(\Omega_k,\tilde{g}_k)<\cdots$$
the Steklov spectrum of $(\Omega_k,\tilde{g}_k)$, and let $\varphi_0,\varphi_1,\ldots$ be an associated collection of eigenfunctions, normalized so that
$$\int_{\partial\Omega_k}\beta_k \varphi_i\varphi_j=\delta_{ij}.$$
For any $\eta>0$, consider the space
$$
W_{\eta}:=\mathrm{Span}\{\varphi_i\mid \Lambda_1-\eta\leq \sigma_i\leq \Lambda_1+\eta \}
$$
spanned by all Steklov eigenspaces of $(\Omega_k,\tilde{g}_k)$ corresponding to eigenvalues in $\left[\Lambda_1-\eta,\Lambda_1+\eta\right]$. Using the preceding lemma, we can prove the following.

\begin{lemma}\label{band.mult} There exists a constant $C_2(M,g_0)$ such that 
$$
\dim\left(W_{C_2\frac{\log k}{k}}\right)\geq \dim V.
$$
\end{lemma}
\begin{proof}

Fix $\eta\in (0,1)$, and let 
$$\Lambda_1-\eta\leq\sigma_m\leq \cdots\leq \sigma_{m+\ell}\leq \Lambda_1+\eta$$
be the portion of the Steklov spectrum of $(\Omega_k,\tilde{g}_k)$ lying in $[\Lambda_1-\eta,\Lambda_1+\eta]$. Consider the projection map
$$
\Pi_{\eta}: V\to W_{\eta}
$$
given by
$$
\Pi_{\eta}\phi:=\sum_{i=m}^{m+\ell}\left(\int_{\partial\Omega_k}\beta_k\phi \varphi_i\right)\varphi_i.
$$
Suppose that $\phi\in \ker(\Pi_{\eta})$ is an element of the kernel.
The harmonic extension $\hat\phi = \sum a_i\varphi_i$ of $\phi|_{\partial\Omega_k}$ to $\Omega_k$ can be written 
$$
\hat\phi=\hat\phi_- + \hat\phi_+ := \sum_{\sigma_i<\Lambda_1-\eta} a_i\varphi_i + \sum_{\sigma_i>\Lambda_1+\eta}a_i\varphi_i,
$$
where $a_i=0$ if $\Lambda_1-\eta\leq \sigma_i\leq \Lambda_1+\eta$. Setting
$$I_{\pm}:=\{i\in \mathbb{N}\cup \{0\} \mid \pm (\sigma_i-\Lambda_1)>\eta\},$$
one then has
\begin{equation}
\label{ineq:phipm_low}
\begin{split}
\left|\int_{\Omega_k}|d\hat\phi_\pm|^2 - \Lambda_1\int_{\partial\Omega_k}\beta_k\hat\phi_\pm^2 \right| &=\left| \sum_{i\in I_{\pm}}(\sigma_i-\Lambda_1)a_i^2\right|\geqslant\\
&\geqslant  \min_{i\in I_{\pm}}\{\pm(\sigma_i-\Lambda_1)\}\sum_{i\in I_{\pm}} a_i^2 \geqslant \eta\|\phi_\pm\|^2_{L^2(\partial\Omega_k,\tilde g_k)}.
\end{split}
\end{equation}

At the same time, since $\hat\phi_+\perp\hat\phi_-$ in $L^2(\partial\Omega_k,\tilde g_k)$, one has
$$
\int_{\Omega_k}\la d\phi, d\hat\phi_{\pm}\ra= \int_{\Omega_k}\la d\hat\phi, d\hat\phi_{\pm}\ra = \int_{\Omega_k} |d\hat\phi_{\pm}|^2.
$$
Therefore, Lemma~\ref{quasimode.lem} implies that 
\begin{equation*}
\begin{split}
&\left|\int_{\Omega_k}\la d\phi, d\hat\phi_\pm\ra - \Lambda_1\int_{\partial\Omega_k}\beta_k\hat\phi\phi_\pm \right| = \left|\int_{\Omega_k}|d\hat\phi_\pm|^2 - \Lambda_1\int_{\partial\Omega_k}\beta_k\hat\phi_\pm^2 \right|\leqslant \\
&\leqslant C\frac{\log k}{k}\|\phi\|_{L^2(\partial\Omega_k,\tilde g_k)}\|\hat\phi_{\pm}\|_{\mL_k}.
\end{split}
\end{equation*}
Combining this with~\eqref{ineq:phipm_low} and adding up the inequalities for $\hat\phi_{\pm}$ yields
\begin{equation}
\label{ineq:almost_qsm}
\eta\|\phi\|^2_{L^2(\partial\Omega_k,\tilde g_k)}\leqslant C'\frac{\log k}{k}\|\phi\|_{L^2(\partial\Omega_k,\tilde g_k)}\|\hat\phi\|_{\mL_k}.
\end{equation}
Finally, the inequality~\eqref{phi.quot} gives that
$$
\|d\hat\phi\|_{L^2(\Omega_k)}\leqslant \|d\phi\|_{L^2(M)}\leqslant C\|\phi\|_{L^2(\partial\Omega_k,\tilde g_k)},
$$
which together with~\eqref{ineq:almost_qsm} implies
$$
\eta\|\phi\|^2_{L^2(\partial\Omega_k,\tilde g_k)}\leqslant C''\frac{\log k}{k}\|\phi\|^2_{L^2(\partial\Omega_k,\tilde g_k)}.
$$
If $\phi\neq 0$, dividing by $\|\phi\|_{L^2(\partial\Omega_k,\tilde{g}_k)}^2$ on both sides yields $\eta\leqslant C''\frac{\log k}{k}$. In other words, $\Pi_{\eta}\colon V\to W_{\eta}$ must be \emph{injective} whenever $\eta>C''\frac{\log k}{k}$, so setting, e.g., $C_2=2C''$, it follows that
$$
\dim (W_{C_2\frac{\log k}{k}})\geq \dim V,
$$
as desired.

\end{proof}

With Lemma \ref{band.mult} in hand, we argue finally that the first Steklov eigenfunction for $(\Omega_k,\tilde{g}_k)$ must also lie in $W_{C_2\frac{\log k}{k}}$ for $k\geq k_0$ sufficiently large, to complete the proof of Proposition \ref{construct}.

\begin{proof}[Proof of Proposition \ref{construct}] Let $C_2=C_2(M,g_0)$ be the constant from Lemma \ref{band.mult}, and consider the set
$$
S:=\left\{k\in \mathbb{N}\mid \sigma_1(\Omega_k,\tilde{g}_k)<\Lambda_1-C_2\frac{\log k}{k}\right\}.
$$
To prove Proposition \ref{construct}, the main step consists of showing that $S$ is \emph{finite}. 

To show that $S$ is finite, we will argue by contradiction. Letting $N:=\dim V$ denote the multiplicity of the first eigenvalue of $\Delta_{g_{\max}}$, it follows from Lemma \ref{band.mult} that if $k\in S$, then at least the first $N+1$ nontrivial Steklov eigenvalues of $(\Omega_k,\tilde{g}_k)$ must lie below $\Lambda_1+C_2\frac{\log k}{k}$; i.e.,
$$0<\sigma_1(\Omega_k,\tilde{g}_k)\leq \cdots\leq \sigma_{N+1}(\Omega_k,\tilde{g}_k)\leq \Lambda_1+C_2\frac{\log k}{k}.$$
Denote by $\varphi_{k,1},\ldots, \varphi_{k,N+1}$ the corresponding Steklov eigenfunctions, normalized so that
$$
\int_{\partial\Omega_k}\beta_k\varphi_{k,i}\varphi_{k,j}=\delta_{ij},
$$
and let $\hat{\varphi}_{k,i}\in W^{1,2}(M)$ denote the harmonic extension into $M\setminus \Omega_k$; note then that
\begin{equation}
\|\hat{\varphi}_{k,i}\|_{W^{1,2}(M)}\leq C\|\varphi_{k,i}\|_{W^{1,2}(\Omega_k)}\leq C,
\end{equation}
by Lemma \ref{ext.lem}. 

To obtain a contradiction, suppose that $S$ is infinite, and pass to a subsequence $k_j\in S$ such that
$$
\hat{\varphi}_{k_j,i}\rightharpoonup \hat{\varphi}_i\in W^{1,2}(M)
$$
weakly in $W^{1,2}$ and strongly in $L^2$ as $k_j\to \infty$. (In what follows, we write $k_j=k$ for simplicity.) For any $\chi\in C^{\infty}(M)$, we note then that
\begin{equation*}
\begin{split}
&\left|\int_M\langle d\hat{\varphi}_{k,i},d\chi\rangle-\sigma_i(\Omega_k,\tilde{g}_k)\int_Mf\hat{\varphi}_{k,i}\chi\right|\\
\leq &\left|\int_{\Omega_k}\langle d\varphi_{k,i},d\chi\rangle-\sigma_i(\Omega_k,\tilde{g}_k)\int_{\Omega_k}f\varphi_{k,i}\chi\right|\\
+&\left|\int_{M\setminus\Omega_k}\left(\langle d\hat{\varphi}_{k,i},d\chi\rangle-\sigma_i(\Omega_k,\tilde{g}_k)f\hat{\varphi}_{k,i}\chi\right)\right|
\\
\leq &\sigma_i(\Omega_k,\tilde{g}_k)\left|\int_{\partial\Omega_k}\beta_k \varphi_{k,i}\chi-\int_{\Omega_k}f\varphi_{k,i}\chi\right|
\\
&+C\|\chi\|_{C^1}\|\hat{\varphi}_{k,i}\|_{W^{1,2}}\area(M\setminus\Omega_k)^{1/2}.
\end{split}
\end{equation*}
In particular, since
$$
\area(M\setminus\Omega_k)\leq \frac{C}{k^2}\to 0
$$
as $k\to\infty$ and, by definition of $\beta_k=\frac{\partial\psi_k}{\partial\nu}$,
\begin{eqnarray*}
\left|\int_{\partial\Omega_k}\beta_k\varphi_{k,i}\chi-\int_{\Omega_k}f\varphi_{k,i}\chi\right|&=&\left|\int_{\Omega_k}\langle d\psi_k,d(\chi\varphi_{k,i})\rangle\right|\\
&=&\left|\int_{\Omega_k}\psi_k \Delta(\chi\varphi_{k,i})\right|\\
&=&\left|\int_{\Omega_k}\psi_k(\varphi_{k,i}\Delta\chi-2\langle d\chi,d\varphi_{k,i})\right|\\
&\leq &C\|\psi_k\|_{L^2}\|\chi\|_{C^2}\|\varphi_{k,i}\|_{W^{1,2}}\to 0
\end{eqnarray*}
as $k\to\infty$ by Lemma \ref{psi.est}, it follows that
$$
\lim_{k\to\infty}\left|\int_M\langle d\hat{\varphi}_{k,i},d\chi\rangle-\sigma_i(\Omega_k,\tilde{g}_k)\int_Mf\hat{\varphi}_{k,i}\chi\right|=0
$$
for any $\chi\in C^{\infty}(M)$. Thus, the weak limit $\hat{\varphi}_i$ of $\hat{\varphi}_{k,i}$ along the subsequence $k_j\in S$ satisfies
\begin{equation}\label{wklim.eqn}
\Delta_{g_0}\hat{\varphi}_i=\tilde{\sigma}_if\hat{\varphi}_i,
\end{equation}
where
\begin{equation}\label{wklim.sigma}
\tilde{\sigma}_i:=\lim_{k\to\infty}\sigma_i(\Omega_k,\tilde{g}_k)\leq \lim_{k\to\infty}\left[\Lambda_1+C_2\frac{\log k}{k}\right]=\Lambda_1.
\end{equation}

Moreover, we see that
\begin{equation*}
\begin{split}
\left|\int_Mf\hat{\varphi}_i\hat{\varphi_j}-\delta_{ij}\right|\leq &\left|\int_{\Omega_k}f\hat{\varphi}_i\hat{\varphi}_j-\int_{\partial\Omega_k}\beta_k\varphi_{k,i}\varphi_{k,j}\right|\\
&+\|\hat{\varphi}_i\|_{L^{\infty}}\|\hat{\varphi}_j\|_{L^{\infty}}\area(M\setminus\Omega_k)\\
\leq &\int_{\Omega_k}f|\hat{\varphi}_i\hat{\varphi}_j-\varphi_{k,i}\varphi_{k,j}|+\left|\int_{\Omega_k}f\varphi_{k,i}\varphi_{k,j}-\int_{\partial\Omega_k}\beta_k\varphi_{k,i}\varphi_{k,j}\right|\\
&+\|\hat{\varphi}_i\|_{L^{\infty}}\|\hat{\varphi}_j\|_{L^{\infty}}\cdot\frac{C}{k^2},
\end{split}
\end{equation*}
and in view of the (strong) $L^2$ convergence $\varphi_{k,i}\cdot {\bf 1}_{\Omega_k}\to \hat{\varphi}_i$, it follows that
\begin{eqnarray*}
\left|\int_Mf\hat{\varphi}_i\hat{\varphi_j}-\delta_{ij}\right|&\leq &\lim_{k\to\infty}\left|\int_{\Omega_k}f\varphi_{k,i}\varphi_{k,j}-\int_{\partial\Omega_k}\beta_k\varphi_{k,i}\varphi_{k,j}\right|\\
&=&\lim_{k\to\infty}\left|\int_{\Omega_k}\psi_k\Delta(\varphi_{k,i}\varphi_{k,j})\right|\\
\text{(since $\Delta \varphi_{k,i}=0$) }&=&2\lim_{k\to\infty}\left|\int_{\Omega_k}\psi_k\langle d\varphi_{k,i},d\varphi_{k,j}\rangle\right|\\
&\leq &C\lim_{k\to\infty}\|\psi_k\|_{L^{\infty}}\sqrt{\sigma_i(\Omega_k,\tilde{g}_k)\sigma_j(\Omega_k,\tilde{g}_k)}\\
\text{(by Lemma \ref{psi.linfty}) }&\leq &C\lim_{k\to\infty}\sqrt{\frac{\log k}{k}}=0,
\end{eqnarray*}
so that the functions $\{\hat{\varphi}_i\}_{i=1}^{N+1}$ are orthonormal in $L^2(M,g_{\max})$. 

Similarly, since $\int_{\Omega_k}f\varphi_{k,i}=\int_{\partial\Omega_k}\beta_k\varphi_{k,i}=0$, it is easy to see that
$$\int_Mf\hat{\varphi}_i=\lim_{k\to\infty}\int_{\Omega_k}f\varphi_{k,i}=0,$$
so that each $\hat{\varphi}_i$ is likewise orthogonal to the constant functions in $L^2(M,g_{\max})$, and in particular, we must have equality in \eqref{wklim.sigma}. Putting all this together, we see that $\{\hat{\varphi}_i\}_{i=1}^{N+1}$ gives an $L^2(M,g_{\max})$-orthonormal collection of first eigenfunctions for $\Delta_{g_{\max}}$. But $N=\dim V$ is precisely the dimension of the $\Lambda_1$-eigenspace, so we've arrived at a contradiction.

We've therefore confirmed that the set $S$ of integers $k$ for which $\sigma_1(\Omega_k,\tilde{g}_k)<\Lambda_1-C_2\frac{\log k}{k}$ must be \emph{finite}; in other words, there exists $k_0(M,g_0)$ such that
$$\sigma_1(\Omega_k,\tilde{g}_k)\geq \Lambda_1-C_2\frac{\log k}{k}$$
whenever $k\geq k_0$. Finally, we know from \eqref{beta.mass} that
$$\length_{\tilde{g}_k}(\partial\Omega_k)=\int_{\partial\Omega_k}\beta_k\geq 1-\frac{C}{k^2},$$
so indeed we must have
$$
\bar{\sigma}_1(\Omega_k,\tilde{g}_k)=\sigma_1(\Omega_k,\tilde{g}_k)L(\partial\Omega_k,\tilde{g}_k)\geq \Lambda_1-C\frac{\log k}{k},
$$
as desired.

\end{proof}

\section{Upper bounds for $\Sigma_1(N_k)$}\label{sec:u.bds}

In this section, we complete the proof of Theorem \ref{thm:asymptotics_intro}, by proving the upper bound \eqref{ineq:upperbound_intro}, which we reformulate as the following proposition.

\begin{proposition}\label{ubd.prop}
Let $M=\Sp,\mathbb{RP}^2, \mathbb{T}^2$, or the Klein bottle $\mathbb{K}$, and let $N_k$ be the compact surface with boundary given by removing $k$ disjoint disks from $M$. Then there exists a constant $c(M)>0$ such that for any metric $g$ on $N_k$, 
\begin{equation}
\bar{\sigma}_1(N_k,g)\leq \Lambda_1(M)-c(M)\frac{\log k}{k}.
\end{equation}
\end{proposition}

As discussed in the introduction, it is quite possible that the estimate holds for all closed surfaces $M$, not just those listed here. From the arguments below, it follows that the upper bound \eqref{ineq:upperbound_intro} holds for all $M$ satisfying the hypotheses of Theorem 6.1 in \cite{KNPS}--i.e., all those $M$ for which the minimal surfaces in $\mathbb{S}^n$ realizing $\Lambda_1(M)$ have maximal possible Morse index $n+1+\dim (\mathcal{M}(M))$ as critical points of the area functional, where $\mathcal{M}(M)$ denotes the moduli space of conformal structures on $M$.

\subsection{Refined quantitative stability for Steklov-maximizing metrics}\label{ref.stab}

As an important first step toward proving Proposition \ref{ubd.prop}, we need to refine the quantitative stability results of \cite{KNPS} for nearly $\bar{\lambda}_1$-maximizing metrics. The difference between the results of \cite{KNPS} and those below is that here we are interested in obtaining lower bounds on the gap
$$\Lambda_1(M)-\bar{\lambda}_1(\Omega,[g],\mu)$$
between the maximum $\Lambda_1(M)$ and the normalized first eigenvalue \emph{restricted to a domain} $\Omega\subset M$ for a measure $\mu$ supported on $\Omega\subset M$, whereas the results in \cite{KNPS} provide lower bounds for the gap $\Lambda_1(M)-\bar{\lambda}_1(M,[g],\mu)$. While the proofs are quite similar, we note that the refinement is necessary to obtain the sharp upper bound, as a direct application of the results in \cite{KNPS} seems to yield at best the non-sharp bound $\Sigma_1(N_k)\leq \Lambda_1(M)-\frac{c}{k}$.

We begin with the following straightforward adaptation of Lemma 2.1 in \cite{KNPS}.

\begin{lemma}\label{stek.stab.lem}
Let $\Omega\subset M$ be a smooth domain in a closed Riemannian surface $(M,g)$, and let $\mu$ be an admissible measure supported in $\Omega$ with the first nontrivial eigenvalue
$$
\sigma_1:=\lambda_1(\Omega,[g],\mu).
$$
If $u\in W^{1,\infty}(M,\mathbb{S}^n)$ is a sphere-valued map such that
$$
\int_Mu d\mu=0,
$$
then 
\begin{equation}\label{stek.stab.est}
\|\sigma_1\mu-|du|_{g}^2dv_g\lfloor \Omega\|_{\left(W^{1,2}(\Omega,g)\right)^*}^2\leq \|u\|_{W^{1,\infty}(g)}^2\left(2E_{g}(u;\Omega)-\sigma_1\mu(M)\right).
\end{equation}
\end{lemma}

\begin{proof} Denote by $V\subset W^{1,2}(\Omega,\mathbb{R}^{n+1})$ the subspace of maps $v\colon\Omega\to \mathbb{R}^{n+1}$ for which
$$
\int v d\mu=0,
$$
and consider the quadratic form $Q$ on $V$ given by
$$
Q(v,v):=\int_{\Omega} |dv|_{g}^2dv_{g}-\sigma_1\int_{\Omega} |v|^2d\mu.
$$
By definition of $\sigma_1$, it is clear that $Q$ is nonnegative definite on $V$, and therefore the Cauchy-Schwarz inequality for the associated bilinear form gives
\begin{equation}\label{cs.cons}
\left|\int_{\Omega}\langle du,dv\rangle_{g}dv_{g}-\sigma_1\int_\Omega\langle u,v\rangle d\mu\right|\leq \sqrt{Q(u,u)}\sqrt{Q(v,v)}.
\end{equation}

Now, let $u\in V\cap W^{1,\infty}(M,\mathbb{S}^n)$, as in the hypotheses of the lemma. Then since $|u|\equiv 1$, we have
$$Q(u,u)=2E_{g}(u;\Omega)-\sigma_1\mu(M),$$
and for any $v\in W^{1,2}(\Omega,\mathbb{R}^{n+1})$, applying \eqref{cs.cons} to $u$ and the map
$$
v_1=v-\frac{1}{\mu(M)}\int_{\Omega} v d\mu,
$$
we see that
\begin{eqnarray*}
\left|\int_{\Omega}\langle du,dv\rangle_{g}dv_g-\sigma_1\int_\Omega \langle u,v\rangle d\mu\right|&=&\left|\int_{\Omega}\langle du,dv_1\rangle_gdv_g-\sigma_1\int_\Omega\langle u,v_1\rangle d\mu\right|\\
&\leq &\sqrt{Q(v_1,v_1)}\sqrt{2E_{g}(u;\Omega)-\sigma_1\mu(M)}\\
&\leq &\|dv\|_{L^2(\Omega)}\sqrt{2E_{g}(u;\Omega)-\sigma_1\mu(M)}.
\end{eqnarray*}
In particular, taking $v=\varphi u$ for some $\varphi\in W^{1,2}(\Omega)$, and recalling that $\langle du,d(\varphi u)\rangle=\varphi |du|^2$ since $|u|\equiv 1$, it follows that
$$
\left|\int_{\Omega}\varphi |du|_g^2dv_g-\sigma_1\int_\Omega \varphi d\mu\right|\leq \|d(\varphi u)\|_{L^2(\Omega)}\sqrt{2E_{g}(u;\Omega)-\sigma_1\mu(M)}.
$$
In particular, since
$$
\|d(\varphi u)\|_{L^2(\Omega)}^2=\int_{\Omega}\varphi^2|du|_g^2+|u|^2|d\varphi|_g^2dv_g\leq \|u\|_{W^{1,\infty}(g)}^2\|\varphi\|_{W^{1,2}(\Omega,g)}^2,
$$
it follows that
$$
\frac{|\langle \varphi, |du|_g^2dv_g-\sigma_1\mu\rangle|}{\|\varphi\|_{W^{1,2}(\Omega,g)}}\leq \|u\|_{W^{1,\infty}(g)}\sqrt{2E_{g}(u;\Omega)-\sigma_1\mu(M)},
$$
which is precisely what we wanted to show.
\end{proof}

As an immediate consequence, for surfaces of genus $0$, we have the following stability estimate--which, combined with uniformization and the standard Hersch trick, will suffice for our purposes in the genus $0$ case.

\begin{proposition}\label{s2.quant.stab} Let $\Omega\subset \Sp$ be a domain in the round unit sphere $(\Sp,g_0)\subset \mathbb{R}^3$, and let $\tilde{g}\in [g_0]$ be a conformal metric such that the identity map $I\colon\Sp\hookrightarrow \mathbb{R}^3$ satisfies
$$
\int_{\Sp}I d\mu=0,
$$
where $\mu$ is the length measure $\mu=ds_{\tilde g}$ of $\partial\Omega$. Then for the first nontrivial Steklov eigenvalue $\sigma_1=\sigma_1(\Omega,\tilde{g})$, we have
\begin{equation}
\|\sigma_1\mu-2dv_{g_0}\lfloor\Omega\|_{(W^{1,2}(\Omega,g_0))^*}^2+6\area_{g_0}(M\setminus\Omega)\leq 3\left(8\pi-\bar{\sigma}_1(\Omega,\tilde{g})\right)
\end{equation}
\end{proposition} 
\begin{proof} Applying Lemma \ref{stek.stab.lem} with $\mu=\mathcal{H}^1_{\tilde{g}}\lfloor \partial\Omega$ and the identity map $u=I$, for which $|du|_{g_0}^2\equiv 2$ and $2E_{g_0}(u;\Omega)=2\area_{g_0}(\Omega)$, we see that \eqref{stek.stab.est} gives
\begin{eqnarray*}
\|\sigma_1\mu-2dv_{g_0}\lfloor\Omega\|_{(W^{1,2}(\Omega,g_0))^*}^2&\leq & 3\left(2\area_{g_0}(\Omega)-\bar{\sigma}_1(\Omega,\tilde{g})\right)\\
&=&3\left(8\pi-2\area_{g_0}(M\setminus\Omega)-\bar{\sigma}_1(\Omega,\tilde{g})\right),
\end{eqnarray*}
from which the desired estimate immediately follows.
\end{proof}

To obtain analogous estimates in the cases where $M=\mathbb{RP}^2,\mathbb{T}^2,$ or the Klein bottle $\mathbb{K}$, we combine Lemma \ref{stek.stab.lem} with the techniques of \cite[Section 6]{KNPS}. The case of $M=\mathbb{RP}^2$--which carries only one conformal structure--is in principle simpler, but we group it with the others for convenience.

\begin{proposition}
\label{quant.stab.prop}
 Let $M$ be a closed surface homeomorphic to $\mathbb{RP}^2$, $\mathbb{T}^2$, or the Klein bottle $\mathbb{K}$, and let $g_1\in \Met_{\can}(M)$ be a unit-area, constant curvature metric on $M$. There exist constants $C(M),\delta_1(M)\in (0,\infty)$ such that the following holds. If $\Omega\subset M$ is a smooth domain in $M$ with a conformal metric $\tilde{g}\in [g_1]$ such that
$$
\bar{\sigma}_1(\Omega,\tilde{g})\geq \Lambda_1(M)-\delta_1,
$$
then there exists a $\bar{\lambda}_1$-maximal metric $g_{\max}$ conformal to some $g_0\in \Met_{\can}(M)$, such that
\begin{equation}
\label{qstab.est.0}
\|g_0-g_1\|_{C^1(g_0)}^2\leq C\left(\Lambda_1(M)-\bar{\sigma}_1(\Omega,\tilde{g})\right)
\end{equation}
and the length measure $\mu=ds_{\tilde g}$ of $\partial\Omega$, normalized by $\sigma_1=\sigma_1(\Omega,\tilde{g})$ satisfies
\begin{equation}\label{qstab.est.1}
\|\sigma_1\mu-\lambda_1(g_{max})dv_{g_{max}}\|_{\left(W^{1,2}(\Omega,g_0)\right)^*}^2+\area_{g_0}(M\setminus\Omega)\leq C\left(\Lambda_1(M)-\bar{\sigma}_1(\Omega,\tilde{g})\right).
\end{equation}
\end{proposition}

\begin{proof} The proof follows closely that of \cite[Theorem 1.17]{KNPS}, with Lemma \ref{stek.stab.lem} replacing \cite[Lemma 2.1]{KNPS} at the final step. As discussed in \cite[Section 6]{KNPS}, the minimal immersions $u\colon M\to \mathbb{S}^n$ that induce the $\bar{\lambda}_1$-maximizing metrics on $M=\mathbb{RP}^2$, $\mathbb{T}^2$, and $\mathbb{K}$ all have \emph{maximal Morse index} as critical points of the area functional, in the sense that $\ind_A(u)=n+1+\dim(\mathcal{M}_0(M))$, where $\mathcal{M}_0(M)=\Met_{\can}(M)/\Diff_0(M)$ denotes the Teichm\"{u}ller space of conformal structures on $M$. In particular, these minimal immersions satisfy the hypotheses of \cite[Lemma 6.5]{KNPS}.

Following the proof of \cite[Theorem 6.1]{KNPS}, let $\mathcal{C}_{\max}\subset \mathcal{M}_0(M)$ denote the set of (equivalence classes of) conformal structures $\langle g\rangle$ achieving the maximum $\Lambda_1(M,[g])=\Lambda_1(M)$. By \cite[Lemma 6.5]{KNPS}, there exists a neighborhood $\mathcal{U}$ of $\mathcal{C}_{\max}$ in $\mathcal{M}_0(M)$ and a family of maps
\begin{equation}\label{map.fam}
\mathcal{U}\ni \tau \mapsto F_{\tau}\in C^{\infty}(M,\mathbb{S}^n)
\end{equation}
such that the constant curvature metric $g_{\tau}$ conformal to $F_{\tau}^*(g_{\mathbb{S}^n})$ lies in $\tau\in \mathcal{M}_0(M)$, for every $\langle g_0\rangle\in \mathcal{C}_{max}$ the map $F_{\langle g_0\rangle}=u_0$ is a minimal immersion inducing the $\bar{\lambda}_1$-maximizing metric, and denoting by 
$$
\mathbb{B}^{n+1}\ni a\mapsto G_a\in \Conf(\mathbb{S}^n)
$$
the canonical family of conformal dilations, for every $(a,\tau)\in \mathbb{B}^{n+1}\times \mathcal{U}$ such that
\begin{equation}\label{area.big}
\area(G_a\circ F_{\tau})\geq \frac{1}{2}[\Lambda_1(M)-\delta_0(M)],
\end{equation}
for a small constant $\delta_0(M)>0$, we have 
\begin{equation}\label{area.stab}
\|g_{\tau}-g_0\|_{C^1(g_0)}^2+\|G_a\circ F_{\tau}-u_0\|_{C^2(g_0)}^2\leq C(M)[\Lambda_1(M)-2\area(G_a\circ F_{\tau})],
\end{equation}
for some $\langle g_0\rangle \in \mathcal{C}_{max}$ with $u_0=F_{\langle g_0\rangle}$.

Now, let $\Omega\subset M$ and $\tilde{g}\in [g_1]$ be as in the hypotheses of the proposition, with $\mu=ds_{\tilde g}$. It follows from~\eqref{ineq:Sdom} that 
\begin{equation}\label{gap.small}
\Lambda_1(M)-\delta_1<\bar{\sigma}_1(\Omega,\tilde{g})\leq \bar{\lambda}_1(M,[g_1],\mu).
\end{equation}
Theorem \ref{thm:glob_qual_stab} then implies that for $\delta_1=\delta_1(M)>0$ sufficiently small, one has
$$
\langle g_1\rangle\in \mathcal{U},
$$
where $\mathcal{U}$ is the neighborhood of $\mathcal{C}_{\max}$ given above. Assume now that \eqref{gap.small} holds, and let $F_1=F_{\langle g_1\rangle}$ be the map associated to $\langle g_1\rangle$ as in \eqref{map.fam}. We know then that $F_1\circ \Phi\colon (M,g_1)\to \mathbb{S}^n$ is conformal for some diffeomorphism $\Phi\in \Diff_0(M)$, and since the desired estimates \eqref{qstab.est.0}-\eqref{qstab.est.1} are invariant under the change
$$
(\Omega,\tilde{g},g_1,g_0)\mapsto \left(\Phi^{-1}(\Omega),\Phi^*\tilde{g},\Phi^*g_1,\Phi^*g_0\right)
$$
for any diffeomorphism $\Phi\in \Diff(M)$, we may assume without loss of generality that
$$
F_1\colon (M,g_1)\to \mathbb{S}^n
$$
is conformal.

By the standard Hersch trick (see, e.g., \cite{LY}), there must exist a conformal dilation $G_a\in \Conf(\mathbb{S}^n)$ for which the map
$$
u_1:=G_a\circ F_1
$$
satisfies
$$
\int_{\Omega} u_1 d\mu=0\in \mathbb{R}^{n+1}.
$$
Therefore, by the definition of $\sigma_1$, we see that
$$
2\area(u_1(M))=2E_{g_1}(u_1(M))\geq \sigma_1\int |u_1|^2d\mu=\bar{\sigma}_1(\Omega,\tilde{g}),
$$
so that, by \eqref{gap.small},
$$
\area(u_1(M))\geq \frac{1}{2}\left(\Lambda_1(M)-\delta_1\right).
$$
In particular, taking $\delta_1(M)<\delta_0(M)$, we see that \eqref{area.big} is satisfied, so there exists a minimal immersion $u_0\colon M\to \mathbb{S}^n$ inducing a $\bar{\lambda}_1$-maximizing metric $g_{\max}$ and a unit-area constant curvature metric $g_0\in [g_{\max}]$ such that
\begin{equation}
\label{map.met.close}
\|g_1-g_0\|_{C^1(g_0)}^2+\|u_1-u_0\|_{C^2(g_0)}^2\leq C\left(\Lambda_1(M)-2\area(u_1(M))\right)\leq C\delta_1(M).
\end{equation}
As an immediate consequence, we have
$$
\|g_1-g_0\|_{C^1(g_0)}^2\leq C\left(\Lambda_1(M)-\bar{\sigma}_1(\Omega,\tilde{g})\right),
$$
giving the first desired estimate~\eqref{qstab.est.0}.

Now, by Lemma \ref{stek.stab.lem}, we have that
\begin{eqnarray*}
\|\sigma_1\mu-|du_1|^2_{g_1}dv_{g_1}\lfloor\Omega\|_{[W_{g_1}^{1,2}(\Omega)]^*}&\leq & \|u_1\|_{W^{1,\infty}(g_1)}\left(2E_{g_1}(u_1;\Omega)-\bar{\sigma}_1(\Omega,\tilde{g})\right)^{1/2}\\
&\leq &C\left(2\area(u_1(\Omega))-\bar{\sigma}_1(\Omega,\tilde{g})\right)^{1/2},
\end{eqnarray*}
where in the last line we used the conformality of $u_1$ and the fact that $\|u_1\|_{C^1(g_1)}\leq 2\|u_0\|_{C^1(g_0)}\leq C'(M)$ for $\delta_1(M)$ sufficiently small, by \eqref{map.met.close}. Furthermore, It follows from~\eqref{map.met.close} that
$$
\frac{1}{2}\|\cdot\|_{(W^{1,2}(\Omega,g_1))^*}\leq \|\cdot\|_{(W^{1,2}(\Omega,g_0))^*}\leq 2\|\cdot\|_{(W^{1,2}(\Omega,g_1))^*}
$$
provided $\delta_1(M)$ is sufficiently small, so that
\begin{equation}
\|\sigma_1\mu-|du_1|_{g_1}^2dv_{g_1}\lfloor \Omega\|_{\left(W^{1,2}(\Omega,g_0)\right)^*}\leq C\left(2\area(u_1(\Omega))-\bar{\sigma}_1(\Omega,\tilde{g})\right)^{1/2}.
\end{equation}
Repeatedly using \eqref{map.met.close}, we also see that
$$
\||du_1|_{g_1}^2dv_{g_1}-|du_0|_{g_0}^2dv_{g_0}\|_{\left(W^{1,2}(\Omega,g_0)\right)^*}\leq C\left(\Lambda_1(M)-2\area(u_1(M))\right)^{1/2},
$$
and recalling that $|du_0|_{g_0}^2dv_{g_0}=2dv_{g_{\max}}$, we can combine this with the preceding estimate to find that
\begin{equation*}
\begin{split}
\|\sigma_1\mu-2dv_{g_{\max}}\|_{\left(W^{1,2}(\Omega,g_0)\right)^*}^2 &\leq C'[\Lambda_1(M)-2\area(u_1(M))]\\
&+C'[2\area(u_1(\Omega))-\bar{\sigma}_1(\Omega,\tilde{g})] \\
&=C'[\Lambda_1(M)-\bar{\sigma}_1(\Omega,\tilde{g})]-2C'\area(u_1(M\setminus\Omega));
\end{split}
\end{equation*}
i.e.,
\begin{equation}\label{almost.est}
\|\sigma_1\mu-2dv_{g_{max}}\|_{\left(W^{1,2}(\Omega,g_0)\right)^*}^2+2C'\area(u_1(M\setminus \Omega))\leq C'\left(\Lambda_1(M)-\bar{\sigma}_1(\Omega,\tilde{g})\right).
\end{equation}
Finally, it follows from \eqref{map.met.close} that
$$
g_{\max}=u_0^*(g_{\mathbb{S}^n})\leq 2u_1^*(g_{\mathbb{S}^n})
$$
provided $\delta_1(M)$ is sufficiently small, and we know that
$$
g_0\leq C(M)g_{\max},
$$
since the maximizing metrics on $\mathbb{RP}^2$, $\mathbb{T}^2$, and $\mathbb{K}$ are smooth. Hence, we have
$$
\area_{g_0}(M\setminus\Omega)\leq C \area(u_1(M\setminus\Omega)),
$$
and combining this with \eqref{almost.est}, we arrive at the desired estimate \eqref{qstab.est.1}.
\end{proof}

\begin{remark}\label{conf.ubds} Using the techniques of \cite[Section 2.3]{KNPS} in place of \cite[Section 6]{KNPS}, it is straightforward to prove a simpler, conformally-constrained version of the preceding lemma for those conformal classes induced by (non-branched) minimal immersions $M\to \mathbb{S}^n$ by first eigenfunctions on any closed surface $M$. Combining this with the estimates of Subsection \ref{gap.ctrl} below, one can easily prove a conformal analog of Proposition \ref{ubd.prop} for such conformal classes. Namely, if $(M,[g_1])$ is a conformal class arising from a minimal immersion $M\to \mathbb{S}^n$ by first eigenfunctions, then for any $g\in [g_1]$ and any domain $\Omega_k\subset M$ with $k$ boundary components, one has
\begin{equation}\label{conf.ubds}
\bar{\sigma}_1(\Omega_k,g)\leq \Lambda_1(M,[g_1])-c(M,[g_1])\frac{\log k}{k}.
\end{equation}
\end{remark}

\subsection{Structure of nearly-$\bar{\sigma}_1$-maximizing metrics with many boundary components}\label{gap.ctrl}

We collect now some of the key estimates which, together with Proposition \ref{quant.stab.prop} and \eqref{s2.quant.stab}, yield the proof of Proposition \ref{ubd.prop}.

Let $g_1\in \Met_{\can}(M)$ be a unit-area metric of constant curvature $K=2\pi \chi(M)$ on the closed surface $M$, with injectivity radius 
$$
\inj(M,g_1)\geq \iota_0>0.
$$
In particular, there exists $C=C(\iota_0)>0$ such that for any $\varphi\in W^{1,2}(M,g_1)$ one has
\begin{equation}
\label{ineq:CSobolev}
\|\varphi\|_{W^{1,2}(M,g_1)}\leqslant C\left(\|d\varphi\|_{L^2(M,g_1)} + \|\varphi\|_{L^1(M,g_1)}\right).
\end{equation}

Given a collection of disjoint geodesic balls $B_{r_1}(p_1),\ldots, B_{r_k}(p_k)$ in $(M,g_1)$, set
$$
\mathcal{B}:=B_{r_1}(p_1)\cup\cdots\cup B_{r_k}(p_k)
$$
and
$$
\Omega:=M\setminus \mathcal{B}.
$$
For some small $\delta>0$, suppose that
\begin{equation}
\label{area.bd.hyp}
\area_{g_1}(\mathcal{B})\leq \delta.
\end{equation}
Let $\mu$ be an admissible measure supported on $\partial\Omega$ and $0\leqslant\rho\in C^{\infty}(\Omega)$ a non-negative function such that
\begin{equation}\label{h-1.hyp}
|\langle \varphi, \sigma d\mu-\rho dv_{g_1}\rangle|^2\leq \delta \|\varphi\|_{W^{1,2}(\Omega,g_1)}^2
\end{equation}
for every $\varphi\in W^{1,2}(\Omega)$ and some constant $\sigma>0$. For the application we have in mind, one should think of $\sigma$ as the the first eigenvalue of $\mu$ on $\Omega$, i.e. $\sigma=\lambda_1(\Omega,[g_1],\mu)$.

By Lemma \ref{std.disk.lem}, we know that there exists $r_0(\chi(M))$ depending only on the curvature of $(M,g_1)$ such that a geodesic ball $B_t(x)$ of radius $t<t_0(\chi(M),\iota_0)=\min\{\iota_0,r_0\}$ in $(M,g_1)$ has area
$$
\frac{3\pi}{4}t^2\leq \area_{g_1}(B_t(x))\leq \frac{5\pi}{4}t^2.
$$
In particular, provided 
$$
\delta<\frac{3\pi}{4}t_0^2
$$
in~\eqref{area.bd.hyp}, it follows that
\begin{equation}
\label{rad.bd.1}
\frac{3\pi}{4}\sum_{j=1}^kr_j^2\leq \area_{g_1}(\mathcal{B})\leq \delta.
\end{equation}

Denote by $S$ the collection $S:=\{p_1,\ldots, p_k\}$ of all centers of the disks $B_{r_j}(p_j)$, and consider the subset
\begin{equation}
\label{sprime.def}
S':=\left\{p_i\in S\mid r_i\geq \frac{\sqrt{\delta}}{k^{1/4}}\right\}.
\end{equation}
It follows from \eqref{rad.bd.1} that
\begin{equation}\label{sprime.count}
|S'|\leq \frac{4}{3\pi}\sqrt{k},
\end{equation}
and it is of course possible that $S'=\varnothing$. In general, we have the following.

\begin{lemma}
\label{bigrad.smallmass}  
Let $(M,g_1)$, $\Omega$, $\mu$, and $\rho$ be as above satisfying \eqref{area.bd.hyp}-\eqref{h-1.hyp} with
$$
\delta<\frac{1}{\sqrt{k}}\leq t_0(\chi(M),\iota_0).
$$
Then 
$$
\sigma \mu\left(\bigcup_{p_j\in S'}B_{r_j}(p_j)\right)\leq C\left(1+\|\rho\|_{L^{\infty}}\right)\delta^{1/2}k^{1/4}
$$
for some constant $C=C(\chi(M),\iota_0)$.
\end{lemma}

\begin{proof} Let $\chi\in C^{\infty}(\mathbb{R})$ be a smooth, decreasing function such that
$$
\chi(t)\equiv 1\text{ for }t\leq \sqrt{\frac{4\delta}{3\pi}},
$$
$$
\chi(t)\equiv 0\text{ for }t\geq 2\sqrt{\frac{4\delta}{3\pi}},
$$
and
$$
|\chi'|\leq \frac{10}{\sqrt{\delta}}.
$$
Denoting by $d_{S'}\in \Lip(M)$ the distance function
$$
d_{S'}(x)=\min\left\{\dist_{g_1}(x,p_i)\mid p_i\in S'\right\},
$$
let
$$\varphi:=\chi\circ d_{S'}.$$
Since each $r_j\leq \sqrt{\frac{4\delta}{3\pi}}$ by \eqref{rad.bd.1}, it follows from the definition of $\varphi$ that
$$\varphi\equiv 1\text{ on }\bigcup_{p_j\in S'}B_{r_j}(p_j),$$
and consequently
\begin{equation}\label{sprime.mass}
\mu\left(\bigcup_{p_j\in S'}B_{r_j}(p_j)\right) = \langle \varphi,\mu\rangle.
\end{equation}
On the other hand, since $0\leq \varphi\leq 1$, $|d\varphi|\leq \frac{10}{\sqrt{\delta}}$, and $\varphi$ is supported on $\bigcup_{p_j\in S'}B_{\sqrt{16\delta/3\pi}}(p_j)$ by construction, we see that
$$
\|\varphi\|_{L^1(M,g_1)}\leq \sum_{p_j\in S'}\area_{g_1}\left(B_{\sqrt{16\delta/3\pi}}(p_j)\right)\leq C|S'|\delta
$$
and
$$
\|d\varphi\|_{L^2(M,g_1)}^2\leq \frac{100}{\delta}\sum_{p_j\in S'}\area_{g_1}\left(B_{\sqrt{16\delta/3\pi}}(p_j)\right)\leq C'|S'|.
$$
In particular, combining this with~\eqref{ineq:CSobolev} and~\eqref{sprime.count}, it follows that
$$
\|\varphi\|_{W^{1,2}(\Omega,g_1)}^2\leq C|S'|\leq C\sqrt{k}
$$
and
$$
\int_{\Omega}\rho \varphi dv_{g_1}\leq C\|\rho\|_{L^{\infty}}|S'|\delta\leq C'\|\rho\|_{L^{\infty}}\delta \sqrt{k};
$$
putting this together with \eqref{sprime.mass} and \eqref{h-1.hyp}, we find that
\begin{eqnarray*}
\sigma \mu\left(\bigcup_{p_j\in S'}B_{r_j}(p_j)\right) &= & \langle \sigma \mu,\varphi\rangle\\
&\leq &\left|\langle \sigma \mu-\rho dv_{g_1},\varphi\rangle\right|+\int_\Omega \rho \varphi dv_{g_1}\\
&\leq & \sqrt{\delta}\|\varphi\|_{W^{1,2}(\Omega,g_1)}+C\|\rho\|_{L^{\infty}}\delta\sqrt{k}\\
&\leq & C\left(\delta^{1/2}k^{1/4}+\|\rho\|_{L^{\infty}}\delta\sqrt{k}\right).
\end{eqnarray*}
Finally, recalling that $\delta \sqrt{k}<1$ by assumption, we have that $\delta\sqrt{k}\leq \delta^{1/2}k^{1/4}$, and the desired estimate follows.

\end{proof} 


\begin{lemma} Let $(M,g_1)$, $\Omega$, $\mu$, and $\rho$ be as above satisfying \eqref{area.bd.hyp}-\eqref{h-1.hyp} with
$$
\delta<\frac{1}{\sqrt{k}}\leq t_0(\chi(M),\iota_0).
$$
Then we have 
\begin{equation}
\label{big.sec4.est}
\sigma \mu(M)\leq C\frac{(1+\|\rho\|_{L^{\infty}})}{\log k}\left(k\delta+\sqrt{\delta k\log k}\right)
\end{equation}
for some $C=C(\chi(M),\iota_0)<\infty$. 

\end{lemma}

\begin{proof} Denote by $d_S\in \Lip(M)$ the distance to the full set $\{p_1,\ldots, p_k\}$ of centers of the geodesic disks $B_{r_1}(p_1),\ldots, B_{r_k}(p_k)$, and define a test function $\varphi\in \Lip(M)$ by
$$
\varphi(x):=\max\left\{\log\left(\sqrt{\delta}/d_S(x)\right),0\right\}\quad\text{ if }\quad d_S(x)\geq \frac{\sqrt{\delta}}{k}
$$
and
$$
\varphi(x):=\log(k)\quad\text{ if }\quad d_S(x)<\frac{\sqrt{\delta}}{k}.
$$
For $p_j\in S\setminus S'$, note that 
$$
\varphi\equiv \log\left(\sqrt{\delta}/r_j\right)\geq \log\left(k^{1/4}\right).
$$
Appealing to Lemma \ref{bigrad.smallmass}--and the nonnegativity of $\varphi$--it then follows that
\begin{eqnarray*}
\langle \sigma \mu,\varphi\rangle &\geq &\log\left(k^{1/4}\right)\sigma \mu \left(\bigcup_{p_j\in S\setminus S'}B_{r_j}(p_j)\right)\\
&=&\log\left(k^{1/4}\right)\left(\sigma \mu(M)-\sigma \mu\left(\bigcup_{p_j\in S'}B_{r_j}(p_j)\right)\right)\\
&\geq &\log \left(k^{1/4}\right)\sigma \mu(M)-C\log\left(k^{1/4}\right)\left(1+\|\rho\|_{L^{\infty}}\right)\delta^{1/2}k^{1/4};
\end{eqnarray*}
i.e.,
\begin{equation}\label{muphi.lbd}
\langle \sigma \mu,\varphi\rangle \geq \frac{1}{4}\log(k)\left(\sigma \mu(M)-C(1+\|\rho\|_{L^{\infty}})\delta^{1/2}k^{1/4}\right).
\end{equation}

Next, writing
$$
L_S(t):=\mathcal{H}^1(\{d_S=t\})\leqslant \mathcal{H}^1\left(\bigcup_{j=1}^k\partial B_t(p_j)\right);
$$
note that, since $t<t_0(\chi(M),\iota_0)<\min\{r_0(\chi(M)),\inj(M)\},$ Lemma \ref{std.disk.lem} gives
$$
L_S(t)\leq \sum_{j=1}^k\mathcal{H}^1(\partial B_t(p_j))\leq k \frac{5\pi}{2}t.
$$
By definition of $\varphi$, we can then employ the coarea formula for $d_S$ to see that
\begin{eqnarray*}
\int_{\Omega}\varphi dv_{g_1}&\leq &\int_M\varphi dv_{g_1}\\
&=&\log k\int_0^{\sqrt{\delta}/k}L_S(t)dt+\int_{\sqrt{\delta}/k}^{\sqrt{\delta}}\log(\sqrt{\delta}/t)L_S(t)dt\\
&\leq &\frac{5\pi}{2}k\left(\log k\int_0^{\sqrt{\delta}/k}tdt+\int_0^{\sqrt{\delta}}\log\left(\sqrt{\delta}/t\right)tdt\right)\\
&=&\frac{5\pi}{4}k\delta\left(\frac{\log k}{k^2}+\frac{1}{2}\right)\leq \frac{5\pi}{4}k \delta.
\end{eqnarray*}
Similarly, since $|d\varphi|_{g_1}^2=d_S^{-2} \chi_{\left\{\frac{\sqrt{\delta}}{k}\leq d_S\leq \sqrt{\delta}\right\}}$, we can compute
$$
\int_{\Omega}|d\varphi|^2_{g_1}\,dv_{g_1}\leq \int_{\sqrt{\delta}/k}^{\sqrt{\delta}}\frac{1}{t^2}L_S(t)dt\leq \frac{5\pi}{2}k\int_{\sqrt{\delta}/k}^{\sqrt{\delta}}\frac{1}{t}dt
= \frac{5\pi}{2}k\log k.
$$
In particular, since $\varphi\geqslant 0$ and $\delta^2k<1$, it follows from~\eqref{ineq:CSobolev} that
$$
\|\varphi\|_{W^{1,2}(\Omega,g_1)}^2\leq C_0 \left(k^2\delta^2+k\log k\right)\leqslant C_0'k\log k
$$
and
$$
\int_{\Omega}\rho \varphi\,dv_{g_1} \leq C_0\|\rho\|_{L^{\infty}}k \delta.
$$
From \eqref{h-1.hyp}, we deduce that
\begin{eqnarray*}
\langle \varphi, \sigma \mu\rangle&\leq & \int_{\Omega}\rho\varphi dv_{g_1}+\sqrt{\delta}\|\varphi\|_{W^{1,2}(\Omega,g_1)}\\
&\leq &C_0\|\rho\|_{L^{\infty}}k\delta +C_0\sqrt{\delta}\left(k\log k\right)^{1/2}\\
&\leq & C_1\left (1+\|\rho\|_{L^{\infty}}\right)\left(k\delta+\sqrt{\delta k \log k}\right).
\end{eqnarray*}
Finally, combining this with \eqref{muphi.lbd}, we obtain
$$
\frac{1}{4}\log(k)\left(\sigma \mu(M)-C(1+\|\rho\|_{L^{\infty}})\delta^{1/2}k^{1/4}\right)\leq C\left(1+\|\rho\|_{L^{\infty}}\right)\left(k\delta+\sqrt{\delta k \log k}\right),
$$
and noting that $\delta^{1/2}k^{1/4}\leq C\sqrt{\delta k \log k}$ for $k>1$, this in turn gives
$$
\sigma \mu(M)\leq C \frac{(1+\|\rho\|_{L^{\infty}})}{\log k}\left(k\delta+\sqrt{\delta k \log k}\right),
$$
which is the desired estimate \eqref{big.sec4.est}.
\end{proof}

\subsection{Proof of Proposition \ref{ubd.prop}}

With the results of Sections \ref{ref.stab} and \ref{gap.ctrl} in place, the proof of Proposition \ref{ubd.prop} is now relatively straightforward.

\begin{proof}[Proof of Proposition \ref{ubd.prop}]

Let $N_k$ be the compact surface with boundary given by removing $k$ disks from $M=\Sp,\mathbb{RP}^2,\mathbb{T}^2$, or $\mathbb{K}$, and suppose that $g$ is a metric on $N_k$ for which 
\begin{equation}
\label{del.gap}
\bar{\sigma}_1(N_k,g)\geq \Lambda_1(M)-\eta.
\end{equation}
We wish to show that $\eta\geq c \frac{\log k}{k}$ for some constant $c=c(M)>0$.

By Theorem \ref{thm:uni}, we may identify $(N_k,g)$ isometrically with a domain $(\Omega_k,\tilde{g})$ satisfying the hypotheses of Proposition \ref{s2.quant.stab} or \ref{quant.stab.prop}; i.e., we may assume that there exists a unit-area, constant curvature metric $g_1$ on $M$ such that $\tilde{g}\in [g_1]$ and $\Omega_k=M\setminus \left(B_1\cup\cdots\cup B_k\right)$, where $\{B_i\}_{i=1}^k$ is a collection of disjoint geodesic disks in $(M,g_1)$. Moreover, if $M=\Sp$, we may assume without loss of generality (by the standard Hersch trick) that the length measure $\mu_k=ds_{\tilde g}$ of $\partial\Omega_k$ satisfies the balancing condition $\int_{\Sp}I d\mu_k=0$.

If $\eta<\delta_1(M)$, it then follows from Proposition \ref{s2.quant.stab} or \ref{quant.stab.prop} that there exists a $\bar{\lambda}_1$-maximal metric $g_{\max}$ on $M$ conformal to some $g_0\in \Met_{\can}(M)$ such that 
\begin{equation}\label{mets.close}
\|g_0-g_1\|_{C^1(g_0)}^2\leq C(M) \eta
\end{equation}
and the length measure $\mu_k=ds_{\tilde g}$ of $\partial\Omega_k$ satisfies
\begin{equation}\label{main.gap.cons}
\|\sigma_1\mu_k-\lambda_1(g_{\max})dv_{g_{\max}}\|_{\left(W^{1,2}(\Omega,g_0)\right)^*}^2+\area_{g_0}(M\setminus \Omega_k)\leq C(M)\eta.
\end{equation}
In particular, for $\eta<\delta_2(M)$ sufficiently small, it follows from \eqref{mets.close} that
\begin{equation}
\inj(M,g_1)\geq \frac{1}{2}\inj(M,g_0)\geq c_0(M),
\end{equation}
and 
$$
\frac{1}{2}g_1\leq g_0\leq 2g_1,
$$
so we can replace $\area_{g_0}$ and $\left(W^{1,2}(\Omega,g_0)\right)^*$ in~\eqref{main.gap.cons} with $\area_{g_1}$ and $\left(W^{1,2}(\Omega,g_1)\right)^*$, adjusting the constant $C(M)$ on the right-hand side accordingly. Moreover, writing
$$
\lambda_1(g_{\max})dv_{g_{\max}}=\rho dv_{g_1},
$$
we see that 
\begin{equation}
\|\rho\|_{L^{\infty}}\leq C(M).
\end{equation}

Assume now that 
$$\eta<\frac{\log k}{k}.$$
By the preceding observations, we then see that $(M,g_1)$, $\Omega_k$, $\mu_k$, and $\rho$ satisfy the hypotheses of Lemma \ref{big.sec4.est} for $k\geq k_0(M)$ sufficiently large, with $\sigma=\sigma_1(\Omega,\tilde{g})$ and
$$\delta=C(M)\eta.$$
In particular, since in this case we have $\|\rho\|_{L^{\infty}}\leq C(M)$ and $\inj(M,g_1)\geq c_0(M)$, it follows from Lemma \ref{big.sec4.est} that
$$
\Lambda_1(M)-\eta\leq \bar{\sigma}_1(M,\tilde{g})\leq \frac{C(M)}{\log k}\left(k\eta+\sqrt{\eta k\log k}\right),
$$
so that
\begin{equation}\label{key.gap.est}
\frac{1}{2}\Lambda_1(M)\leq C(M)\left(\frac{k\eta}{\log k}+\sqrt{\frac{k\eta}{\log k}}\right)\leq C'(M)\left(\frac{k\eta}{\log k}\right)^{1/2},
\end{equation}
using the assumption that $\eta<\frac{\log k}{k}$. Squaring both sides and rearranging, we obtain
\begin{equation}
\eta \geq \left(\frac{\Lambda_1(M)}{2C'(M)}\right)^2\frac{\log k}{k},
\end{equation}
giving the desired bound.

\end{proof}


\begin{thebibliography}{ABCDEF}
\bibitem[ACS]{ACS} L.~Ambrozio, A.~Carlotto, B.~Sharp, Compactness analysis for free boundary minimal hypersurfaces. {\em Calculus of Variations and PDEs} {\bf 57}:1 (2018), 22.
\bibitem[BK]{BK} S.~R.~Bell, S.~G.~Krantz, Smoothness to the boundary of conformal maps. {\em The Rocky Mountain Journal of Mathematics} {\bf 17}:1 (1987), 23--40.
\bibitem[Be]{Besson} G.~Besson, Sur la multiplicit\'e de la premi\`ere valeur propre des surfaces riemanniennes,
{\em Ann. Inst. Fourier} {\bf 30} (1980), 109--128
\bibitem[Ch]{Cheng} S.~Y.~Cheng, Eigenfunctions and nodal sets, {\em Comment. Math. Helv.} {\bf 51} (1976), 43--55.
\bibitem[CKM]{CKM} D.~Cianci, M.~Karpukhin, V.~Medvedev, On branched minimal immersions of surfaces by first eigenfunctions. {\em Annals of Global Analysis and Geometry} {\bf 56}:4 (2019), 667--690.
\bibitem[CM]{CM} T.~H.~Colding, W.~P.~Minicozzi, A course in minimal surfaces. {\em Graduate studies in mathematics vol. 121.} American Mathematical Society, 2011.
\bibitem[EGJ]{EGJ} A.~El Soufi, H.~Giacomini, M.~Jazar, A unique extremal metric for the least eigenvalue of the Laplacian on the Klein bottle. {\em Duke Mathematical Journal} {\bf 135}:1 (2006), 181--202.
\bibitem[ESI]{ESI} A.~El Soufi, S.~Ilias, Laplacian eigenvalues functionals and metric deformations on compact manifolds. {\em J. Geom. Phys.} {\bf 58}:1 (2008), 89--104. 
\bibitem[EG]{EG} L.~C.~Evans, R.~F.~Gariepy. Measure theory and fine properties of functions, {\em Studies in advanced mathematics} (1992).
\bibitem[FS1]{FS} A.~Fraser, R.~Schoen, Sharp eigenvalue bounds and minimal surfaces in the ball.  {\em Inventiones mathematicae} {\bf 203}:3 (2016), 823--890.
\bibitem[FS2]{FS:conf_volume} A.~Fraser, R.~Schoen, The first Steklov eigenvalue, conformal geometry, and minimal surfaces. {\em Advances in Mathematics} {\bf 226}:5 (2011), 4011 -- 4030.
\bibitem[FS3]{FS:extremal} A.~Fraser, R.~Schoen, Minimal surfaces and eigenvalue problems. {\em Geometric analysis, mathematical relativity, and nonlinear partial differential equations}. {\bf 599} (2012), 105--121.
\bibitem[GL]{GL} A.~Girouard, J.~Lagac\'e, Large Steklov eigenvalues via homogenisation on manifolds. To appear in {\em Inventiones mathematicae}. Published online at {\tt https://doi.org/10.1007/s00222-021-01058-w}. 
\bibitem[GKL]{GKL} A.~Girouard, M.~Karpukhin, J.~Lagacé. Continuity of eigenvalues and shape optimisation for Laplace and Steklov problems. To appear in {\em Geometric and Functional Analysis}. Published online at {\tt https://doi.org/10.1007/s00039-021-00573-5}.
\bibitem[GP1]{GP:HPS} A.~Girouard, I.~Polterovich. On the Hersch-Payne-Schiffer inequalities for Steklov eigenvalues. {\em Functional Analysis and its Applications} {\bf 44}:2 (2010), 106--117.
\bibitem[GP2]{GP:survey} A.~Girouard, I. Polterovich, Spectral geometry of the Steklov problem (survey article) {\em  Journal of Spectral Theory} {\bf 7}:2 (2017), 321--360.
\bibitem[Ha]{Haas} A. Haas, Linearization and mappings onto pseudocircle domains, 
{\em Trans. Amer. Math. Soc.} {\bf 282}
(1984), 415--429.
\bibitem[Her]{Hersch} J. Hersch, Quatre propri\'et\'es isop\'erim\'etriques de membranes sph\'eriques homog\`enes.  {\em C. R. Acad. Sci. Paris S\'er A-B} 270 (1970), A1645--A1648.
\bibitem[JLNNP]{JLNNP} D.~Jakobson, M.~Levitin, N.~Nadirashvili, N.~Nigam,  I.~Polterovich,  How large can the first eigenvalue be on a surface of genus two? {\em Int. Math. Research Notices}, {\bf 63} (2005), 3967--3985.
\bibitem[JNP]{JNP} D.~Jakobson, N.~Nadirashvili, and I.~Polterovich, Extremal metric for the first eigenvalue on a Klein bottle. {\em Canadian J. of Mathematics} {\bf 58}:2 (2006), 381--400.
\bibitem[J]{Jammes} P.~Jammes,  Multiplicit\'e du spectre de Steklov sur les surfaces et nombre chromatique. {\em Pacific Journal of Mathematics} {\bf 282}:1 (2016), 145--171.
\bibitem[KOO]{KOO} C.-Y. Kao, \'E.~Oudet, B.~Osting, Computation of free boundary minimal surfaces via extremal Steklov eigenvalue problems. {\em ESAIM: Control, Optimisation and Calculus of Variations} {\bf 27}:34 (2021).
\bibitem[KKP]{KKP} M. Karpukhin, G. Kokarev, I. Polterovich, Multiplicity bounds for Steklov eigenvalues on Riemannian surfaces. {\em Annales de l'Institute Fourier} {\bf 64}:6 (2014), 2481--2502.
\bibitem[K1]{KarNonOrientable} M.~Karpukhin, Upper bounds for the first eigenvalue of the Laplacian on non-orientable surfaces. {\em Int. Math. Research Notices}, {\bf 20}  (2016), 6200--6209.
\bibitem[K2]{KarRP2} M.~Karpukhin, Index of minimal spheres and isoperimetric eigenvalue inequalities. {\em Inventiones Mathematicae} {\bf 223} (2021), 335--377.
\bibitem[KM]{KM} M. Karpukhin, A. Métras, Laplace and Steklov extremal metrics via $n$-harmonic maps. Preprint {\tt arXiv:2103.15204}.
\bibitem[KNPP]{KNPP} M.~Karpukhin, N.~Nadirashvili, A.~Penskoi, I.~Polterovich, An isoperimetric inequality for Laplace eigenvalues on the sphere. {\em J. Differential Geom.} {\bf 118}:2 (2021), 313--333. 
\bibitem[KNPP2]{KNPP2} M.~Karpukhin, N.~Nadirashvili, A.~Penskoi, I.~Polterovich, Conformally maximal metrics for Laplace eigenvalues on surfaces. To appear in {\em Surveys in Differential geometry}. Preprint {\tt arXiv:2003.02871}.
\bibitem[KNPS]{KNPS} M.~Karpukhin, M.~Nahon, I.~Polterovich, D.~Stern, Stability of isoperimetric inequalities for Laplace eigenvalues on surfaces. Preprint {\tt arXiv:2106.15043}.
\bibitem[KS]{KS} M. Karpukhin, D. L. Stern, Min-max harmonic maps and a new characterization of conformal eigenvalues. Preprint {\tt arXiv:2004.04086}.
\bibitem[Kok]{Kokarev} G.~Kokarev, Variational aspects of Laplace eigenvalues on Riemannian surfaces. {Advances in Mathematics,} {\bf 258} (2014), 191--239.
\bibitem[L]{Lawson} H.~B.~Lawson Jr, Complete minimal surfaces in $\mathbb{S}^3$. {\em Annals of Mathematics} (1970), 335--374.
\bibitem[Li]{Li_survey} M. Li, Free boundary minimal surfaces in the unit ball: recent advances and open questions {\em 
Proceedings of the International Consortium of Chinese Mathematicians, 2017 (First Annual Meeting)}, International Press of Boston, Inc. (2020), 401--436.
\bibitem[LY]{LY} P.~Li, S.-T.~Yau, A new conformal invariant and its applications to the Willmore conjecture and the first eigenvalue of compact surfaces. {\em Inventiones mathematicae}, {\bf 69}:2 (1982), 269--291.
\bibitem[MN]{MN} F.~C.~Marques, A.~Neves, Min-max theory and the Willmore conjecture. {\em Annals of mathematics} (2014), 683--782.
\bibitem[Ma]{Maskit} B.~Maskit, Canonical domains on Riemann surfaces. {\em Proc. Amer. Math. Soc.} {\bf 106}:3 (1989), 713--721.
\bibitem[MP1]{MP} H.~Matthiesen, R.~Petrides, Free boundary minimal surfaces of any topological type in Euclidean balls via shape optimization. Preprint {\tt arXiv:2004.06051}.
\bibitem[MP2]{MP2} H.~Matthiesen, R.~Petrides. A remark on the rigidity of the first conformal Steklov eigenvalue. Preprint {\tt arXiv:2006.04364}.
\bibitem[MS]{MS} H.~Matthiesen, A.~Siffert, Handle attachment and the normalized first eigenvalue. Preprint {\tt arXiv:1909.03105}.
\bibitem[Med]{Medvedev} V. Medvedev, Degenerating sequences of conformal classes and the conformal Steklov spectrum. To appear in {\em Canadian Journal of Mathematics}. Published online at {\tt http://dx.doi.org/10.4153/S0008414X21000171}.
\bibitem[N1]{Nad:mult} N.~Nadirashvili, Multiple eigenvalues of the Laplace operator. {\em Mathematics of the USSR-Sbornik}  {\bf 61}:1 (1988), 225.
\bibitem[N2]{Nad} N. Nadirashvili, Berger's isoperimetric problem and minimal immersions of surfaces. {\em Geom. Funct. Anal}, {\bf 6}:5 (1996), 877--897.
\bibitem[NS]{NaySh} S.~Nayatani, T.~Shoda, Metrics on a closed surface of genus two which maximize the first eigenvalue of the Laplacian. {\em Comptes Rendus Mathematique}, {\bf 357}:1 (2019), 84--98.
\bibitem[OPS]{OPS} B.~Osgood, R.~Phillips, P.~Sarnak, Extremals of determinants of Laplacians. {\em J. Funct. Anal.}, {\bf 80} (1998), 148--211.
\bibitem[P1]{PetL1} R.~Petrides,  Existence and regularity of maximal metrics for the first Laplace eigenvalue on surfaces. {\em Geometric and Functional Analysis,} {\bf 24}:4 (2014), 1336--1376.
\bibitem[P2]{PetLk} R.~Petrides, On the existence of metrics which maximize Laplace eigenvalues on surfaces.
{\em Int. Math. Research Notices}, {\bf 14} (2018), 4261--4355.
\bibitem[P3]{PetS} R.~Petrides, Maximizing Steklov eigenvalues on surfaces. {\em  Journal of Differential Geometry} {\bf 113}:1 (2019), 95--188.
\bibitem[R]{Ros} A.~Ros, On the first eigenvalue of the laplacian on compact surfaces of genus three. Preprint {\tt arXiv:2010.14857}.
\bibitem[YY]{YY} P.~C.~Yang, S.-T.~Yau, Eigenvalues of the Laplacian of compact Riemann surfaces and minimal submanifolds, {\em Annali della Scuola Normale Superiore di Pisa-Classe di Scienze} {\bf 7}:1 (1980), 55--63.
\end{thebibliography}
\end{document}